\documentclass[11pt,a4paper]{amsart}
\usepackage{graphicx}

\usepackage{amsmath}
\usepackage{amssymb}
\usepackage{amsthm}

\usepackage{xifthen}

\usepackage{aliascnt}
\usepackage[bookmarksopen,bookmarksdepth=2,backref=none,hidelinks]{hyperref}

\theoremstyle{definition}
\newtheorem{definition}{Definition}[section]
\newaliascnt{lemma}{definition}
\newaliascnt{theorem}{definition}
\newaliascnt{corollary}{definition}
\newaliascnt{proposition}{definition}
\newaliascnt{assumption}{definition}
\newaliascnt{remark}{definition}
\newaliascnt{example}{definition}
\newaliascnt{conjecture}{definition}
\usepackage{esint}

\newtheoremstyle{remark2}{}{}{}{}{\bfseries}{.}{ }{}
\theoremstyle{remark2}

\newtheorem{remarkx}[remark]{Remark}
\newenvironment{remark}
  {\pushQED{\qed}\remarkx}
  {\popQED\endremarkx}

\theoremstyle{plain}
\newtheorem{lemma}[lemma]{Lemma}
\newtheorem{theorem}[theorem]{Theorem}
\newtheorem{corollary}[corollary]{Corollary}
\newtheorem{proposition}[proposition]{Proposition}

\newtheorem{assumption}[assumption]{Assumption}

\newcommand{\nablasym}{{\varepsilon}}

\aliascntresetthe{lemma}
\aliascntresetthe{theorem}
\aliascntresetthe{corollary}
\aliascntresetthe{proposition}
\aliascntresetthe{remark}
\aliascntresetthe{conjecture}
\aliascntresetthe{assumption}

\newcommand{\R}{\ensuremath{\mathbb{R}}}

\newcommand{\N}{\ensuremath{\mathbb{N}}}

\newcommand{\abs}[1]{\left|{#1}\right|}
\newcommand{\norm}[2][]{\left\|#2\right\|_{{#1}}} %
\newcommand{\inner}[3][]{\left\langle #2,#3 \right \rangle_{#1}} %
\newcommand{\pd}[2][]{\frac{\partial #1}{\partial #2}} %

\DeclareMathOperator*{\argmin}{arg\,min}
\DeclareMathOperator{\diver}{div}
\DeclareMathOperator{\cof}{cof}
\DeclareMathOperator{\tr}{tr}
\DeclareMathOperator{\dist}{dist}

\DeclareMathOperator{\supp}{supp}
\DeclareMathOperator{\Lip}{Lip}
\newcommand{\bog}{\mathcal{B}}
\newcommand{\eps}{\varepsilon}

\numberwithin{equation}{section}

\newcommand{\weakto}{\rightharpoonup}

\allowdisplaybreaks

\title[A variational approach to FSI]{A variational approach to hyperbolic evolutions and fluid-structure interactions}
\author{B.~Bene\v{s}ov\'{a}, M.~Kampschulte, S.~Schwarzacher}
\address{Department of Mathematics and Physics, Charles University Prague,  \nolinkurl{benesova@karlin.mff.cuni.cz}, \nolinkurl{kampschulte@karlin.mff.cuni.cz}, \nolinkurl{schwarz@karlin.mff.cuni.cz} }
\date{\today}

\begin{document}
 \begin{abstract}
We introduce a two time-scale scheme which allows to extend the method of minimizing movements to hyperbolic problems. This method is used to show the existence of weak solutions to a fluid-structure interaction problem between a nonlinear, visco-elastic, $n$-dimensional bulk solid governed by a hyperbolic evolution and an incompressible fluid governed by the ($n$-dimensional) Navier-Stokes equations for $n\geq 2$. %
 \end{abstract}

 \maketitle

 \tableofcontents

\section{Introduction}
\label{sec-intro}

In the seminal work~\cite{degiorgiNewProblemsMinimizing1993} De Giorgi introduced a {\em variational approach} to construct solutions to gradient flows. %
The method, known as {\em minimizing movements}, is an incremental time-stepping scheme in which discrete in time approximations are constructed using step-wise minimization. %

In this work, we present two extensions to De Giorgi's approach. %
The first (developed in Section~\ref{sec:qs}) is a minimizing movements approximation for systems of \emph{partial differential equations} %
that are coupled over a \emph{common, co-evolving interface influencing their domain of definition}. %
The second (developed in Section~\ref{sec:so}) is a two time-scale method that allows to approximate {\em hyperbolic evolutions} via De Giorgi's scheme. Both provide new view points on the {\em existence theory} for hyperbolic evolutions as well as geometrically coupled PDEs respectively.
 
Furthermore, the methods are compatible. In this paper they are applied to prove the existence of a weak solutions for a wide class of {\em fluid-structure interaction problems}. In particular %
the methods are equipped with the necessary analysis to gain the existence of a weak solution describing the coupling between the {\em unsteady incompressible Navier-Stokes equations}~(\ref{strong2}), (\ref{strong3}), with a hyperbolic evolution describing the deformation of a {\em visco-elastic bulk solid} characterized by a {\em non-convex stored energy functional}~\eqref{strong1}\footnote{See \eqref{st-venant} for our prototypical energy for the solid deformation.}. %

The present work focuses on {\em PDEs related to continuum mechanics} and in particular fluid-structure interactions. However, the developed variational approximation theory turns out to be rather flexible: In this work it is demonstrated that the methods are suitable for coupled systems of PDEs with very different non-linear characteristics, geometric constraints and coupling conditions, non-monotone operators, hyperbolic evolutions, variable in time geometries and material time derivatives. Hence it well might have the potential for a wide range of other applications within the range of continuum mechanics and possibly beyond.  %

Fluid-structure interactions are concerned with the way the movement of a fluid influences the deformation of a solid body and vice versa. A scheme is developed that can handle fluid-structure interactions in a general setting; i.e.\ allowing for both the fluid and the structure to have full dimension and permitting large, unrestricted forces and deformations, which in particular preclude approaches relying on linearization or convexity.

Though the literature on fluid-structure interaction is rich and diverse, the mathematical theory of elastic, large deformation bulk solids interacting with viscous fluids is rather unexplored. 
However, as far as the theoretical analysis is concerned, a large amount of research concentrates on rigid bodies surrounded by a fluid \cite{desjardinsExistenceWeakSolutions1999,Fei03,
Takahashi03,GalSil07,GalSil09,Gal13,GGH13,TakTucWei15,Gal16bf,CheSar17} or plates or shells (lower dimensional solids, usually additionally restricted to deform in a prescribed direction) interacting with fluids~ \cite{grandmontExistenceGlobalStrong2016,muhaExistenceWeakSolution2013,muhaNonlinear3DFluidstructure2013,muhaFluidstructureInteractionIncompressible2015,muhaExistenceWeakSolution2016,lengelerWeakSolutionsIncompressible2014,chemetovWeakstrongUniquenessFluidrigid2019}. Analytical results concerning the existence theory for the interaction of a fluid with an elastic body of the same dimension are few and generally include some smallness condition that restricts the problem to the regime of small deformations; cf. e.g. \cite{DesGra01,grandmontExistenceThreeDimensionalSteady2002,galdiSteadyFlowNavier2009}.
Mathematical works that are allowing for bulk solids with large deformations can be found in the community related to computer simulations~(see for instance \cite{quarteroni2000computational,GalRan08,RichterFluidStructureInteractions17} and the references there).
Finally, we wish to mention that due to the high relevance of applications related to fluid-structure interactions (e.g.\ Airoplanes, blood vessels and heart valves, pipes, ect.) the physical and engineering literature is numerous. %

The mathematical description of fluid-structure interaction is given through a \emph{system of mutually coupled (hyperbolic) PDEs} that is required to satisfy additional \emph{geometric restrictions}. To be more precise, the {\em common interface between the fluid and the solid is a part of the solution} and hence has to be constructed simultaneously. Moreover, as the solid is allowed to perform large structural deformations, it is necessary to assure that the solutions we consider possess deformations that are known to be injective a-priori.
This is not only essential in order for the solutions to have physical meaning, but also for the coupling to the fluid variables to be well-posed.

By reflecting these geometric restrictions, the operator appearing in the equations describing the movement of the solid is necessarily  \emph{non-monotone} and possesses a \emph{non-convex} domain of definition.%

In order to handle these (highly nonlinear) limitations %
we develop a methodology that focuses on the underlying \emph{energetic} structure of the PDEs and a \emph{variational} point of view. %
The existence theory is built upon the following three ideas each introduced in a separate section.
\begin{enumerate}
    \item[Sec.~\ref{sec:qs}] {\bf Minimizing movements for fluid-structure interaction}: We introduce an adaptation of De Giorgis \emph{minimizing  movements} \cite{degiorgiNewProblemsMinimizing1993,ambrosioGradientFlows2005} scheme. This is used to show the existence of weak solutions to the fluid-structure interaction problem in the simplified {\em parabolic} case where inertial effects are omitted. Consequently the fluid motion is quasi-steady following Stokes equation and the solid evolves along a gradient flow. In more detail, we construct an approximate, time-discretized solution by %
    iteratively solving a {\em coupled minimization problem}. %
    Each minimization produces the subsequent deformation and thus a new interface to be used in the next time step to separate the fluid and the solid. %
    The new variational approach we introduce is essential, as it can cope both with the non-monotonicity of the involved operators and the non-convex nature of the underlying state space. Moreover, used properly, it will also automatically induce the correct coupling conditions between the fluid and solid velocities and forces at their common interface. %
    
   While the minimizing movements method and its variations are commonly used for quasi-static evolution problems in the field of viscoelastic solids (see e.g. \cite{kruzik2019mathematical}), it has, to the best of our knowledge, never been applied to fluid-structure interaction before.
    \item[Sec.~\ref{sec:so}]
     {\bf Minimizing movements for hyperbolic evolutions}: The original minimizing movements scheme is restricted to gradient flows. 
    We propose an extension to this scheme to prove existence for hyperbolic equations. The method is built on an approach involving\emph{two time-scales}: The larger \emph{acceleration time-scale} and the smaller {\em velocity time-scale}. Consequently, the hyperbolic second time-derivative is approximated by a double difference quotient with respect to the different scales. This provides the setting of a minimizing movements iteration for all fixed positive {\em acceleration scales}. In doing so, we ``combine the best of two worlds'', by first using the variational approach to handle all the non-linearities during the construction of approximate solutions and then the limits of the Euler-Lagrange equations for the hyperbolic evolution to gain respective a-priori estimates that are uniform with respect to both scales.  %

    \item[Sec.~\ref{sec:full}] {\bf Bulk elastic solids coupled to Navier-Stokes equations}:
    In this section the ideas from the previous two sections are combined to establish the main result of the paper, which is the existence of solutions to a problem involving a solid (governed by a hyperbolic PDE) interacting with the {\em unsteady incompressible Navier-Stokes equations}. %
  
    In order to adapt the method to the \emph{distinctly Eulerian} description of the fluid by its {\em velocity} and {\em pressure} in the time changing domain, we use an {\em Lagrangian approximation of the material time derivative} on the acceleration scale. This is done by using the method of characteristics and constructing a \emph{flow map} for short times. Since flow-velocity and flow map are inextricably linked and both are additionally connected to the changing fluid domain, the main difficulty here is that the three have to be constructed simultaneously. In particular we introduce a discretized flow-map in the discretized variable domain %
    for which we prove various natural regularity and approximability results. This allows for a coupling between the hyperbolic minimizing movements scheme for the solid deformation and the respective semi-Lagrangian approximation of the unsteady Navier Stokes equation.   
    
   Some of the ideas in this section are inspired by the works  \cite{Pir82,gigliVariationalApproachNavier2012} on the Navier-Stokes equations. In particular, in~\cite{gigliVariationalApproachNavier2012} the authors apply minimizing movements to the construction of solutions to the Navier-Stokes equations. Their approach however strongly depends on having a fixed fluid domain.

\end{enumerate}

In sum, we intend to introduce a new viewpoint on fluid-structure interaction and hyperbolic problems in general. In particular (for the first time) we prove existence of a bulk, large deformations interaction problem between a viscoelastic solid and the incompressible Navier Stokes equation until the point of self-contact of the fluid boundaries.

\subsection{Setup}
\label{sec:setup}

We consider the following set-up for the fluid-structure interaction problem: The fluid together with the elastic structure are both confined to a container $\Omega \subset \R^n$ that is fixed in time. The deformation of the solid is at any instant of time $t$ described via the deformation function $\eta(t):= \eta(t,.) : Q   \to \Omega$. Here $Q$ is a given reference configuration of the solid. 
We assume that both $Q,\Omega \subset \R^n$ are Lipschitz-domains.%
Here, $n\geq 2$ is the dimension of the problem with $n=2$ corresponding to the planar case and $n=3$ to the bulk case. The fluid velocity
is defined in the variable in time domain $\Omega(t):=\Omega\setminus \eta(t,Q)$.
It is determined by its velocity $v(t): \Omega(t) \to \R^n$ and its pressure $p:\Omega(t)\to \R$. 
Thus, the solid is described in \emph{Lagrangian} and the fluid in \emph{Eulerian} coordinates. Observe, that a similar configuration has already been studied for linear elasticity in~\cite{desjardinsWeakSolutionsFluidelastic2001}. We refer to Figure \ref{fig:scheme} for a better orientation.

\begin{figure}[!ht]
\label{fig:scheme} 
\begin{center}
    \includegraphics[width=0.7\paperwidth]{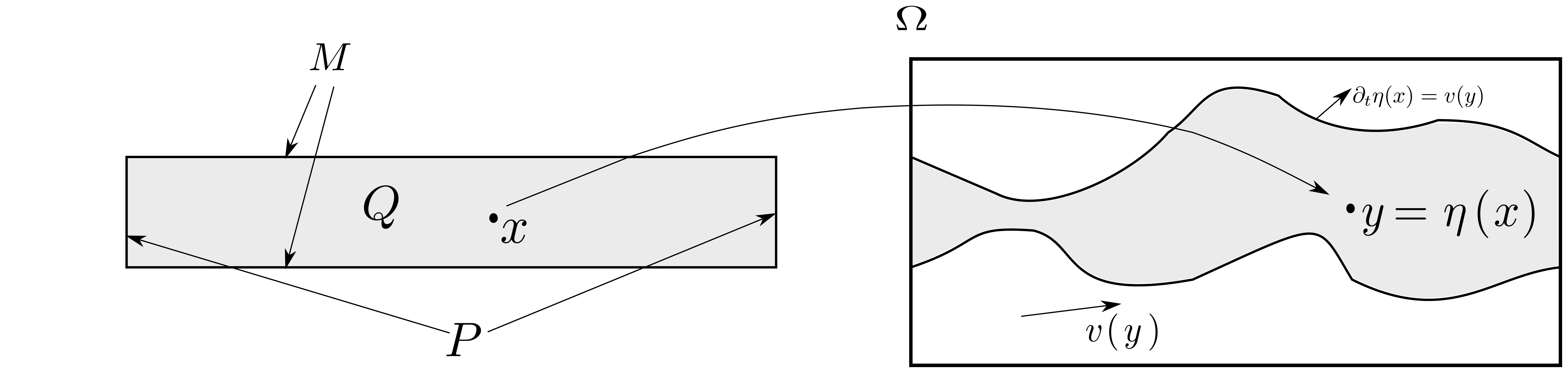}
    \end{center}
    \caption{A scheme of the geometry of the fluid-structure interaction. The reference configuration is at the left while at the right we depict the situation in a given time instant $t$ (the actual configuration).}
\end{figure}
For setting up the evolution equation, we will need the basic physical balances to be fulfilled. As we are not modeling any thermal effects, this reduces to the balance of momentum for both the fluid and the structure together with suitable conditions on their mutual boundary, as well as conservation of mass. In the interior these balance equations read in strong formulation as 
\begin{align}
\label{strong1}
\rho_s\partial_t^2\eta+\diver \sigma &=\rho_s f\circ \eta &\text{ in } Q, \\
\label{strong2}
\rho_f (\partial_t v +[\nabla v]v)&=\nu \Delta v  -\nabla p +\rho_f f     &\text{ on } \Omega(t), \\
\diver v&=0 &\text{ on } \Omega(t).\label{strong3}
\end{align}
Here, $\sigma$ is the {\em first Piola--Kirchhoff stress tensor} of the solid, $\nu$ is the {\em viscosity constant} of the fluid, %
$\rho_s$ and $\rho_f$ are the {\em densities} of the solid and fluid respectively and $f$ is the actual applied force in the current (Eulerian) configuration. Thus, the fluid is assumed to be Newtonian with the \emph{Navier-Stokes equation} modeling its behavior. For the solid, we will restrict ourselves to materials for which the first Piola--Kirchhoff stress tensor $\sigma$ can be derived from underlying \emph{energy and dissipation potentials}; i.e. 
\begin{align*}
{\diver} \sigma := DE(\eta) +  D_2 R(\eta,\partial_t\eta)
\end{align*} %
with $E$ being the energy functional describing the elastic properties while $R$ is the dissipation functional used to model the viscosity of the solid. Here $D$ denotes the Fr\'{e}chet derivative and $D_2$ the Fr\'{e}chet derivative with respect to the second argument. Such materials are often called \emph{generalized standard materials} \cite{halphenMateriauxStandardGeneralises1975,nguyenStabilityNonlinearSolidMechanics2000,kruzik2019mathematical}. For the analysis performed in this paper, quite general forms of $E$ and $R$ can be admitted (see \autoref{subsec:restrict} below). The prototypical examples for the potentials are the following:
\begin{align}
\label{kelvinVoigt}
R(\eta,\partial_t\eta) &:= \int_Q |(\nabla \partial_t \eta)^T \nabla \eta + (\nabla \eta)^T (\nabla \partial_t \eta)|^2 dx = \int_Q |\partial_t (\nabla \eta^T \nabla \eta)|^2 dx  \\ \label{st-venant}
E(\eta) &:= \begin{cases}
 \int_Q \frac{1}{8}|\nabla \eta^T \nabla \eta-I|_\mathcal{C} + \frac{1}{(\det \nabla \eta)^a} +\frac{1}{q} \abs{\nabla^2 \eta }^q dx & \text{if $\det \nabla \eta > 0$ a.e. in $Q$ } \\
+ \infty & \text{otherwise} 
\end{cases}
\end{align}
where we use the notation $|\nabla \eta^T \nabla \eta-I|_\mathcal{C}:= \big(\mathcal{C} (\nabla \eta^T \nabla \eta- I) \big) \cdot \big( \nabla \eta^T \nabla \eta- I \big)$ where $\mathcal{C}$ is the positive definite tensor of elastic constants and $q>n$ and $a>\frac{qn}{q-n}$. 

Notice that in \eqref{st-venant} the first term corresponds to the Saint Venant-Kirchhoff energy, the second models the resistance of the solids to infinite compression and the last is a regularisation term.

From an analytical point of view, the potentials \eqref{kelvinVoigt}-\eqref{st-venant} already carry the \emph{full difficulty} that we will need to cope with. Namely, any deformation $\eta$ of finite energy will necessarily be a 
local diffeomorphism; in fact we will even strengthen this condition in the presented analysis and construct weak solutions to \eqref{strong1}-\eqref{strong3} for which $\eta$ is \emph{globally {injective}}. This \emph{geometrical restriction} is necessary not only from the point of view of physics of solids per se but also essential to properly set-up the fluid-structure interaction problem. Moreover, $E(\eta)$ from \eqref{st-venant} is neither convex nor quasi-convex. And for physical reasons explained below the dissipation potential $R$ has to depend on the state $\eta$ and cannot depend just on $\partial_t \eta$. We discuss the modeling issues in \autoref{subsec:restrict} 
while we explain the mathematical difficulties in \autoref{subsec:results} 
below.

Additionally, we impose coupling conditions between $\eta$ and $v$ on their common interface; %
namely, we will assume the \emph{continuity of deformation} (i.e.\ no-slip conditions adapted to the moving domain) as well as \emph{traction on the boundary between the fluid and the solid}. We denote by $M$ the portion of the boundary of $Q$ that is mapped to the contact interface between the fluid and the solid. While $Q$ is only assumed to be a Lipschitz-domain, we assume that the pieces of its boundary that belong to $M$ are additionally $C^2$. %
\begin{align}
\label{couplin1} v(t,\eta(x))&=\partial_t \eta(t,x) &\text{ in }[0,T]\times M, \\ \label{coupling2}
\sigma(t,x) n(x)& = \big(\nu \nablasym v(t,\eta(t,x)) + p(t, \eta(t,x))I\big) \hat{n}(t,\eta(t,x)) &\text{ in }[0,T]\times M,
\end{align}  
where $n(x)$ is the unit normal to $M$ while $\hat{n}(t,\eta(t,x)):=\cof(\nabla \eta(t,x))n(x)$ is the normal transformed to the actual configuration and $\nablasym v := \nabla v + (\nabla v)^T$ is the symmetrized gradient. 
Additionally, there are second order Neumann-type zero boundary conditions arising from the second gradient.\footnote{Specifically, these naturally occur while minimizing and can be seen as a kind of integrability condition for $\sigma$. I.e.\ for $\sigma$ to be defined as a measure, we need that $\inner{DE(\eta)}{\phi_\delta} \to 0$ for $\phi_\delta$ a regularized version of $\xi \delta (1-\dist(.,\partial Q)/\delta)^+$ and some $\xi \in C_0^\infty(M)$ extended constantly along the normal direction. For our example energy this simply reduces to $\pd[^2\eta]{n^2} = 0$ on $M$.} %

Finally, we will prescribe Dirichlet boundary conditions on  $P:=\partial Q \setminus M$, i.e. %
\begin{align}
\eta(x,t) = \gamma(x) \qquad \text{ in } [0,T] \times P
\end{align}
for some fixed boundary displacement $\gamma:P \to \Omega$. Together with the injectivity of deformations, we will encode this condition in the set of admissible deformatione $\mathcal{E}$ (See \autoref{rem:ciarletNecas} for the precise definition).

We close the system by introducing initial conditions for $v,\eta,\partial_t\eta$:
\begin{align}
\label{eq:initial}
\begin{aligned}
\eta(0,x)&=\eta_0(x)\text{ for }x\in Q
\\
\partial_t\eta(0,x)&=\eta_*(x)\text{ for }x\in Q
\\
\Omega(0)&=\Omega\setminus \eta_0(Q)
\\
v(0,y)&=v_0(y)\text{ for all }y\in \Omega(0). 
\end{aligned}
\end{align}

\subsection{Main result}  
\label{subsec:results}

The final objective of this paper is to prove existence of weak solutions to the system \eqref{strong1}-\eqref{strong3} subject to the coupling conditions \eqref{couplin1}-\eqref{coupling2} and the remaining boundary and initial conditions detailed in the previous subsection. 

As is customary in fluid-structure interaction problems (see e.g.~\citation{grandmontExistenceThreeDimensionalSteady2002,lengelerWeakSolutionsIncompressible2014}),  the weak formulation is designed in such a way that the coupling conditions are realized by choosing \emph{well-fitted test functions}. Indeed, we have the following definition:

\begin{definition} 
\label{def:NSweakSol}
Let $f\in C^0([0,\infty)\times\Omega)$, $v_0\in L^2(\Omega)$, $\eta_*\in L^2(Q)$ and $\eta_0\in \mathcal{E}$, such that $v_0\circ\eta_0=\eta_*$. 
We call\footnote{We use standard notation for Bochner spaces over Lebesgue spaces and Sobolev spaces with time changing domains.  By the subscript $\diver$ we mean the respective solenoidal subspace: $W^{1,2}_{\diver}(\Omega(t);\R^n)=\{v\in W^{1,2}(\Omega(t);\R^n)|\diver v=0\}$.} $\eta:[0,T]\times Q\to \Omega$, $v:[0,T]\times \Omega(t)\to \mathbb{R}^n$ and $p:[0,T]\times \Omega(t)\to \mathbb{R}$, where $\Omega(t) := \Omega\setminus \eta(t,Q)$, a a weak solution to the fluid-structure interaction problem \eqref{strong1}--\eqref{eq:initial}, if the following holds:

The deformation satisfies $\eta\in L^2([0,T];W^{2,q}(Q;\R^n))\cap W^{1,2}([0,T];W^{1,2}(Q;\R^n))$, $\eta(0)=\eta_0$ such that $\partial_t\eta\in C_w([0,T];L^2(\Omega))$. The velocity satisfies $v\in L^2([0,T];W^{1,2}_{\diver}(\Omega(\cdot);\R^n))$ %
 and the pressure satisfies\footnote{From the given weak formulation one can deduce some more regularity of the pressure. However as is known from the non-variable theory, regularity of the pressure in time can be obtain merely in a negative Sobolev space. See the estimates in \autoref{subsec:proofNS}, Step 3b, which show that the pressure is in the respective natural class.} $p\in D'([0,T]\times \Omega)$ with $\supp(p)\subset [0,T]\times \overline{\Omega(t)}$.
 
For all
\[(\phi,\xi)\in L^2([0,T];W^{2,q}(Q;\R^n)) \cap W^{1,2}([0,T];W^{1,2}(Q;\R^n))\times C^\infty([0,T];C^\infty_0(\Omega;\R^n))
\]
 satisfying $\xi(T)=0$, $\phi(t)=\xi(t) \circ \eta(t)$ on $Q$, $\phi(t) = 0$ on $P$ for all $t \in [0,T]$, we require that
 \begin{align*}
  & \int_0^T -\rho_s \inner[Q]{\partial_t \eta}{\partial_t \phi} - \rho_s \inner[\Omega(t)]{v}{\partial_t \xi- v \cdot \nabla \xi} + \inner{DE(\eta)}{\phi} + \inner{D_2R(\eta,\partial_t \eta)}{\phi} + \nu \inner[\Omega(t)]{\nablasym v}{\nablasym \xi} dt\\
  &\,\,= \int_0^T \inner[\Omega(t)]{p}{ \diver \xi} + \rho_s \inner[Q]{f\circ \eta}{\phi} + \rho_f \inner[\Omega(t)]{f}{\xi} dt - \rho_s \inner[Q]{\eta_*}{\phi(0)} - \rho_f \inner[\Omega(0)]{v_0}{\xi(0)}
 \end{align*} 

and that $\partial_t \eta(t) = v(t) \circ \eta(t)$ on $M$, $\eta(t) \in \mathcal{E}$ and $v(t)|_{\partial \Omega} = 0$ for almost all $t \in [0,T]$. 
\end{definition}

We can then formulate our main theorem as 

\begin{theorem}[Existence of weak solutions] \label{thm:NSexistence}
Assume that $E$ satisfies Assumption~\ref{ass:energy} and $R$ satisfies Assumption~\ref{ass:dissipation} given in \autoref{subsec:restrict}.
 Then for any $\eta_0\in \operatorname{int}(\mathcal{E})=\mathcal{E} \setminus \partial \mathcal{E}$ (see \autoref{rem:ciarletNecas}) with $E(\eta_0) < \infty$, any $\eta_*\in L^2(Q; \R^n)$, $v_0\in L^2(\Omega(0); \R^n)$ and any right hand side $f\in C^0([0,\infty)\times \Omega; \R^n)$ there exists a $T >0$ such that a weak solution to \eqref{strong1}-\eqref{eq:initial} according to \autoref{def:NSweakSol} exists on $[0,T)$. Here either $T=\infty$ or the time $T$ is the time of the first contact of the free boundary of the solid body either with itself or $\partial \Omega$ (i.e.\ $\eta(T) \in \partial \mathcal{E}$). %
 
 Moreover, the solution satisfies the energy inequality~\eqref{eq:energ}  and the respective a-priori estimates~\eqref{eq:apri}; for additional regularity of the pressure see~(\ref{eq:press}).%
\end{theorem}

Let us remark that the assumptions on the energy and dissipation functional are in particular satisfied by the model case energies \eqref{kelvinVoigt}-\eqref{st-venant}. (See \autoref{subsec:example})

The coupled system possesses a natural energy inequality which reads
\begin{align}
\label{eq:energ}
\begin{aligned}
&\phantom{{}={}}E(\eta(t))+\rho_s\int_Q\frac{\abs{\partial_t\eta(t)}^2}{2} dx+\rho_f\int_{\Omega(t)}\frac{\abs{v(t)}^2}{2}\, dy 
\\
& + \int_0^t 2R(\eta(s),\partial_t\eta(s))\,+\nu\int_{\Omega(s)}\abs{\nablasym v(s)}^2\, dy\, ds
\\
& \leq E(\eta_0)+\rho_s\int_Q\frac{\abs{\eta_*}^2}{2} dx+\rho_f\int_{\Omega(0)}\frac{\abs{v_0}^2}{2}\, dy
\\
& +\int_0^t\rho_s\int_Q f(s)\circ\eta(s)\cdot \partial_t\eta(s)\, dx+\rho_f \int_{\Omega(s)}f(s)\cdot v(s)\, dy\, ds.
\end{aligned}
\end{align}
As is usual in evolution-equations, this inequality holds as an equality for sufficiently regular solutions, i.e.\ if $(\partial_t \eta,v)$ can be used as a pair of test functions.%

We introduce the methodology of the proof of Theorem \ref{thm:NSexistence} along three main ideas, with each of them being of independent interest and each explained in a separate section. The \autoref{sec:qs} and \autoref{sec:so} can be read pretty much independently. The last section \autoref{sec:full} however, relies on both preceding sections. 
We will at this point shortly comment on them and then refer the reader to the respective parts of the paper in which the analysis is carried out in full rigor. 

\noindent \textbf{\autoref{sec:qs}: Minimizing movements for fluid-structure interactions.}
In this section we consider a reduced parabolic system of equations with the inertial terms omitted, namely
\begin{align}
\label{strong1-parabolic}
\diver \sigma(\eta) &=\rho_s f\circ \eta &\text{ in } Q, \\
0&=\nu \Delta v  -\nabla p +\rho_f f     &\text{ on } \Omega(t), \\
\diver v&=0 &\text{ on } \Omega(t),\label{strong3-parabolic}
\end{align}
together with the same coupling and boundary conditions as before. 

We construct a weak solution to \eqref{strong1-parabolic}-\eqref{strong3-parabolic} by an implicit-explicit time-discretization scheme that exploits the \emph{variational structure} of the problem, namely that both the stress tensor for the solid as well as the {fluid} have a potential. %

Indeed, let us split $[0,T]$ into $N$ equidistant time steps of length $\tau$. Assume, for $k\in \{0, \ldots, N-1\}$, that $\eta_k$ is given and denote $\Omega_{k} = \Omega \setminus \eta_{k}(Q)$. We then define $\eta_{k+1},v_{k+1}$ to be the minimizer of the following functional
\begin{align} \label{eq:introDiscrete}
E(\eta) + \tau R\left(\eta_k,\frac{\eta-\eta_k}{\tau}\right) +  \frac{\tau \nu}{2} \norm[\Omega_k]{\nabla v}^2 - \tau \rho_s \inner[Q]{f\circ \eta_k}{\frac{\eta-\eta_k}{\tau}} - \tau \rho_f\inner[\Omega_k]{f}{v} \longrightarrow \text{min}
\end{align}
under all $\eta \in \mathcal{E}$, $v \in W^{1,2}(\Omega_k;\R^n)$ satisfying $\diver v = 0$, $v|_{\partial \Omega} = 0$ and the affine coupling condition $v \circ \eta_k=\frac{\eta-\eta_{k}}{\tau}$ in $M$. Minimizing with respect to this coupling is essential here, as it not only represents the discretized version of the coupling of velocities \eqref{couplin1}, but also will result in the equality of tractions \eqref{coupling2} via the minimization procedure. %

Searching for time-discrete approximations of the weak solution to \eqref{strong1-parabolic} alone via a minimization problem similar to the one above is actually well known and heavily used in the mathematics of continuum mechanics of solids. The method is known as the \emph{method of minimizing movements} %
or, in particularly in the engineering literature, %
also called the \emph{time-incremental problem.} As far as the authors are aware this method has not been applied to the theory of fluid-structure interaction problems before.

The advantage of the variational approach in contrast to directly solving the corresponding Euler-Lagrange equations is twofold. Not only do we deal with the non-convexity of $E$ and the underlying non-convex space $\mathcal{E}$ in a natural way, but also we automatically gain an \emph{energetic a-prori estimate}. Indeed comparing the value of the functional in \eqref{eq:introDiscrete} in $(\eta_{k+1}, v_{k+1})$ with its value for $(\eta_{k}, 0)$ and iterating, we get

 \begin{align}
 \label{apriori-intro}
  &\phantom{{}={}}\underbrace{E(\eta_{k+1})}_{\text{Final energy}} + \underbrace{\sum_{l=0}^{k} \tau \left[R\left(\eta_l, \frac{\eta_{l+1}-\eta_l}{\tau}\right) + \frac{\nu}{2} \norm[\Omega_k]{\nablasym v_{l+1}}^2 \right]
   }_{\text{$1/2$ of Dissipation}} \\ \nonumber
   &\leq \underbrace{E(\eta_0)}_{\text{Initial energy}} + \underbrace{\sum_{l=0}^{k} \tau \left[\rho_s \inner[Q]{f\circ \eta_l}{\frac{\eta_{l+1}-\eta_l}{\tau}}+ \rho_f \inner{f}{v_{l+1}}\right] }_{\text{Work from forces}}.
 \end{align}
Starting from this energetic estimate some other analytic tools need to be developed: In order to deduce a-priori estimates a Korn's inequaliy (\autoref{lem:globalKorn}) estimating fluid and solid velocity simultaneously is introduced. The limit equation is then established via the usual weak compactness results, the so-called Minty-Method (See \autoref{prop:QSELeq}) and a subtle approximation of test-functions (See \autoref{lem:approxTestFcts}). The latter is necessary due to the fact that the fluid-domain (the part where the test-function is supposed to be solenoidal) is a part of the solution.

 \begin{remark}[Chain rule] \label{rem:IntroChainRule} %
 If we considered Euler-Lagrange equations instead of the variational problem \eqref{eq:introDiscrete} the standard way to obtain a-priori estimates  would be testing with $(\eta_{k+1}-\eta_k)$ and $v_{k+1}$. To obtain similar estimate to before, we would then need to use a discrete variant of a chain-rule to see that 
 \[
E(\eta_{k+1})-E(\eta_k) \leq \inner{DE(\eta_{k+1})}{\eta_{k+1}-\eta_k}.
 \]
 While this inequality is valid if $E$ is convex, it is generally not true otherwise. In particular note (see also \autoref{subsec:restrict}), that convexity of $E$ is ruled out for physical reasons in mechanics of solids.
 \end{remark}

\noindent \textbf{\autoref{sec:so}: Minimizing movements for hyperbolic evolutions.} 

As we have seen in the previous step, discretizing the parabolic equation via a variational scheme is essential to deal with the non-convexities involved in the problem. However, once an inertial, or hyperbolic, term is present, comparing values in the minimization problem no longer leads to useful estimates. Instead, upon presence of this term, it seems more advantageous to obtain a-priori estimates from the Euler-Lagrange equation, which in turn is incompatible with the non-convexity of the energy again. 
To overcome this difficulty, we consecutively approximate %
using \emph{two different time-scales}: the {\em velocity scale} $\tau$ and the {\em acceleration scale} $h$. Keeping the acceleration scale fixed at first, we may use the strategy from the previous step and obtain, after passing to the limit $\tau \to 0$, a \emph{continuous in time} equation from which a second set of a-priori estimates can be seen.

In order to explain the concept, we us first illustrate this procedure on the  solid alone: We thus consider $\eta: [0,T]\times Q \to \R^n$, evolving according to
 \[DE(\eta) + DR(\eta, \partial_t \eta) - f\circ \eta = \rho_s \partial_t^2 \eta\]
 without any coupling with a fluid, but with otherwise similar initial data to before.
 
 We now want to turn this hyperbolic problem into a sequence of short time consecutive parabolic problems, by replacing the second time derivative $\partial_t^2 \eta$ with a difference quotient and solving what we will call the \emph{time-delayed problem}
 \begin{align} \label{eq:introTimeDelayed}
  DE(\eta(t)) + DR(\eta(t), \partial_t \eta(t)) - f\circ \eta(t)  = \frac{\partial_t \eta(t) - \partial_t \eta(t-h)}{h}
 \end{align}
 for some fixed $h$. 
 
 Considered on a short interval of length $h$, the term $\partial_t \eta(t-h)$ can be seen as fixed given data. Then on this interval the problem is parabolic and can be solved using a similar minimizing movements approximation as to what was described before, where for fixed $h$ we pick $\tau << h$ and solve a problem similar to %
 \begin{align*}
E(\eta) + \tau R\left(\eta_k,\frac{\eta-\eta_k}{\tau}\right) - \tau \rho_s \inner[Q]{f\circ \eta_k}{\frac{\eta-\eta_k}{\tau}}+ \frac{1}{2h} \norm{\frac{\eta-\eta_k}{\tau} - \partial_t \eta(\tau k -h)}^2 \longrightarrow \text{min.}
\end{align*} 

Upon sending $\tau \to 0$ using the same techniques as before, we then obtain a weak solution to \eqref{eq:introTimeDelayed} on $[0,h]$ which can be used as data on $[h,2h]$ and so on, until we have derived a solution on $[0,T)$.

Now, as the time-delayed equation \eqref{eq:introTimeDelayed} is continuous, we avoid the problem with the chain rule (\autoref{rem:IntroChainRule}) and  we can test\footnote{In order to guarantee that $(\partial_t\eta)$ is a admissible test-function we have to include a parabolic regularizer in the dissipation functional of the solid. As the resulting estimate is only needed for $h>0$ and independent of the regularizer, we can choose it in such a way that it vanishes in the limit $h\to 0$.} with $\partial_t \eta$ and obtain an energy inequality for the time-delayed problem in the form of
\begin{align*}
 E(\eta(t)) + \rho_s \fint_{t-h}^t \frac{\norm[Q]{\partial_t \eta(s)}^2}{2} ds +\int_0^t 2 R(\eta,\partial_t \eta) ds \leq E(\eta_0) + \rho_s \frac{\norm[Q]{\eta_*}^2}{2} + \int_0^t \inner[Q]{f\circ \eta}{\partial_t \eta} ds
\end{align*}
Similar to before we can then turn this into an a-priori estimate for the hyperbolic problem, in order to finally converge the acceleration scale and send $h \to 0$.

We will discuss this in full rigour and prove existence for solutions to the hyperbolic evolution of a solid in \autoref{sec:so}.
 
 \begin{remark}[Previous works]
While our approach to existence of solutions to hyperbolic evolutions seems to be new, we have to mention that previous variational approaches for similar PDEs have been proposed. However, they either rely on convexity or, more general, \emph{polyconvexity} (which is the convexity on minors) \cite{demouliniVariationalApproximationScheme2001,miroshnikovVariationalApproximationScheme2012} or use more explicit schemes which don't work well with the injectivity considerations \cite{demouliniWeakSolutionsClass2000}. In particular let us notice that, for the structure, the scheme proposed here is fully implicit which has the advantage that we are assured that the constructed solution is, even in the discrete setting, fulfilling all relevant non-linear constraints at all times, in particular global (almost-)injectivity.\footnote{To be more precise, we will consider the Ciarlet-Ne\v{c}as condition, which guarantees injectivity up to a set of measure zero and which has the important property of being preserved under the relevant convergence.} %
 \end{remark}

\noindent \textbf{\autoref{sec:full}: Bulk elastic solids coupled to Navier-Stokes equations}

In this section we combine the previous two sections and apply the two time-scale approach to the fluid-structure interaction problem. The main obstacle here lies in the Eulerian description of the fluid. What turned out to be natural is to approximate the \emph{material derivative of the fluid velocity} $(\partial_tv+[\nabla v]v)$ by a time-discrete differential quotient. This is done by subsequently introducing a flow map $\Phi_s(t): \Omega(t) \to \Omega(t+s)$ fulfilling $\partial_s \Phi_s(t,y)=v(t+s,\Phi_s(t,y))$ (resp. a discrete version of this) and $\Phi_0(t,y) = y$ in both the discrete and the time-delayed approximation layers. It means that $\Phi$ \emph{transports} the domain of the fluid along with its velocity. 

In particular the fluid analogue of the difference quotient in the time-delayed problem will be a ``material difference quotient'' in the size of the acceleration scale $h$, which is essentially of the form
\[
\frac{v(t,\Phi_h(t-h,y))-v(t-h,y)}{h}.
\]
As $\Phi$ and $v$ are inseperably linked we already need to construct discrete counterparts of both alongside each other in the $\tau$ scale. This discrete construction of the highly nonlinear $\Phi$ and its subsequent convergence are one of the main additions in the proof and one of technical centerpieces of the paper.%

As in Step 2, an approximation of the final energy estimates~\eqref{eq:energ} is obtained only \emph{after the first limit passage} $\tau \to 0$. More precisely, they are again obtained by testing the coupled Euler-Lagrange equation of the continuous in time quasi-steady approximation with the fluid and solid velocities. %
The energy inequality obtained by testing differs from \eqref{eq:energ} by replacing the kinetic energies with their moving averages. From this we can then derive an a-priori estimate and start on convergence. A particular difficulty in this is the material derivative of the fluid-velocity. Here we need to derive a modified Aubin-Lions lemma in order to obtain stronger convergence for a time-average approximation of $u$ which is the natural quantity in this context.

\begin{remark}[Previous variational approaches to Navier-Stokes]
 While the minimizing movements method seem to be new in the field of fluid-structure interactions, it has been previously used to show existence of solutions to the Navier-Stokes equation. In particular we want to highlight \cite{gigliVariationalApproachNavier2012} as an inspiration. There the authors also employ flow maps to obtain the material derivative, but as they work on a fixed domain, they do not need to construct them iteratively but can instead rely on the respective existence theory for the Stokes-problem. As an indirect consequence, their minimization happens on what we would consider the $h$-level, which makes it incompatible with our way of handling the solid evolution. Instead it is more of an improvement of the numerical scheme \cite{Pir82} which is has been developed much earlier.
\end{remark}

\subsection{Mechanical and analytical restrictions on the energy/dissipation functional}
\label{subsec:restrict}

As introduced in Section \ref{sec:setup}, we consider solid materials for which the stress tensor can be determined by prescribing two functionals; \emph{the energy and dissipation functional}. Materials admitting such modeling are called \emph{generalized standard materials} \cite{halphenMateriauxStandardGeneralises1975,nguyenStabilityNonlinearSolidMechanics2000,kruzik2019mathematical} %
and many available rheological models fall into this frame~\cite{kruzik2019mathematical}. Nonetheless, the two functionals cannot be chosen completely freely, but have to comply with certain physical requirements. We summarize these at this point.

As in examples \eqref{kelvinVoigt} and \eqref{st-venant}, we will for the sake of discussion assume that the energy and dissipation functional have a density, i.e.
\begin{equation}
E(\eta) = \int_Q e(\nabla \eta, \nabla^2 \eta) dx \qquad R(\eta, \partial_t \eta) = \int_Q r(\nabla \eta, \partial_t \nabla \eta) dx,
\label{eq_densities}
\end{equation}
for all smooth vector fields $\eta: Q \to \R^n$.

Here, the energy density depends on the first and second gradient\footnote{The formalism of our proofs naturally also allow for dependence on material and spatial positions $x$ and $\eta(x)$, but the latter dependence is non-physical and the former does not add much to the discussion. Nevertheless we emphasize that our results hold also for inhomogeneous materials.} of the deformation which puts us into the class of so-called non-simple (or second grade) materials (See the pioneering work \cite{toupinElasticMaterialsCouplestresses1962} as well as \cite{silhavyPhaseTransitionsNonsimple1985,friedTractionsBalancesBoundary2006} for later development). In fact, allowing for the energy density to depend on higher order gradients puts us beyond the standard theory of hyperelasticty but allows us to bring along more regularity to the problem. 

Any admissible $e(F,G)$ in \eqref{eq_densities} should satisfy the frame-indifference
\[
e(RF, GR) = e(F, G) \qquad \forall F \in \R^{n \times n}, G \in \R^{n \times n \times n},
\]
for any proper rotation $R$. In other words, the energy remains unchanged upon a change of observer. Moreover, any physical energy will blow up if the material is to be extended or compressed infinitely i.e.
\[
e(F,G) \to \infty \qquad \text{if $|F| \to \infty$ or $\det F \to 0$}
\]
and it should prohibit change of orientation for the deformations i.e.\ 
\[
e(F,G) = \infty \qquad \text{if $\det F \leq 0$.}
\]
From an analytical point of view these basic requirement have the following consequences: $e(F,G)$ \emph{cannot be a convex function} of the first variable and $e(F,G)$ \emph{cannot be bounded}. As a result considering the variational approach in the parabolic fluid-structure interaction as well as the two-scale approximation in the hyperbolic case are essential. 

Going even further, while these conditions lead to non-convexities, they are at the same time beneficial. Assuming appropriate growth conditions, in particular for $\det F \to 0$, we will be able to deduce a uniform lower bound on the determinant of $\nabla \eta$ in the style of \cite{healeyInjectiveWeakSolutions2009}. This will not only result in a meaningful boundary for the fluid domain, but also help us to readily switch between Lagrangian and Eulerian descriptions of the solid velocity.

As for the dissipation potential, we will also need it to independent of the observer, i.e.\ for all smoothly time-varying proper rotations $R(t)$, and all smooth time-dependent $F: [0,T] \to \R^n$
\[
r(RF, \partial_t(RF)) = r(F, \partial_t F).
\]
This restriction implies \cite{antmanPhysicallyUnacceptableViscous1998} that $r$ \emph{cannot} depend only on $\partial_t \eta$ but needs to depend on $\eta$, too. This in turn, will require us to use \emph{fine Korn type inequalities} \cite{neffKornFirstInequality2002,pompeKornFirstInequality2003} to deduce a-priori estimates (as already in \cite{mielkeThermoviscoelasticityKelvinVoigtRheology2019}). %
We also note that both for physical as well as for analytic reasons, $r$ should be non-negative and convex in the second variable. We additionally require $R$ to be a quadratic form in its second variable.\footnote{See \autoref{rem:SOdissipation} for possible relaxation of the dissipation potential.}

Taking into account these, as well as some analytical requirements, we will now detail our set-up for the deformation.
\begin{definition}[Domain of definition] \label{rem:ciarletNecas}
The set of internally injective functions in $W^{2,q}(Q; \Omega)$ (and satisfying the Dirichlet boundary condition) used for minimization in \eqref{eq:introDiscrete} can be expressed as 
\begin{align}
\label{eq:etaspace}
  \mathcal{E} := \left\{\eta \in W^{2,q}(Q;\Omega): \abs{\eta(Q)} = \int_Q \det \nabla \eta \,dx, \eta|_{P} = \gamma(x)\right\}.
\end{align}
\end{definition}
Here, the equality $\abs{\eta(Q)} = \int_Q \det \nabla \eta \,dx$ is termed the \emph{Ciarlet-Ne\v{c}as condition} which has been proposed in \cite{ciarletInjectivitySelfcontactNonlinear1987} and, as has also been proved there, it assures that any $C^1$-local homeomorphism is globally injective except for possible touching at the boundary. Working with this equivalent condition bears the advantage that it is easily seen to be preserved under weak convergence in  $ W^{2,q}(Q;\Omega)$.
\begin{remark}
Of particular interest is the topology of $\mathcal{E}$. It is easy to see that the set is a closed subset of the affine space $W_\gamma^{2,q}(Q;\Omega)$ i.e.\ $W^{2,q}(Q;\Omega)$ with fixed boundary conditions. As a subset of this topological space it has both interior points (denoted by $\operatorname{int}(\mathcal{E})$) and a boundary $\partial \mathcal{E}$. As we construct our approximative solutions by minimization over $\mathcal{E}$, it is crucial to know if $\eta_k \in \operatorname{int}(\mathcal{E})$ as only then we are allowed to test in all directions and have the full Euler-Lagrange equation we need.

Luckily however $\operatorname{int}(\mathcal{E})$ and $\partial \mathcal{E}$ are easily quantifiable. As long as $\det \nabla \eta > 0$, which is true for finite energy, we are able to vary in all directions, if and only if $\eta|_{M}$ is injective and does not touch $\partial \Omega$. Thus $\partial \mathcal{E}$, or rather the part that is relevant for us, consists precisely of the $\eta$ which have a collision.
\end{remark}
Throughout the paper, for the elastic energy potential defined on $\mathcal{E}$ we impose the following assumptions.
\begin{assumption}[Elastic energy] 
\label{ass:energy}
We assume that $q>n$ and $E:\mathcal{E} \to \overline{\R}$ satisfies:
\begin{enumerate}
 \item[S1] Lower bound: There exists a number $E_{min} > -\infty$ such that 
 \[E(\eta) \geq E_{min} \text{ for all } \eta \in W^{2,q}(Q; \R^n)\]
 \item[S2] Lower bound in the determinant: For any $E_0>0$ there exists $\epsilon_0 >0$ such that $\det \nabla \eta \geq \epsilon_0$ for all $\eta \in \{\eta \in W^{2,q}(Q; \R^n):E(\eta) <E_0\}$.
 \item[S3] Weak lower semi-continuity: If $\eta_l \rightharpoonup \eta$ in $W^{2,q}(Q; \R^n)$ then $E(\eta) \leq \liminf_{l\to\infty} E(\eta_l)$.
 \item[S4] Coercivity: All sublevel-sets $\{\eta \in \mathcal{E} :E(\eta) <E_0\}$ are bounded in $W^{2,q}(Q; \R^n)$.
 \item[S5] Existence of derivatives: For finite values $E$ has a well defined derivative which we will formally denote by 
 \[DE:  \{ \eta \in \mathcal{E}: E(\eta) < \infty\} \to (W^{2,q}(Q; \R^n))'\]
 Furthermore on any sublevel-set of $E$, $DE$ is bounded and continuous with respect to strong $W^{2,q}$-convergence.
 \item[S6] Monotonicity and Minty type property: If $\eta_l \rightharpoonup \eta$ in $W^{2,q}(Q; \R^n)$, then \[\liminf_{l\to \infty} \inner{DE(\eta_l)-DE(\eta)}{(\eta_l-\eta)\psi} \geq 0 \text{ for all }\psi\in C^\infty_0(Q;[0,1]).\] Additionally if $\limsup_{l\to \infty} \inner{DE(\eta_l)-DE(\eta)}{(\eta_l-\eta)\psi} \leq 0$ then $\eta_l\to \eta$ in $W^{2,q}(Q; \R^n)$.
\end{enumerate}
\end{assumption}
Let us shortly elaborate on the above stated assumptions. As elastic energies are generally bounded from below, assumption S1 is a natural one. Similarly, assumption S5 is to be expected as we need to take the derivative of the energy to determine a weak version of the Piola-Kirchhoff stress tensor. Assumptions S3 and S4 are standard in any variational approach as they open-up the possibility for using the direct method of the calculus of variations. Assumption S6 effectively means that the energy density has to be convex in the highest gradient (but of course not convex overall) and allows us to get weak solutions and not merely measure valued ones (as in the case of a solid material in \cite{demouliniVariationalApproximationScheme2001}). Finally, assumption S2 is probably the most restricting one and, to the authors' knowledge, necessitates the use of second-grade elasticity when, combined with an energy density $e$ which blows up sufficiently fast as $\det F \to 0$ (see \cite{healeyInjectiveWeakSolutions2009}). This is, in particular, the case for the model energy \eqref{st-venant}.

For the dissipation functional we have the following assumption:
\begin{assumption}[Dissipation functional]
The dissipation $R:\mathcal{E}\times W^{1,2}(Q;\R^n)\to \mathbb{R}$ satisfies
\label{ass:dissipation}
\begin{enumerate}
  \item[R1] Weak lower semicontinuity: If $b_l \rightharpoonup b$ in $W^{1,2}$ then
  \[\liminf_{l\to\infty} R(\eta,b_l) \geq R(\eta,b)\]
  \item[R2] Homogeneity of degree two: The dissipation is homogeneous of degree two in its second argument , i.e.
  \[R(\eta, \lambda b) = \lambda^2 R(\eta, b) \qquad \forall \lambda \in \R\]
  In particular, this implies $R(\eta,b) \geq 0$ and $R(\eta,0) = 0$.
  \item[R3] Energy-dependent Korn-type inequality: Fix $E_0 >0$. Then there exists a constant $c_K = c_K(E_0) >0$ such that for all $\eta \in W^{2,q}(Q;\R^n)$ with $E(\eta) \leq E_0$ and all $b \in W^{1,2}(Q;\R^n)$ with $b|_{P} = 0$ we have
  \[c_K \norm[W^{1,2}(Q)]{b}^2 \leq R(\eta,b).\]
  \item[R4] Existence of a continuous derivative: The derivative $D_2R(\eta,b) \in (W^{1,2}(Q))'$ given by
  \[ \frac{d}{d\epsilon} |_{\epsilon = 0} R(\eta,b+\epsilon \phi) =: \inner{D_2R(\eta,b)}{ \phi}\]
  exists and is weakly continuous in its two arguments. Due to the homogeneity of degree two this in particular implies
  \[\inner{D_2R(\eta,b)}{ b } = 2R(\eta,b).\]
 \end{enumerate}
\end{assumption}
Again some remarks are in order. As above assumption R4 is natural as we need do be able to evaluate the actual stress. Assumption R2, on the other hand, reflects the fact that we are considering viscous dissipation. Assumption R1 is again important from the point of view of calculus of variations. Assumption R3 is a coercivity assumption in a sense and needs to be stated in this rather weak form %
to satisfy frame indifference. Indeed, our model dissipation \eqref{kelvinVoigt} satisfies this assumption as shown in, e.g., \cite{mielkeThermoviscoelasticityKelvinVoigtRheology2019} relying on quite general Korn's inequalities due to \cite{neffKornFirstInequality2002,pompeKornFirstInequality2003}.

\subsection{Outlook and further applicability}
Each of the three parts outlined has the potential for further generalizations. We include some of the possible extensions at this point. 

\vspace{1ex}

\noindent \textbf{More general material laws for fluids:} Within this work we cover incompressible Newtonian fluids. However, more general fluid laws could be considered with the main requirement being that the stress tensor posesses a potential that can be used in place of $\nu\norm{\nablasym v}^2$. Examples include non-Newtonain incompressible fluids (power law fluids), or even Newtonian compressible fluids, for which fluid-structure interactions involving simplified elastic models have alread been studied~\cite{lengelerWeakSolutionsIncompressible2014a,breitCompressibleFluidsInteracting2018}. The latter include the added difficulty of density as an additional state variable.

\vspace{1ex}

\noindent \textbf{Evolution in solid mechanics including inertia:} While the hyperbolic evolution of solid materials including inertia has already been studied \cite{demouliniWeakSolutionsClass2000,demouliniVariationalApproximationScheme2001}, the scheme presented here has the potential to provide generalizations of those works. First, in the case when higher gradients are present in the energy, it might be possible to prove existence of injective solutions in full dimension. So far, such results have been limited to special geometries only, like the radial symmetric case \cite{miroshnikovVariationalApproximationScheme2012}. Moreover, it might be possible to prove existence of measure valued solutions to elastodynamics even for quasiconvex energies.

\vspace{1ex}

\noindent \textbf{Non-quadratic dissipation of the solid:}
As is the case for fluids, so it also holds for solids that quadratic dissipation potentials are the most common. They are however by far not the only ones. While seems to be relatively straightforward to extend our proofs to other, strictly convex potentials, a more complicated, but more interesting case might be that of rate-independent problems, i.e. problems for which $R$ only has linear growth in its second argument. These have already been studied extensively in the case of quasistatic evolutions of solids \cite{mielkeRateIndependentSystems2015}. The main problem there is that their invariance under reparametrisation allows for sudden jumps, something that could be mitigated by considering inertial effects or coupling with a fluid.

\vspace{1ex}

\noindent \textbf{Contact conditions:} The time-stepping existence scheme introduced here, i.e.\ the minimization, produces results \emph{for arbitrarily large times}, over the point of self touching. This allows to gain a global in time object. However in order to show that this object has a meaning as a solution to the given contact-problem, further work needs to be done.%
Even for the case of an elastic solid deformations alone the understanding of self-touching of the solid (or alternatively touching the container) is still unclear from a mathematical point of view; in particular understanding the \emph{contact force} on the solid that is preventing the self-penetration. Some contact-forces have been identified in the fully static problem only recently in \cite{palmerInjectivitySelfcontactSecondgradient2017} and a generalization to the quasi-static situation is due to \cite{kromerQuasistaticViscoelasticitySelfcontact2019}. Generalizations to fluid-structure interaction or even hyperbolic evolution, however, remain widely open.

A different point of view would be to use the fluid as a damping substance. It can be shown (in case the solid surface is smooth) that the contact is prevented by the incompressible fluid. Indeed, the works of Hillariet and Hesla \cite{hillairetLackCollisionSolid2007,hesla2004collisions}, see also \cite{grandmontExistenceGlobalStrong2016} show that two smooth objects can not touch each other in case there is a viscous incompressible fluid in between.\footnote{More precisely, the no-contact result relays on the assumption of 1) smooth surfaces of the solids, 2) bounds on the time derivative of the fluid velocity (which for large times and for the Navier Stokes equations is only proven in 2D), 3) the no-slip boundary conditions of the fluid velocity, which means the equality of the velocities at the interface (as we consider in this work).} Certainly, once self-contact of the solid can be excluded a-priori (due to the fluid surrounding it) global solutions are available.%

\vspace{1ex}

\noindent \textbf{Limit passage to simpler geometries:} Mathematically, fluid-stucture evolution is much better understood if the solid object is of lower dimension (see e.g.~\cite{muhaExistenceWeakSolution2013,lengelerWeakSolutionsIncompressible2014}). 

While passage to lower dimensional objects for solid materials has been studied (see e.g. \cite{ledretQuasiconvexEnvelopeSaint1995,frieseckeHierarchyPlateModels2006}), coupling with the fluid and passing to the limit in a similar manner is another relevant topic.

\vspace{1ex}

\noindent \textbf{Coupling to other physical phenomena:} Here we consider  a strictly mechanical system but coupling to further physical phenomena such as heat transfer is worthwhile studying. Indeed, as heat transfer happens in the actual configuration even results for the solid alone are sparse, we refer to \cite{mielkeThermoviscoelasticityKelvinVoigtRheology2019} for a recent work in the quasi-steady case. The methods from Step 2 in this work might allow a generalization to the case with inertia both for the solid alone and, ultimatively, also fluid-structure interaction.

\vspace{1ex}

\noindent \textbf{Analytical bounds on numerical approximations of the scheme:}
The schemes we are proposing are constructive in their very nature. Thus one can construct a numerical approximation without having to deviate much from the main ideas. There are certain difficulties related to minimizing a non-convex functional and the fact that we are dealing with changing domains, but both have been dealt with before. Instead the main point of interest will be the rate of convergence of any such scheme. Fundamentally our method relies on energy bounds on the acceleration scale which we are only able to obtain after convergence of the velocity scale and which may not hold for any approximation. Proving convergence of a numerical scheme would thus neccessitate first finding discrete bounds on that level. %

\subsection{Notation}
Let us detail the notation, some of which we have already used. Throughout the paper we will deal with a number of quantities that may depend on either or both of the parameters $\tau$ and $h$, corresponding to our two timescales. Whenever this dependence is relevant, i.e. when we are varying the parameter or converging to the corresponding limit, we will indicate it by a superscript, e.g. $\eta^{(\tau)} \to \eta$. There will be no context where both are relevant at the same time, but we note that in Sections \ref{sec:so} and \ref{sec:full}, anything that depends on $\tau$ is also depending on the larger scale $h$. Any such dependence may however be combined with a sequence index as well, as for example in $\eta_k^{(\tau)}$.

Of particular interest here is the fluid domain, which technically is always determined through the deformation of the solid alone. However since this deformation turn may in turn depend on those parameters, a sequence index and on time, we choose to copy this dependence over in the notation. For example we might write
$\Omega_l^{(h)}(t) := \Omega \setminus \eta_l^{(h)}(t,Q)$ or similar, whenever applicable.

Coordinate in the reference domain $Q$ will always be denoted by $x$ and coordinates in the physical domain $\Omega$ by $y$.

In order to avoid confusion with changing domains, we will almost always write out all time-integrations and try to only use norms and inner products in a spatial sense. We will also try to give the relevant domain if there is any chance of doubt. In particular a subscript $A$ will always denote the $L^2$ norm or inner product with respect to this domain, i.e.
\begin{align*}
 \norm[A]{f}^2 := \int_A \abs{f}^2 dx \text{ and } \inner[A]{f}{g} = \int_A f\cdot g dx.
\end{align*}
For other norms we will always specify the domain (but ommit the codomain), e.g. we will write $\norm[W^{1,2}(\Omega(t))]{v(t)}$.
The only exceptions to this involve the linear operators $DE(\eta)$ and $D_2R(\eta,b)$ which invariably have their domain of definition associated to them. We will thus simply write $\inner{DE(\eta)}{\phi}$ and $\inner{D_2(\eta,b)}{\phi}$ to denote the underlying $W^{2,q}(Q;\R^n) \times W^{-2,q}(Q;\R^n)$ and $W^{1,2}(Q;\R^n) \times W^{-1,2}(Q;\R^n)$-pairings. Since these only ever occur as linear operators and should in particular never be treated as functions, there should be no confusion in this.

\section{Minimizing movements for fluid-structure interactions.}
\label{sec:qs}

 Within this section, we provide existence of weak solutions to the \emph{parabolic} fluid-structure interaction problem, i.e. the inertia-less balances \eqref{strong1-parabolic}-\eqref{strong3-parabolic} together with the coupling conditions \eqref{couplin1}-\eqref{coupling2} as well as Dirichlet boundary conditions for the deformation and a Navier boundary condition stemming from the higher gradients in the energy. This has the interesting difficulty that in fluid structure interaction a naturally Lagrangian solid needs to be coupled with a naturally Eulerian fluid on a variable domain. In other works on fluid-structure interactions, this is usually done by ``normalizing'' the fluid domain, using what is called an``arbitrary Lagrangian-Eulerian map''. We will not do so, but instead use a variational approach to deal with the full problem in a varying domain.

We shall work with the following weak formulation:

\begin{definition}[Weak solution to the parabolic problem]
\label{def:QSWeakSol}
  We call the pair $(\eta, v)$ a weak solution to the \emph{parabolic fluid structure interaction problem}, if it satisfies
  \begin{align*}
    \eta & \in L^\infty([0,T];\mathcal{E}) \quad &&\text { and }  \partial_t\eta \in  L^2([0,T];W^{1,2}(Q; \R^n)), \\
    v &\in L^2([0,T];W^{1,2}(\Omega(t); \R^n)) \quad &&\text { and } \diver v(t) =0  \quad \text{ for a.a. $t \in [0,T]$,}
  \end{align*}
and there exists a $p \in \mathcal{D}'([0,T] \times \Omega)$ with $\supp p \subset [0,T] \times \Omega(t)$, such that they satisfy the weak equation
  \begin{align}
  \label{eq:QSEL}
  \begin{aligned}
  & \int_0^T \inner{DE(\eta)}{\phi} + \inner{D_2R(\eta,\partial_t \eta)}{\phi}
    + \inner[\Omega(t)]{\nablasym v}{\nablasym \xi}-\inner{p}{\diver\xi}\, dt
     \\
   &\quad 
   = \int_0^T\rho_f\inner[\Omega(t)]{f}{\xi} +  \rho_s\inner[Q]{f\circ \eta }{\phi} dt
   \end{aligned}
  \end{align}
  for all $\phi \in L^2([0,T];W^{2,q}(Q;\R^n))$, with $\phi|_P = 0$ and $\xi \in C_0([0,T];W^{2,q}_0(\Omega;\R^n))$ such that $\phi = \xi \circ \eta$ on $Q$ and as before we set $\Omega(t) := \Omega \setminus \eta(t,Q)$.
  Moreover, the initial condition for $\eta$ is satisfied in the sense that
  \begin{align*}
  \lim_{t\to 0}\eta(t)=\eta_0 \quad \text{ in } L^2(Q; \R^n).
  \end{align*}
\end{definition}

The main goal of this section is to prove existence of weak solutions to the parabolic fluid-structure interaction problem. In particular, we show the following theorem:  

\begin{theorem}[Existence of a parabolic fluid structure interaction] \label{thm:QSexistence}
 Assume that the energy $E$ fulfills \autoref{ass:energy} and the dissipation $R$ fulfills \autoref{ass:dissipation}. Further let $\eta_0 \in \mathcal{E}$ with $E(\eta_0) < \infty$ and $f \in L^\infty(\Omega;\R^n)$. Then there exists a maximal time $T_{\max}> 0$ such that on the interval $[0,T_{\max})$ a weak solution to to the parabolic fluid-structure interaction problem in the sense of \autoref{def:QSWeakSol} exists. 
 
 For the maximal time, we have $T_{\max} = \infty$, or $\liminf_{t\to T_{\max}} E(\eta(t)) = \infty$, or $T_{\max}$ is the time of the first collision of the solid with either itself or the container, i.e.\ the continuation $\eta(T_{\max})$ exists and $\eta(T_{\max}) \in \partial \mathcal{E}$. Furthermore we have $p \in L^2(0,T;L^\infty(\Omega(t)))+L^\infty(0,T;L^2(\Omega(t)))$ for all $T<T_{\max}$.
\end{theorem}

In order to prove Theorem \ref{thm:QSexistence}, we shall exploit the natural gradient flow-structure of the parabolic fluid-structure evolution. Indeed, at the heart of the proof is the construction of time-discrete approximations via variational problems inspired by DeGiorgi's \emph{minimizing movements} method \cite{degiorgiNewProblemsMinimizing1993} given in \eqref{eq:QSdiscreteProblem}. We refer to \autoref{subsec:QSproof} for a detailed proof of Theorem \ref{thm:QSexistence} and to Section \ref{subsec:QSproperties} for the preliminary material. 

\begin{remark}[Maximal existence time]
The maximal existence time in Theorem \eqref{thm:QSexistence} is not only given by possible collisions but also by a possible blow-up of the energy due to the acting forces. It is quite notable, that such a situation cannot appear in the full (hyperbolic) model \autoref{lem:NSfullIterationAPriori}. The reason is that the acting forces can be compared to the intertial term instead of the dissipative one.
\end{remark} %

\subsection{Preliminary analysis}
\label{subsec:QSproperties}

We will start this section with discussing the relevant geometry of the fluid-solid coupling and derive some necessary properties for the coupled system that will also be of use for the full Navier-Stokes system in \autoref{sec:full}.

\begin{lemma} \label{lem:calEclosed}
The set $\mathcal{E}$ defined in \eqref{eq:etaspace} is closed under weak $W^{2,q}$ convergence.
\end{lemma}

\begin{proof}
 The boundary condition holds as $W^{2,q}(Q;\R^n)$ has a continuous trace operator. The only thing left to check is the Ciarlet-Ne\v{c}as-condition. Assume that $\eta_l \rightharpoonup \eta$. Using Rellich's compactness theorem, for a subsequence $\eta_l \to \eta$ in $C^{1,\alpha}$ for some $\alpha > 0$. This implies that $\det \nabla \eta_l \to \det \nabla \eta$ uniformly, which yields that $\int_Q \det \nabla \eta_l dx \to \int_Q \det \nabla \eta dx$.
 
 Moreover, the $C^{1,\alpha}$ bounds imply uniform convergence of all minors of $\nabla \eta_l$. Thus since $\partial Q$ is of finite $n-1$ dimensional measure, so will be $\partial \eta(Q)$ and $\partial \eta_l(Q)$. But then $\abs{\eta_l(Q)} \to \abs{\eta(Q)}$ by the uniform convergence of $\eta_l$.
\end{proof}

\subsubsection*{Injectivity and boundary regularity of the solid}

Further, we discuss the injectivity of deformations in $\mathcal{E}$ up-to-the boundary. In fact, any $\eta \in \mathcal{E}$ is injective on $Q$ but not necessarily on $\bar{Q}$, so collisions are in principle possible. Nonetheless, we ca exclude them for short times as shown via the following two lemmas as well as Corollary \ref{cor:QSshortTimeNoCollision}.
\begin{lemma}[Local injectivity of the boundary]
 For any $E_0 < \infty$ there exists a $\delta_0 > 0$ such that all $\eta \in \mathcal{E}$ with $E(\eta) < E_0$ are locally injective with radius $\delta_0$ at the boundary, i.e.
 \[ \eta(x_0) \neq \eta(x_1) \text{ for all } x_0,x_1 \in \partial Q, \abs{x_0-x_1} < \delta_0.\]
\end{lemma}

\begin{proof}
 Assume that there are two points $x_0,x_1 \in \partial Q$ such that $\eta(x_0) = \eta(x_1)$. Now using embedding theorems and \autoref{ass:energy}, S2 and S4, $E(\eta) < E_0$ implies that $D\eta$ is uniformly continuous and there exists a uniform lower bound on $\det D \eta$. This also results in a uniform continuity of $(D\eta)^{-1} = \frac{\cof D\eta}{\det D\eta}$.

 Let $A$ denote the contact plane spanned by $D\eta(x_0) v$, $v$ tangential to $\partial Q$ in $x_0$ and denote the projection onto this plane using $P_A$. Then for any $x\in \partial Q$, the linear map $\phi_x: v\mapsto P_A D\eta(x) v$ maps the tangential space $T_x \partial Q$ to $A$. In particular for $x_0$, we have $P_A D\eta(x_0) = D\eta(x_0)$ and thus $\phi_{x_0}$ is an isomorphism with determinant bounded from below. Now from the regularity of $\partial Q$ ($T_x\partial Q$ does not change fast depending on $x$) and uniform continuity of $D\eta$, we get that the same has to hold in a $\delta_0$ neighborhood of $x_0$. Further, if we orient the tangential spaces through the exterior normal and $A$ through the orientation inherited from $D\eta(x_0)$, then $\phi_x$ has to be orientation preserving in this neighborhood. So $x_1$ cannot lie in this neighborhood, as a simple geometrical argument shows that the orientation imparted on $A$ through $D\eta(x_1)$ is opposite to the one chosen though $D\eta(x_0)$.
\end{proof}

\begin{remark} %
 The preceeding proof is much easier to formulate in the case $n=2$ as one can deal with tangential vectors directly: Consider the positively oriented unit tangentials $\tau_x$ at $x\in \partial Q$. Then $D\eta(x_0) \tau_{x_0}$ and $D\eta(x_1) \tau_{x_1}$ point in opposite directions and their length is bounded from below. But if $x_0$ and $x_1$ are close, then so are the $\tau_{x_i}$ and the $D\eta(x_i)$, which leads to a contradiction.
\end{remark}

\begin{proposition}[Short time global injectivity preservation] \label{prop:shortInjectivity}
 Fix $E_0 < \infty$ and $\varepsilon_0 > 0$ and let $\delta_0$ be given by the previous lemma. Then there exists a $\gamma_0 >0$ such that for all $\eta_0 \in \mathcal{E}$ with $E(\eta_0)< E_0$ and
 \begin{align} \label{eq:boundaryDistanceCond}
   \abs{\eta_0(x_0) - \eta_0(x_1)} > \varepsilon_0 \text{ for all } x_0,x_1 \in \partial Q, \abs{x_0-x_1} \geq \delta_0
 \end{align}
 we have that for all $\eta \in \mathcal{E}$ with $E(\eta) < E_0$ and $\norm[L^2]{\eta_0-\eta} < \gamma_0$ it holds that
 \[ \abs{\eta(x_0) - \eta(x_1)} > \frac{\varepsilon_0}{2} \text{ for all } x_0,x_1 \in \partial Q, \abs{x_0-x_1} \geq \delta_0\]
\end{proposition}

\begin{proof}
 Let $\eta_0$ be as prescribed and pick $\eta \in \mathcal{E}$, $E(\eta) < E_0$ with $\abs{\eta(x_0) - \eta(x_1)} \leq \frac{\varepsilon_0}{2}$ for two points $x_0,x_1$ with $\abs{x_0-x_1} \geq \delta_0$. But then 
 \begin{align*}
  \abs{\eta_0(x_0) - \eta(x_0)} + \abs{\eta_0(x_1) -\eta(x_1)} \geq \abs{\eta_0(x_0)-\eta_0(x_1)} - \abs{\eta(x_0)-\eta(x_1)} > \frac{\varepsilon_0}{2}
 \end{align*}
 So, without loss of generality, we can assume that $\abs{\eta_0(x_0) - \eta(x_0)} \geq \frac{\varepsilon_0}{4}$. But then since $\eta_0$ and $\eta$ are uniformly continuous independent of the choice of $\eta$, there exists an $r>0$ such that $\abs{\eta_0(x)-\eta(x)} \geq \frac{\varepsilon_0}{8}$ for all $x\in B_r(x_0) \cap Q$. Thus
 \[
  \norm[L^2]{\eta_0-\eta} \geq \sqrt{ \left(\frac{\varepsilon_0}{8}\right)^2 \abs{B_r(x_0) \cap Q}} =: \gamma_0 > 0 \qedhere
 \]
\end{proof}
Since we are concerned with variable in time domains for the fluid flow, we recall here the quantification of uniform regular domains. Later we will use several analytical results which will be used uniformly with respect to these quantifications.
\begin{definition}
\label{def:cka}
For $k\in\N$ and $\alpha\in[0,1]$. We call $\Omega\subset\R^n$ a $C^{k,\alpha}-$domain with {\em characteristics} $L,r$, if for all $x\in \partial \Omega$ there is a  $C^{k,\alpha}$-difeomorphismn $\phi_x:B_1(0)\to B_r(x))$, such that $\phi_x:B^-_1(0)\to  B_r(x)\cap \Omega$, $\phi_x:B^+_1(0)\to  B_r(x)\cap \Omega^c$ and $\phi_x(0)=x$. We require that it can be written as a graph over a direction $e_x\in \mathcal{S}^{n-1}$. This means that for $(z',z_n)\in B_1(0)$ we may write $\phi_x(z)=\phi_x((z',0))+ re_xz_n$. And that it satisfies the bound:
\[
\norm[C^{\alpha,k}(B_1(0))]{\phi_x}+\norm[C^{\alpha,k}(B_r(x))]{\phi_x^{-1}}\leq L.
\]  
\end{definition}

Collecting the regularity that comes from the energy bounds lead to an important (locally) uniform estimate on the $C^{1,\alpha}$ regularity of the fluid domains.

\begin{corollary}[Uniform $C^{1,\alpha}$ domains] \label{cor:c1aBoundary}
Fix $E_0 < \infty$, $\eta_0 \in \mathcal{E}$ with $E(\eta_0)< E_0$ and satisfying \eqref{eq:boundaryDistanceCond} for some $\varepsilon_0 > 0$. Then for all $\eta \in \mathcal{E}$ with $E(\eta) <E_0$ and $\norm[L^2]{\eta_0-\eta} < \gamma_0$ for $\gamma_0$ from \autoref{prop:shortInjectivity} we have that $\Omega_{\eta} := \Omega \setminus \eta(Q)$ is a $C^{1,\alpha}$-domain with characteristics $L,r$ depending only on $E_0, \eta_0$ and $\varepsilon_
0$.
\end{corollary}

\subsubsection*{Global velocity and a global Korn inequality}

A particular useful tool when dealing with fluid structure interaction in the bulk is the global Eulerian velocity field, which is defined on the unchanging domain $\Omega$. In particular this allows us to circumvent the problem of talking about convergence on a changing domain.

\begin{definition}[The global velocity field] \label{def:globalVelocity}
 Let $\eta \in \mathcal{E}$ be a given deformation with $E(\eta) = 0$. Let $v\in W^{1,2}(\Omega;\R^n)$ be a divergence free fluid velocity and $b \in W^{1,2}(Q;\R^n)$ a solid velocity. Furthermore assume the usual coupling condition, $v \circ \eta = b$ on $\partial Q \setminus P$. Then the corresponding global velocity $u \in W^{1,2}_0(\Omega;\R^n)$ is defined by
 \begin{align*}
  u(y) := \begin{cases}
           v(y) &\text{ if } y\in \Omega_\eta := \Omega \setminus \eta(Q) \\
           b \circ \eta^{-1}(y) &\text{ if } y \in \eta(Q).
          \end{cases}
 \end{align*}
 \end{definition}
 Note that this definition does not involve a reference time scale directly. The solid velocity $b$ is equally allowed to be a time derivative $b:=\partial_t \eta$ or a discrete derivative $b:=\frac{\eta_{k+1}-\eta_k}{\tau}$. Furthermore this definition is invertible. Given $u$ and knowing $\eta$, both $v$ and $b$ can be reconstructed and those reconstructed velocities will satisfy the coupling condition.

The only bounds on the velocities that will be available to us are in form of a bounded dissipation. As this dissipation is given in form of a  symmetrised derivative, we will need to use a Korn-inequality, the constants of which generally depend on the domain. However, another benefit of the global velocity and its constant domain is that the Korn-inequalities for solid and fluid can be merged into a global Korn-inequality.

\begin{lemma}[Global Korn inequality]
\label{lem:globalKorn}
Fix $E_0 > 0$. Then there exists $c_{gK} = c_{gK}(E_0) > 0$ such that for any $\eta \in \mathcal{E}$ with $E(\eta) < E_0$ and any $b \in W^{1,2}(Q;\R^n)$ and $u \in W^{1,2}(\Omega;\R^n)$ satisfying the coupling condition 
\[u \circ \eta = b \text{ in Q}\]
and with boundary data $u|_{\partial \Omega} = 0$ and $b|_{P} = 0$, we have
\[
c_{gK} \norm[W^{1,2}(\Omega)]{u}\leq \frac{\nu}{2} \norm[\Omega_{\eta}]{\nablasym u} + R(\eta,b).
\]
where we define $\Omega_\eta = \Omega \setminus \eta(Q)$. In particular this implies a bound on $\norm[W^{1,2}(\Omega_\eta)]{v} \leq \norm[W^{1,2}(\Omega)]{u}$.
\end{lemma}
\begin{proof} 
On the reference domain $Q$ we have per chain rule $\nabla b = \nabla (u\circ \eta) = (\nabla u) \circ \eta \cdot \nabla \eta$. Using this, we can estimate in analogy to \autoref{prop:W2qIsomorph}, as $\eta$ is a diffeomorphism:
\begin{align*}
 \int_{\Omega \setminus \Omega_\eta} \abs{\nabla u}^2 dy &= \int_Q \abs{(\nabla u) \circ \eta}^2 \det \nabla \eta dx = \int_Q \abs{(\nabla b) \cdot (\nabla \eta)^{-1}}^2 \det \nabla \eta dx \\
 \leq \int_Q \abs{\nabla b}^2 \frac{\abs{\cof \nabla \eta}^2}{\det \nabla \eta} dx &\leq \frac{\norm[C^1]{\eta}^{2n-2}}{\epsilon_0} \int_Q \abs{\nabla b}^2 dx \leq \frac{\norm[C^1]{\eta}^{2n-2}}{\epsilon_0} c_K R(\eta,b)
\end{align*}
where $\epsilon_0 > 0$ is the uniform lower bound on $\det \nabla \eta$ as given in Assumption S2, $\norm[C^1]{\eta}$ is uniformly bounded by embeddings and we use the Korn-type inequality from Assumption R3.

But now we can apply Korn's inequality to the fixed domain $\Omega$ to get a constant $c_\Omega$, for which
\begin{align*}
 c_\Omega \norm[W^{1,2}(\Omega)]{u}^2 &\leq \norm[\Omega]{\nablasym u}^2 = \norm[\Omega_\eta]{\nablasym u}^2 + \norm[\Omega \setminus \Omega_\eta]{\nablasym u}^2 
  \\&\leq \norm[\Omega_\eta]{\nablasym u}^2 + \frac{\norm[C^1]{\eta}^{2n-2}}{\epsilon_0} c_K R(\eta,b).
\end{align*}
Collecting all the constants then proves the lemma.
\end{proof}

\subsection{Proof of \autoref{thm:QSexistence}} 
\label{subsec:QSproof}

As mentioned before, we will show \autoref{thm:QSexistence} in several steps using a time-discretisation in the form of a minimizing movements iteration. 

\subsubsection*{Step 1: Existence of the discrete approximation}

For this we will fix a time-step size $\tau$. Setting $\eta_0^{(\tau)}:=\eta_0$ and assuming $\eta^{(\tau)}_k \in \mathcal{E}$ given we define $(\eta^{(\tau)}_{k+1},v^{(\tau)}_{k+1})$ as solutions to the following problem
\begin{align} \label{eq:QSdiscreteProblem} & \text{Minimize} &  &E(\eta) + \tau R\left(\eta_k^{(\tau)},\frac{\eta-\eta_k^{(\tau)}}{\tau}\right) + \tau \frac{\nu}{2} \norm[\Omega_k^{(\tau)}]{\nablasym v}^2  \\ \nonumber & & & \phantom{E(\eta)} -\rho_s \tau\inner[Q]{f\circ\eta_k^{(\tau)}}{\frac{\eta-\eta_k^{(\tau)}}{\tau}} - \rho_f \tau\inner[\Omega_k^{(\tau)}]{f}{v} \\
& \text{subject to } & &\eta \in \mathcal{E}, v \in W^{1,2}(\Omega_k^{(\tau)};\R^n) \text{ with } \diver v = 0, v|_{\partial \Omega} = 0 \nonumber \\ & & &\text{ and }\frac{\eta-\eta_k^{(\tau)}}{\tau} = v \circ \eta_k^{(\tau)} \text{ in } M.\footnotemark \nonumber
\end{align}
\footnotetext{In this part of the proof we will generally assume that the solid is free of collisions, i.e.\ $\eta_k \notin \partial \mathcal{E}$ or in other words that $\eta_k|_{M}$ is injective and does not map to $\partial \Omega$. We will show in \autoref{cor:QSshortTimeNoCollision} that for small enough $T$ this will always be the case. In principle though, it is possible to extend the coupling condition to all situations by using the global velocity field $u:\Omega \to \R^n$ instead.}

We then repeat this process until we reach $k\tau > T$. 

We will now show that this problem has a (not necessarily unique) solution and that the sought minimizer satisfies an Euler-Lagrange equation which already is a discrete approximation of our problem.

\begin{remark}
In \eqref{eq:QSdiscreteProblem} we minimize over the sum of the energy and the dissipation needed to reach the current step from the last one. It should be noted that the scheme is thus implicit in the deformation (what we would consider the single \emph{state variable} of the system) at which the energy is evaluated, while it is in a sense explicit in the dissipation. Here we understand both the solid dissipation and the Stokes potential as dissipative terms and both the change in deformation as well as the velocity as distinct, secondary \emph{rate variables} which can be recovered from knowing successive states.%

The implicit and explicit use of state can be motivated from the point of view of the Euler-Lagrange equations; as they are exactly the Fr\'echet derivatives with respect to the rate variables that appear there. Explicit-implicit schemes are commonly used in fluid-structure interactions~(see e.g.\ \cite{muhaExistenceWeakSolution2013}). Moreover, it is the common way to produce solutions in solid mechanics if the dissipation depends on the state variables \cite{kruzik2019mathematical}. 

For us however there is another important point to this, as the correct coupling of the fluid and solid variables is crucial in fluid-structure interaction. Within the proposed variational approach, it is important that we impose a coupling condition as equality of approximate velocities also in an \emph{implicit-explicit} fashion; i.e.\ we keep the geometry explicit while the rates are implicit. An equality of tractions then needs not be imposed but follows automatically from the variational approach. 
\end{remark}

\begin{proposition}[Existence of solutions to \eqref{eq:QSdiscreteProblem}]\label{prop:QSdiscreteELequation} 
Assume that $\eta^{(\tau)}_k \in \mathcal{E}$ with $E(\eta^{(\tau)}_k)<\infty$. Then the iterative problem \eqref{eq:QSdiscreteProblem} has a minimizer, i.e.\ $\eta^{(\tau)}_{k+1}$ and $v^{(\tau)}_{k+1}$ are defined. Furthermore if $\eta^{(\tau)}_{k+1} \notin \partial \mathcal{E}$ (i.e. $\eta^{(\tau)}_{k+1}$ is injective on $\overline{Q}$) the minimizers obey the following Euler-Lagrange equation:
 \begin{align*}
&\phantom{{}={}}\inner{DE(\eta^{(\tau)}_{k+1})}{\phi} + \inner{D_2R\left(\eta^{(\tau)}_{k},\frac{\eta^{(\tau)}_{k+1}-\eta^{(\tau)}_{k}}{\tau}\right)}{\phi}+ \nu\inner[\Omega_{k}^{(\tau)}]{\nablasym v^{(\tau)}_{k+1}}{\nabla \xi} \\
&= \rho_f\inner[\Omega_{k}^{(\tau)}]{f}{\xi} + \rho_s\inner[Q]{f\circ \eta^{(\tau)}_k}{\phi}.
\end{align*}
 for any $\phi \in W^{2,q}(Q;\R^n)$, $\phi|_P = 0$ and $\xi \in W^{1,2}(\Omega_k^{(\tau)};\R^n)$, $\diver \xi = 0$, $\xi|_{\partial \Omega} = 0$ such that $\xi \circ \eta = \phi$ in $\partial M$.
\end{proposition}

\begin{proof} First we investigate existence using the direct method. The class of admissible functions is non-empty, since $(\eta_k^{(\tau)},0)$ is a possible competitor with finite energy. Next we show that the functional is bounded from below. As energy and dissipation have lower bounds per assumption, the only problematic terms are those involving the force $f$. For those we note that per the weighted Young's inequality and using Assumption R3 it holds that
\begin{align*}
 \abs{\inner[Q]{f\circ \eta_k^{(\tau)}}{\frac{\eta-\eta_k^{(\tau)}}{\tau}}} &\leq \frac{\delta}{2} \norm[Q]{\frac{\eta-\eta_k^{(\tau)}}{\tau}}^2 + \frac{1}{2\delta} \norm[Q]{f \circ \eta_k^{(\tau)}}^2 \\
 &\leq \frac{\delta}{2 c_K} R\left(\eta_k^{(\tau)},\frac{\eta-\eta_k^{(\tau)}}{\tau}\right) + \frac{1}{2\delta} \norm[Q]{f \circ \eta_k^{(\tau)}}^2
\end{align*}
and equally, using \autoref{lem:globalKorn}
\begin{align*}
 \abs{\inner[\Omega_k^{(\tau)}]{f}{v}} \leq \frac{\delta}{2} \norm[\Omega_k^{(\tau)}]{v}^2 + \frac{1}{2\delta} \norm[\Omega_k^{(\tau)}]{f}^2 \leq \frac{\delta}{2c_{gK}} \left( \norm[\Omega_k^{(\tau)}]{\nablasym v}^2 +  R\left(\eta_k^{(\tau)},\frac{\eta-\eta_k^{(\tau)}}{\tau}\right) \right)+ \frac{1}{2\delta} \norm[\Omega_k^{(\tau)}]{f}^2.
\end{align*}
Now if we choose $\delta$ small enough, e.g. $\delta := \frac{\min(c_K,c_{gK})}{2}$, all $v$ and $\eta$-dependent terms can be absorbed to get the lower bound
\begin{align} \label{eq:QSlowerBound}
&\phantom{{}={}} E(\eta) +\tau R\left(\eta_k,\frac{\eta-\eta_k^{(\tau)}}{\tau}\right) + \tau\frac{\nu}{2} \norm[\Omega_k^{(\tau)}]{\nablasym v}^2 - \rho_f \tau\inner[\Omega_k^{(\tau)}]{f}{v} -\rho_s \tau\inner[Q]{f\circ\eta_k^{(\tau)}}{\frac{\eta-\eta_k^{(\tau)}}{\tau}}\\ \nonumber
 &\geq E(\eta) + \frac{\tau}{2} R\left(\eta_k^{(\tau)},\frac{\eta-\eta_k^{(\tau)}}{\tau}\right) + \tau\frac{\nu}{4} \norm[\Omega_k^{(\tau)}]{\nablasym v}^2 - \tau\frac{\max(\rho_f,\rho_s)}{2\delta} \left(\norm[Q]{f\circ \eta_k^{(\tau)}}^2 + \norm[\Omega_k^{(\tau)}]{f}^2 \right)\\ \nonumber
 &\geq E_{\min} - \tau\frac{\max(\rho_f,\rho_s)}{2\delta} \left(\norm[Q]{f\circ \eta_k^{(\tau)}}^2 + \norm[\Omega_k^{(\tau)}]{f}^2 \right)
\end{align}

 Thus a minimizing sequence $\tilde{\eta}_l, \tilde{v}_l$ exists and along that sequence, energy and dissipation are bounded. So by coercivity of the energy we know that $\tilde{\eta_l}$ is bounded in $W^{2,q}(Q;\Omega)$ and using the usual compactness theorems we may extract a subsequence (not relabeled) and a limit $\eta$ for which
 \begin{align*}
  \tilde{\eta}_l &\rightharpoonup \eta &&\text{ in } W^{2,q}(Q;\Omega) \\
  \tilde{\eta}_l &\to \eta &&\text{ in } C^{1,\alpha^-}(Q;\Omega) \text{ for  $0< \alpha^-< \alpha := 1-\frac{n}{q}$.}
 \end{align*}
 By \autoref{lem:calEclosed} we know that $\eta \in \mathcal{E}$. We also know that $E$ and $R$ are lower semicontinuous with respect to the above convergence by Assumptions S3 and R1 respectively.
 
 Next we pass to the limit with the fluid velocity. With no loss of generality, we may assume $\tilde{v}_l$ to be the minimizer of the functional in \eqref{eq:QSdiscreteProblem} holding the deformation $\tilde{\eta}_l$ fixed. As the functional in \eqref{eq:QSdiscreteProblem} is convex with respect to the velocity, minimizing is \emph{equivalent} to solving the appropriate Euler-Lagrange equation, in other words, it is equivalent to finding a weak solution to the following classical Stokes boundary value problem: %
 \begin{align*}
  \begin{cases}
   -\nu \Delta \tilde{v}_l + \nabla p = \rho_f f & \text{ in } \Omega_k^{(\tau)} \\
   \diver \tilde{v}_l = 0 & \text{ in } \Omega_k^{(\tau)}\\
   \tilde{v}_l =  g_l := \frac{(\tilde{\eta}_l - \eta_k^{(\tau)})\circ (\eta_{k}^{(\tau)})^{-1} }{\tau} & \text{ in } \partial \Omega_k^{(\tau)} \cap \partial \eta_k^{(\tau)}(Q) \\
   \tilde{v}_l = 0 &\text{ in } \partial \Omega
  \end{cases}
 \end{align*}

 Now since $\eta_k$ is a fixed diffeomorphism, and $\tilde{\eta}_l$ converges uniformly the boundary data $g_l$ in this problem converges uniformly to $g := \frac{(\eta-\eta_k^{(\tau)}) \circ (\eta_k^{(\tau)})^{-1}}{\tau}$ as well. Furthermore, the solution operator $L^2(\partial \Omega_k^{(\tau)};\R^n) \to W^{1,2}(\Omega;\R^n)$ associated with this boundary value problem is continuous, which implies the existence of a limit $v \in W^{1,2}(\Omega_k^{(\tau)};\R^n)$ with $\tilde{v}_l \to v$ in $W^{1,2}(\Omega_k^{(\tau)};\R^n)$. Then per construction $\eta,v$ satisfy the compatibility condition and since $\norm[\Omega_k^{(\tau)}]{\nablasym v}$ is lower semicontinuous and all terms involving $f$ are continous, the pair $\eta,v$ is indeed a minimizer to the problem.
 
 Next let us derive the Euler-Lagrange equation. Let $(\eta_{k+1}^{(\tau)},v_{k+1}^{(\tau)})$ be a minimizer and $\phi \in C^\infty(Q;\R^n)$ as well as $\xi \in W^{1,2}(\Omega_{k}^{(\tau)};\R^n)$. We require the perturbation $(\eta_{k+1}^{(\tau)}+\varepsilon \phi,v_{k+1}^{(\tau)}+\varepsilon \xi/\tau)$ to also be admissible\footnote{The different scaling of $\phi$ and $\xi/\tau$ with respect to $\tau$ used here allows us to remove most occurrences of $\tau$ in the Euler-Lagrange equation. This does not matter as long as $\tau$ is fixed, but it turns out to be the correct scaling when we take the limit $\tau\to 0$.} for all small enough $\varepsilon$. From this we immediately get the conditions $\diver \xi = 0$, $\xi|_{\partial\Omega} = 0$, $\phi_P = 0$ and for the coupling we require
 \[\left(v_{k+1}^{(\tau)}+\varepsilon\frac{\xi}{\tau}\right) \circ \eta_k^{(\tau)} = \frac{\eta_{k+1}^{(\tau)}+\varepsilon\phi-\eta_{k}^{(\tau)}}{\tau}\]
 on $M$ which reduces to $\xi \circ \eta_k^{(\tau)} = \phi$. 

 Now since we assume $\eta^{(\tau)}_{k+1} \notin \partial \mathcal{E}$, for small enough $\varepsilon$, we have $\eta_{k+1}^{(\tau)}+\varepsilon \phi \in \mathcal{E}$.
 Thus we are allowed to take the first variation with respect to $(\phi,\xi/\tau)$ which immediately results in the weak formulation.
\end{proof}

Now let us give some a-priori estimates on the solutions derived in this way. This is possible because of the fact that we have a minimizer and not just a weak solution to the equation. 

\begin{lemma}[Parabolic a-priori estimates]\label{lem:QSdiscreteAPrioriEstimates} 
We have
 \begin{align*}
 &\phantom{{}={}} E(\eta_{k+1}^{(\tau)}) + \tau R\left(\eta_k^{(\tau)},\frac{\eta_{k+1}^{(\tau)}-\eta_k^{(\tau)}}{\tau}\right) + \tau \nu\norm[\Omega_k^{(\tau)}]{\nablasym v_{k+1}^{(\tau)}}^2 
 \\
 &\leq E(\eta_{k}^{(\tau)}) + \tau\rho_f \inner[\Omega_k]{f}{v_{k+1}^{(\tau)}} + \tau\rho_s  \inner[Q]{f\circ\eta_k^{(\tau)}}{\frac{\eta_{k+1}^{(\tau)}-\eta_k^{(\tau)}}{\tau}}
 \end{align*}

Furthermore take a number $E_0$ such that $E_0 > E(\eta_0)$. Then there exists a time $T > 0$ depending only on $E_0$ and the difference $E_0-E(\eta_0)$ as well as $\norm[L^\infty(\Omega)]{f}$, such that for all $\tau>0$, and all $N \in \N$ with $N\tau \leq T$ we have
 \begin{align*}
 E(\eta_{N}^{(\tau)}) + \frac{\tau}{2}\sum_{k=1}^{N} \left[ \frac{\nu}{2}\norm[\Omega_{k-1}^{(\tau)}]{\nablasym v_{k}^{(\tau)}}^2 +R\left(\eta_{k-1}^{(\tau)},\frac{\eta_{k}^{(\tau)}-\eta_{k-1}^{(\tau)}}{\tau}\right) \right] \leq  E_0
\end{align*}
\end{lemma}

\begin{proof}[Proof of \autoref{lem:QSdiscreteAPrioriEstimates}]
 As before, for fixed $k$, we may compare the value of the cost functional in \eqref{eq:QSdiscreteProblem}  at the minimizer with its value for the pair $(\eta_{k}^{(\tau)},0)$. As $R(\eta_k^{(\tau)},0) = 0$ and the terms involving $v$ vanish for $v=0$, the comparison yields the first line.
 
 Now we proceed by induction over $N$. Assume that $E(\eta_{N-1}^{(\tau)}) \leq E_0$ and let $c_{gK}$ be the Korn-constant corresponding to $E_0$ from \autoref{lem:globalKorn}. Using \eqref{eq:QSlowerBound} again, we end up with
 \begin{align} \label{eq:QSsingleStepLower}
 &\phantom{{}={}} E(\eta_{k+1}^{(\tau)}) + \frac{\tau}{2} R\left(\eta_k^{(\tau)},\frac{\eta_{k+1}^{(\tau)}-\eta_k^{(\tau)}}{\tau}\right) + \tau\frac{\nu}{4} \norm[\Omega_k^{(\tau)}]{\nablasym v^{(\tau)}_k}^2 \\
  & \nonumber \leq E(\eta_k^{(\tau)}) + \tau \frac{1}{2\delta} \left(\norm[Q]{f \circ \eta_k^{(\tau)}}^2 + \norm[\Omega_k^{(\tau)}]{f}^2\right)  \leq E(\eta_k^{(\tau)}) + \tau \frac{\max(\rho_f,\rho_s)}{2\delta} \norm[L^\infty(\Omega)]{f }^2
 \end{align}
 where for all $k\in \{1,...,N\}$ the $\delta$ does only depend on $c_{gK}$ and $c_K$ and thus only on $E_0$.

Hence we may sum this estimate over $k$, yielding
 \begin{align*}
  &\phantom{{}={}} E(\eta_{N}^{(\tau)}) + \frac{\tau}{2}\sum_{l=1}^{N} \left[ \frac{\nu}{2}\norm[\Omega_{k-1}^{(\tau)}]{\nablasym v_{k}^{(\tau)}}^2 +R\left(\eta_{k-1}^{(\tau)},\frac{\eta_{k}^{(\tau)}-\eta_{k-1}^{(\tau)}}{\tau}\right) \right] \\
  \nonumber &\leq E(\eta_0) + N \tau \frac{\max(\rho_f,\rho_s)}{2\delta} \norm[L^\infty(\Omega)]{f}^2 \leq E_0
 \end{align*}
 assuming that $N\tau \leq T$ for $T>0$ given by $\frac{T \max(\rho_f,\rho_s)}{2\delta} \norm[L^\infty(\Omega)]{f}^2 = E_0 - E(\eta_0)$. But then in particular $E(\eta_{N}^{(\tau)}) \leq E_0$ and we can continue the induction until $N\tau$ reaches $T$.
\end{proof}

\begin{remark} 
Clearly, the maximal length of the time-interval on which the a-priori estimates are true depends on the choice of $E_0$ and could be, thus, optimized. However, we do not enter this investigation here, since, later we may prolong the solution to the maximal existence time.

Let us also mention that a slightly better single-step estimate could have been gotten by comparing with $(\eta_{k},\tilde{v})$, where $\tilde{v}$ is the minimizer of $\frac{\nu}{2}\norm{\nablasym v}^2+\inner{f}{v}$ under $\diver v=0$ and $0$-boundary conditions. 
\end{remark}

\subsubsection*{Step 2: Time-continuous approximations and their properties}

Now as a next step, we use these iterative solutions to construct approximations of the continuous problem. At this point, we will completely switch over to the global velocity $u$. We will also approximate the deformation $\eta$ in two different ways, a piecewise constant approximation, which we will need to keep track of the fluid-domain, and a piecewise affine approximation, which will give us the correct time derivative $\partial_t \eta$. To be more precise, we define:

\begin{definition}[Discrete parabolic approximation]
 For some $E_0 > 0$ fix $T>0$ as given by \autoref{lem:QSdiscreteAPrioriEstimates}. We now define the piecewise constant $\tau$-approximation as
 \begin{align*}
  \eta^{(\tau)}(t,x) &:= \eta_k^{(\tau)}(x) &\text{ for }& t\in [\tau k,\tau (k+1)), x\in Q \\
  u^{(\tau)}(t,y) &:= v_k^{(\tau)}(y) &\text{ for }& t\in [\tau k,\tau (k+1)), y\in \Omega_k^{(\tau)}\\
  u^{(\tau)}(t,y) &:= \frac{(\eta_{k+1}^{(\tau)} - \eta_{k}^{(\tau)}) \circ (\eta_k^{(\tau)})^{-1}(y)}{\tau}&\text{ for }&t\in [\tau k,\tau (k+1)), y\in \eta_{k}(\overline{Q})
  \\
  \Omega^{(\tau)}(t)&:=\Omega_k^{(\tau)} &\text{ for }&t\in [\tau k,\tau (k+1))
  \end{align*}
  where $(\eta^{(\tau)}_k,v^{(\tau)}_k)$ is the iterative solution for timestep $\tau$. We also define the affine approximation for $\eta$ as
  \begin{align*}
  \tilde{\eta}^{(\tau)}(t,.) &:= ((k+1)-t/\tau)\eta_k^{(\tau)} - (t/\tau-k)\eta_{k+1}^{(\tau)} &\text{ for }& t\in [\tau k,\tau (k+1)), x\in Q.
 \end{align*}
\end{definition}
 
 Note that $\tilde{\eta}^{(\tau)}$ is Lipschitz-continuous in time, $\tilde{\eta}^{(\tau)}(k\tau) = \eta^{(\tau)}(k\tau)$ for all $k\in \{0,...,N\}$ and 
 \[\partial_t \tilde{\eta}^{(\tau)}(t) = \frac{1}{\tau} (\eta_{k+1}^{(\tau)} - \eta_k^{(\tau)}) = u^{(\tau)}(t)\circ \eta^{(\tau)}(t)\]
 for all $t \in (\tau k,\tau(k+1))$. Also from this point on, we will only work with the global velocity field $u^{(\tau)}$ as given in \autoref{def:globalVelocity}

\begin{lemma}[Basic a-priori estimates]
 For any fixed upper energy bound $E_0$ and the resulting $T$ from \autoref{lem:QSdiscreteAPrioriEstimates} there exists a constant $C$ independent of $\tau$ such that
 \begin{align*}
 E(\eta^{(\tau)}(t))  + \int_0^t R(\eta^{(\tau)}(t),\partial_t \tilde{\eta}^{(\tau)}(t)) + \frac{\nu}{2} \norm[\Omega^{(\tau)}(t)]{\nablasym u^{(\tau)}}^2 dt \leq E_0.
\end{align*}
for all $t \in[0,T]$ as well as
\begin{align*}
\sup_{t\in[0,T]}\norm[W^{2,q}(Q)]{\eta^{(\tau)}(t)} \leq C,\quad 
\int_0^T \norm[W^{1,2}(Q)]{\partial_t \tilde{\eta}^{(\tau)}}^2 dt  \leq C, \text{ and }
\int_0^T \norm[W^{1,2}(\Omega)]{u^{(\tau)}}^2 dt \leq C.
\end{align*}
\end{lemma} 

\begin{proof}
The first statement is a direct translation of \autoref{lem:QSdiscreteAPrioriEstimates} while the latter ones follow from this. In particular, since $E(\eta^{(\tau)}(t)) < E_0$ on any of its constant intervals and thus on all of $[0,T]$ its supremum is bounded. Similarly the two integral inequalities follow from the boundedness of the dissipation combined  with the Korn-inequalities R2 and \autoref{lem:globalKorn}.
\end{proof}

 \begin{lemma}[Energy and H\"{o}lder-estimates] \label{lem:QSdiscreteHoelderEstimate}
  For any fixed upper energy bound $E_0$ and the resulting $T$ from \autoref{lem:QSdiscreteAPrioriEstimates}, there exists a constant $C$ independent of $\tau<<1$ such that we have the following estimates:
  \begin{enumerate}
   \item For all $t\in[0,T]$
  \[\norm[W^{1,2}(Q)]{\eta^{(\tau)}(t)-\tilde{\eta}^{(\tau)}(t)} \leq C \sqrt{\tau}\]
   \item $E(\eta^{(\tau)}(t))$ is nearly monotone, i.e.\ for any $t>t_0$, $t,t_0\in[0,T]$ with $t-t_0 \geq \tau$ we have
  \[
   E(\eta^{(\tau)})(t) - E(\eta^{(\tau)})(t_0) \leq C(t-t_0)
  \]
  \item $\eta^{(\tau)}(t)$ is nearly H\"older-continuous in $W^{1,2}(Q)$, i.e.\ for any $t>t_0$, $t,t_0\in[0,T]$ with $t-t_0 > \tau$ we have
  \[\norm[W^{1,2}(Q)]{\eta^{(\tau)}(t)-\eta^{(\tau)}(t_0)} \leq C\sqrt{t-t_0} \]
  \end{enumerate}
 \end{lemma}
 
 \begin{proof}
 Consider the lower bound on a single step given in \eqref{eq:QSsingleStepLower}. Singling out the dissipation of the solid material, and dropping some terms with compatible sign, we get using the Korn's inequality R3
 \begin{align*}
  c_K \frac{1}{\tau} \norm[W^{1,2}(\Omega)]{\eta^{(\tau)}_{k+1}-\eta^{(\tau)}_k}^2 \leq \tau R\left(\frac{\eta^{(\tau)}_{k+1}-\eta^{(\tau)}_k}{\tau}\right) \leq 2\left(E(\eta^{(\tau)}_k) - E(\eta^{(\tau)}_{k+1})  + \tau \frac{\max(\rho_f,\rho_s)}{2\delta} \norm[\infty]{f}^2\right)
 \end{align*}
 Now as the energy is bounded uniformly from above and from below, we can derive that
 \begin{align*}
  \norm[W^{1,2}(Q)]{\eta^{(\tau)}_{k+1}-\eta^{(\tau)}_k} \leq C\sqrt{\tau}
 \end{align*}
 for some constant $C$ depending only on $E_0$ and $f$. In particular, due to the definition of  $\tilde{\eta}^{(\tau)}(t)$ iand $\eta^{(\tau)}(t)$ this implies (1).
 
 Equally, reordering the terms in a different way, we get
 \begin{align*}
  E(\eta^{(\tau)}_{k+1}) - E(\eta^{(\tau)}_k) \leq \tau \frac{\max(\rho_f,\rho_s)}{2\delta} \norm[\infty]{f}^2.
 \end{align*}
 Now fix $T \geq t>t_0\geq 0$ and let $M:= \lfloor \frac{t}{\tau} \rfloor, N:= \lfloor \frac{t_0}{\tau} \rfloor$. Summing up the inequality yields
 \begin{align*}
  E(\eta^{(\tau)}(t)) - E(\eta^{(\tau)}(t_0)) \leq \tau(M-N) \frac{\max(\rho_f,\rho_s)}{2\delta} \norm[\infty]{f}^2
 \end{align*}
 Now either $\tau \leq t-t_0 < 2\tau$, in which case $\tau(M-N) < 2\tau < 2(t-t_0)$ or $t-t_0 \geq 2\tau$ and thus $\tau(M-N) < (t-t_0)+\tau < \frac{3}{2}(t-t_0)$, so this estimate proves (2).\footnote{The lower bound on $t_0-t$ is somewhat arbitrary and is only due to the jumps in the piecewise constant approximation. As we are generally interested in $\tau\to 0$ for fixed $t,t_0$, it will not bother us much.}

 Finally we again use the first estimate and H\"{o}lder's inequality to sum up the distances:
 \begin{align*}
  &\phantom{{}={}} \norm[W^{1,2}(Q)]{\eta^{(\tau)}(t)-\eta^{(\tau)}(t_0)} \leq \sum_{k=N}^{M-1}\norm[W^{1,2}(Q)]{\eta^{(\tau)}_{k+1}-\eta^{(\tau)}_k} \\
  &\leq \sqrt{\sum_{k=N}^{M-1} \tau}\sqrt{ \sum_{k=N}^{M-1} \frac{1}{\tau} \norm[W^{1,2}(Q)]{\eta^{(\tau)}_{k+1}-\eta^{(\tau)}_k}^2}\\
  &\leq \sqrt{\tau(M-N)}\sqrt{ \sum_{k=N}^{M-1} 2\left(E(\eta^{(\tau)}_k) - E(\eta^{(\tau)}_{k+1})  + \tau \frac{\max(\rho_f,\rho_s)}{2\delta} \norm[\infty]{f}^2\right)}\\
  & \leq c\sqrt{t-t_0} \sqrt{ E(\eta^{(\tau)}(t_0))-E(\eta^{(\tau)}(t)) + (t-t_0) \frac{\max(\rho_f,\rho_s)}{2\delta} \norm[\infty]{f}^2}
 \end{align*}
 which proves (3).
 \end{proof}

A direct consequence of the last estimate is that the solid cannot move much in short time. In particular, this implies the following result on injectivity:
 
\begin{corollary}[Short-time collision exclusion] \label{cor:QSshortTimeNoCollision}
 If $\eta_0$ is injective (i.e.\ $\eta_0 \notin \partial\mathcal{E}$) then there exists $T>0$ such that for all $\tau$ small enough and all $t\in [0,T]$, the deformations $\eta^{(\tau)}(t)$ and $\tilde{\eta}^{(\tau)}(t)$ are injective (i.e.\ not in $\partial\mathcal{E}$).
\end{corollary}

\begin{proof}
 If we choose $T$ small enough, then the near H\"{o}lder continuity from \autoref{lem:QSdiscreteHoelderEstimate} implies that $\norm[Q]{\eta_0-\eta_k^{(\tau)}}$ is uniformly small. In particular we can choose it to be smaller than the constant $\gamma_0$ from \autoref{prop:shortInjectivity} which then results in injectivity.
\end{proof}

In the following, we take $T$ small enough, so that both,  the a-priori estimates \autoref{lem:QSdiscreteAPrioriEstimates} hold and we may assume injectivity.

\subsubsection*{Step 3: Existence and regularity of limits}

As a next step, we will derive limiting objects for $\tau \to 0$ for the deformation and the global velocity, as well as their mode of convergence.

 \begin{proposition}[Convergence of the time-discrete scheme]\label{prop:QSconvergence}
  There exists a (not relabeled) subsequence $\tau \to 0$ and a limit 
  \[\eta\in C^{1/2}([0,T];W^{1,2}(Q;\R^n))\cap C_w([0,T];W^{2,q}(Q;\R^n))\cap C^0([0,T];C^{1,\alpha^-}(Q))
  \]
for $\alpha=1-\frac{n}{q}$ and $u\in L^2(0,T;W^{1,2}(\Omega;\R^n))$,
   such that 
  \begin{align*}
  \tilde{\eta}^{(\tau)} &\to \eta \text{ in }C^{(1/2)^{-}}([0,T];W^{1,2}(Q;\R^n)) \\
   \eta^{(\tau)}, \tilde{\eta}^{(\tau)} &\rightharpoonup^* \eta \text{ in } L^\infty([0,T];W^{2,q}(Q;\R^n)) %
   \\
   u^{(\tau)} &\rightharpoonup u \text{ in } L^{2}([0,T];W^{1,2}(\Omega;\R^n)) \\
   \partial_t \tilde{\eta}^{(\tau)} &\rightharpoonup \partial_t\eta \text{ in } L^{2}([0,T];W^{1,2}(Q;\R^n)).
  \end{align*} 
 Furthermore we have
    \[\partial_t \eta = u \circ \eta \text{ in } [0,T] \times Q\]
and that $\eta^{(\tau)}$ converges uniformly to $\eta$ in the following sense: For all $r>0$, there exists a $\delta_r>0$, such that for all $\tau<\delta_r$ and all $\abs{x_1-x_2}+\abs{t_1-t_2}^\frac{1}{\alpha+n}\leq r^{\alpha^-}$ we have
  \[
  \abs{\nabla (\eta^{(\tau)}(t_1,x_1)-\eta(t_2,x_2))}+ \abs{\eta^{(\tau)}(t_1,x_1)-\eta(t_2,x_2)}\leq Cr^{\alpha^-},
  \]
  for all $0<\alpha^-<1-\frac{n}{q}$.
  And 
  \begin{align}
  \label{eq:uniformconv}
   \tilde{\eta}^{(\tau)}\to \eta \in C^0([0,T];C^{1,\alpha^-}(Q))\text{ and }
  \eta^{(\tau)}\to \eta \in L^\infty([0,T];C^{1,\alpha^-}(Q)).
  \end{align}
 \end{proposition}

 \begin{proof}
  We proceed with a weak version of Arzela-Ascoli theorem. Let $\{t_i\}_{i\in\N} \subset [0,T]$ be a countable dense set. By the upper bound on the energy and the coercivity, we have a uniform bound on $\norm[W^{2,q}(Q)]{\eta^{(\tau)}(t)}$ and by the usual diagonal argument we can pick a subsequence of $\tau$ (not relabeled) and limits $\eta(t_i)$ such that $\eta^{(\tau)}(t_i) \rightharpoonup \eta(t_i)$ in $W^{2,q}(Q;\R^n)$ and uniformly strongly in $W^{1,2}(Q;\R^n)$ for all $i \in \N$. Then by the convergence of norms, the H\"{o}lder-continuity from \autoref{lem:QSdiscreteHoelderEstimate} (3) carries over to 
  \[\norm[W^{1,2}(Q)]{\eta(t_i)-\eta(t_j)} \leq C \sqrt{\abs{t_i-t_j}}\quad \forall i,j \in \N.\]
  This means that $\eta$ has a unique extension onto $[0,T]$ in the space $C^{1/2}([0,T];W^{1,2}(Q;\R^n))$.
  
  By the usual compactness arguments %
  \[\tilde{\eta}^{\tau} \to \eta \in C^{1/2^-}([0,T];W^{1,2}(Q;\R^n)).\]
  
  Now pick $t\in [0,T]$ and a new sequence $\{t_i\}_{i\in\N} \subset [0,T]$, $t_i\to t$. Due to the uniform $W^{2,q}(Q;\R^n)$-bounds resulting from the bounded energy, the sequence $\{\eta(t_i)\}_{i \in  \N}$ has a weakly converging subsequence which, by the uniqueness of limits, must converge to $\eta(t)$ weakly in $W^{2,q}(Q;\R^n)$. As the original sequence $\{t_i\}_{i\in\N}$ was arbitrary this means that $\eta$ is weakly $W^{2,q}(Q;\R^n)$-continuous. By the same argument, for any subsequence of $\tau$'s there exists a sub-subsequence such that $\eta^{(\tau)}(t) \rightharpoonup \eta(t)$ in $W^{2,q}(Q;\R^n)$ and thus $\eta^{(\tau)} \rightharpoonup \eta$ in $W^{2,q}(Q;\R^n)$ pointwise. By \autoref{lem:QSdiscreteHoelderEstimate} (1), we know that $\tilde{\eta}^{(\tau)}(t)$ converges to the same limit as $\eta^{(\tau)}(t)$ in a $W^{1,2}(Q;\R^n)$-sense. Since $\tilde{\eta}^{(\tau)}(t)$ satisfies the same $W^{2,q}(Q;\R^n)$ bounds, we can then also prove weak $W^{2,q}(Q;\R^n)$ convergence by the same argument.

   Next we interpolate in order to proof that $\eta\in C^0([0,T]; C^{1,\alpha}(Q;\R^n))$. Actually we show more, namely that $\nabla \eta$ is H\"older continuous in time-space.\footnote{Note that due to the zero boundary values on $P$ the estimates on the continuity of $\eta$ follow directly by the gradient estimates.} For that we take $(s_1,x_1),(s_2,x_2)\in [0,T]\times Q$ with $B_r\ni x_1,x_2$ a ball with radius $\abs{x_1-x_2}\leq r$ and $\abs{s_1-s_2}\leq r^{2\alpha +n}$.
  
 \begin{align}
 \label{eq:etacont}
 \begin{aligned}
 &\phantom{{}={}} \abs{\nabla \eta(s_1,x_1)-\nabla \eta(s_2,x_2)}
 \\
 &\leq \abs{\nabla \eta(s_1,x_1)-\fint_{B_r}\nabla \eta(s_1)\,dx}+\abs{\fint_{B_r}\nabla (\eta(s_1)-\eta(s_2))\,dx}+\abs{\nabla \eta(s_2,x_2)-\fint_{B_r}\nabla \eta(s_2)\,dx}
 \\ 
 &\leq Cr^{\alpha} +  \abs{s_2-s_2}\abs{\fint_{s_1}^{s_2} \fint_{B_r} \partial_t \nabla \eta \,dx \,ds}  \leq Cr^{\alpha} +  \abs{s_2-s_2}\left(\fint_{s_1}^{s_2} \fint_{B_r} \abs{\partial_t \nabla \eta} \,dx \,ds\right)^{1/2} \\
 &\leq Cr^{\alpha} +  \frac{\abs{s_2-s_2}^{1/2}}{r^{n/2}} \left(\int_{s_1}^{s_2} \int_{B_r} \abs{\partial_t \nabla \eta}^2 \,dx \,ds \right)^{1/2} 
\leq  Cr^{\alpha}+C\abs{\frac{s_1-s_2}{r^n}}^\frac12\leq Cr^{\alpha}.
\end{aligned}
 \end{align} 
  
   Now let $0< \alpha^- < \alpha = 1- \frac{n}{q}$. 
   For $r>0$ we may choose a finite subset $\{t_i\}_{i=1}^{m_r}$ such that for every $t\in [0,T]$ there exists a $t_i$ such that $\abs{t_i-t}\leq r^{2\alpha^-+n}$.
Using the compactness result of Arzela-Ascoli we may choose a subsequence of $\eta^{(\tau)}$ such that for all $\delta_r>0$ and for all $\tau\leq \delta_r$
   \[
   \max_{i\in\{1,...,m_r\}}\norm[C^{1,\alpha^-}(Q)]{\eta^{(\tau)}(t_i)-\eta(t_i)}\leq 1,
   \]
where we may assume that $\delta_\tau<r^{2\alpha^-+n}$.

Now for all $(s_1,x_1),(s_2,x_2)\in [0,T]\times Q$ with $B_r\ni x_1,x_2$ a ball with radius $\abs{x_1-x_2}\leq r$ and $\abs{s_1-s_2}\leq r^{2\alpha^- +n}$ there is a $t_i\in [s_1,s_2]$ such that analogous to the previous %
 \begin{align*}
 &\abs{\nabla \eta^{(\tau)}(s_1,x_1)-\nabla \eta(s_2,x_2)}
\leq \abs{\nabla \eta^{(\tau)}(s_1,x_1)-\fint_{B_r}\nabla \eta(s_1)\,dx}+\abs{\fint_{B_r}\nabla (\eta^{(\tau)}(s_1)-\eta^{(\tau)}(t_i))\,dx}
\\
&\quad  +\abs{\fint_{B_r}\nabla (\eta^{(\tau)}(t_i)-\eta(t_i))\,dx} 
 +\abs{\nabla \eta(s_2,x_2)-\fint_{B_r}\nabla \eta(t_i)\,dx}
 \\
 &\leq Cr^{\alpha}+\abs{\Big(\fint_{B_r}\abs{\nabla \eta(s_1)}^2\,dx\Big)^\frac{1}{2}-\Big(\fint_{B_r}\abs{\nabla \eta(s_2)}^2\,dx\Big)^\frac{1}{2}}
\leq  Cr^{\alpha}%
 \end{align*}

  Finally, we use the uniform $L^2([0,T];W^{1,2}(\Omega;\R^n))$ bound on $u^{\tau}$ derived through \autoref{lem:QSdiscreteAPrioriEstimates} and \autoref{lem:globalKorn} to extract another subsequence such that $u^{\tau}$ converges weakly in the space $L^2([0,T];W^{1,2}(\Omega;\R^n))$ to some limit $u$. Equally, the uniform $L^2([0,T];W^{1,2}(Q;\R^n))$ bound on $\partial_t \tilde{\eta}^{\tau}$ implies existence of a subsequence weakly converging to the weak time-derivative $\partial_t \eta$.

Directly from the definiton, we see that
  \[\partial_t \tilde{\eta}^{\tau}(t) = u^{\tau}(t) \circ \eta^{\tau}(t)\]
  for almost all times $t$. So in particular for all $\phi \in C_0^\infty([0,T] \times Q;\R^n)$
  \begin{align*}
   &\phantom{{}={}}\int_0^T \inner[Q]{\partial_t \tilde{\eta} }{\phi} dt \leftarrow \int_0^T \inner[Q]{\partial_t \tilde{\eta}^{(\tau)} }{\phi} dt =\int_0^T \inner[Q]{u^{(\tau)} \circ \eta^{(\tau)} }{\phi} dt \\
   &= \int_0^T \inner[Q]{u^{(\tau)} \circ \eta }{\phi} dt  + \int_0^T \inner[Q]{u^{(\tau)} \circ \eta^{(\tau)} -u^{(\tau)} \circ \eta  }{\phi}dt
  \end{align*}
  Now the first integral converges to $\int_0^T \inner[Q]{u \circ \eta}{\phi} dt$ as $\eta$ is a diffeomorphism, while the second vanishes in the limit by the following argument: Let $\pi_s(t,x) := s\eta^{(\tau)}(t,x) + (1-s)\eta(t,x))$. Then
  \begin{align*}
   &\phantom{{}={}}\abs{u^{(\tau)}(t,\eta^{(\tau)}(t,x)) -u^{(\tau)}(t, \eta(t,x))}^2 = \abs{ \int_0^1 \pd{s} u^{(\tau)}(t,\pi_s(t,x)) ds}^2 \\
   &\leq \int_0^1 \abs{\nabla u^{(\tau)}(t, \pi_s(t,x)) \cdot (\eta^{(\tau)}(t,x)-\eta(t,x)) }^2 ds \\
   &\leq \int_0^1 \abs{\nabla u^{(\tau)}(t, \pi_s(t,x))}^2 ds \sup_{t\in[0,T], x\in Q} \abs{\eta^{(\tau)}(t,x)-\eta(t,x) }^2.
  \end{align*}
  Now, as $\eta^{(\tau)}(t)$ and $\eta(t)$ are both diffeomorphisms with lower bound on the determinant and uniformly close gradients, the linear interpolation $\pi_s$ also has to be a similar diffeomorphism. %
  So integrating the equation yields
  \begin{align*}
   &\phantom{{}={}}\int_0^T \int_Q \abs{u^{(\tau)}(t,\eta^{(\tau)}(t,x)) -u^{(\tau)}(t, \eta(t,x))}^2 dx dt\\
   &\leq  \int_0^T \int_Q \int_0^1 \abs{\nabla u^{(\tau)}(t, \pi_s(t,x))}^2 ds dx dt \sup_{t\in[0,T], x\in Q} \abs{\eta^{(\tau)}(t,x)-\eta(t,x) }^2 \\
   &\leq c \int_0^T \int_\Omega \abs{\nabla u^{(\tau)}}^2 dx dt \sup_{t\in[0,T], x\in Q} \abs{\eta^{(\tau)}(t,x)-\eta(t,x) }^2.
  \end{align*}
  Here the first term is uniformly bounded and the second converges to $0$, by the uniform convergence of $\eta^{(\tau)}$ outlined above.
  
  Thus we have $\partial_t \eta = u \circ \eta$ almost everywhere. %
 \end{proof} 
 
\subsubsection*{Step 4: Convergence of the equation}

Using the convergences we derived in Proposition \ref{prop:QSconvergence}, we proceed by showing that the discrete Euler-Lagrange equations from \autoref{prop:QSdiscreteELequation} converge to the equation of a weak solution. This is not a straightforward task, as we have to deal with coupled pairs of test functions with the coupling being non-linearly dependent on the deformation. We will deal with this issue by focusing on a global test function $\xi$ on $\Omega$ from which we derive the test functions on the discrete level. In order to do so, we need to be able to approximate the test functions smoothly while also maintaining the coupling condition. This is shown in Proposition \ref{lem:approxTestFcts}. 

For the approximation of test-functions we make use of a Bogovski\u{\i}-type theorem.
\begin{theorem}[Bogovski\u{\i}-Operator {\cite[Theorem 2.4]{borchersEquationsRotDiv1990}}]
\label{thm:sohr}
Let $\Omega$ be a bounded Lipschitz domain, then there is a linear operator $\bog:\{g\in C^\infty_0(\Omega)\,|\int_\Omega g\, dy=0\}\to C^\infty_0(\Omega)$,
such that
\[
\diver\bog(g)=g.
\]
Moreover, for $k\in \{0,1,2,...\}$ and $a\in (1,\infty)$ the operator extends to Sobolev-spaces in the form of $\bog:\{g\in W^{k-1,a}_0(\Omega)\,|\int_\Omega g\, dy=0\}\to W^{k,a}_0(\Omega)$, such that
\[
\norm[W^{k,a}_0(\Omega)]{\bog(g)}\leq c\norm[W^{k-1,a}_0(\Omega)]{g},
\]
where the constant just depending on $k,a,n$ and $\Omega$. 
\end{theorem}

Next we introduce the following approximation result. It is introduced in order to approximate test-functions and later in section 4 in order to extend the Aubin-Lions lemma to the variable domain set up. The proof is quite involving and for that reason put in the appendix (see \autoref{app:testfcts}). 
\begin{proposition}[Approximation of test functions] \label{lem:approxTestFcts}
 Fix a function 
\[
\eta \in L^\infty([0,T];\mathcal{E}) \cap W^{1,2}([0,T];W^{1,2}(Q;\R^n)) \text{ with } \sup_{t\in T} E(\eta(t)) < \infty,\]
such that $\eta(t) \notin \partial \mathcal{E}$ for all $t \in [0,T]$.
As before we set $\Omega(t)=\Omega \setminus \eta(t,Q)$.
 Let $\mathcal{T}_\eta$ be the set admissible test functions, which is defined as
\begin{align*}
\mathcal{T}_\eta&:=\{(\phi,\xi)\in W^{1,2}([0,T];W^{1,2}(Q;\R^n))
 \times L^2([0,T];W^{1,2}_0(\Omega);\R^n))\}
\\
&\qquad\text{ s.t. } \phi = \xi \circ \eta \text{ on }[0,T]\times Q\text{ and } \diver\xi(t)=0\text{ in } \Omega(t)\}.
\end{align*}
 Then the set 
 \begin{align*}
\tilde{\mathcal{T}}_\eta &:= \{(\phi,\xi) \in \mathcal{T}_\eta, \xi \in C^\infty([0,T];C^\infty_0( \Omega;\R^n)) \, |\,  \diver \xi(t,y) = 0 
\text{ for all }t\in [0,T]\text{ and all }y 
\\
&\qquad \text{ with } \dist(y,\Omega(t)) < \varepsilon \text{ for some } \varepsilon > 0  \}
 \end{align*}
 is dense in $\mathcal{T}_\eta$
in the following sense:

For $\epsilon$ sufficiently small there is a linear defined on $\tilde{\mathcal{T}}_\eta$,  $(\xi,\phi)\mapsto (\phi_\eps,\xi_\eps)$, such that $(\phi_\eps,\xi_\eps)\in \tilde{\mathcal{T}}_\eta$ and it satisfies the following 
 \begin{align*}
 \diver(\xi_\eps(t,y))&=0\text{ for all }y\in \Omega\text{ with }\dist(y,\Omega(t))\leq \epsilon
 \end{align*}
 If $\xi\in L^b([0,T];W^{k,a}(\Omega))$, for $k\in \N$, $a\in (1,\infty)$ and $b\in [1,\infty]$, then
 \begin{align*}
 \xi_\eps&\to \xi \text{ in } L^b([0,T];W^{k,a}(\Omega)).
 \end{align*}
 If additionally $\eta \in L^b([0,T]; W^{k,a}(Q;\R^n))$, with $a=2$, if $k\geq 3$, then
 \begin{align*}
 \phi_\eps&\to \phi  \text{ in }L^b([0,T]; W^{k,a}(Q;\R^n))\cap W^{1,2}([0,T];W^{1,2}(Q;\R^n))
  \end{align*}

In case, $\partial_t\xi \in L^2([0,T];W^{1,2}(\Omega))$, we find that $\partial_t\xi_\eps\to \partial_t\xi$ in in $L^2([0,T];W^{1,2}(Q))$. If additionally $\xi\in L^\infty([0,T];W^{3,a}(\Omega))$ with $a>n$ and $\partial_t\xi \in L^2([0,T];W^{1,2}(\Omega))$, we find that
 $\partial_t\phi_\eps\to \partial_t\phi$ in $L^2([0,T];W^{1,2}(Q))$.
 
  Moreover, the following bounds are satisfied for every time-instance where the right hand side is bounded:
  \begin{align*}
    \norm[W^{1,2}(\Omega)]{\xi_\eps}&\leq c\norm[W^{1,2}(\Omega))]{\xi}
   \\
  \norm[L^2(\Omega)]{\xi_\eps-\xi}
  &\leq c\eps^\frac{2}{n+2}\norm[W^{1,2}(\Omega)]{\xi},
  \\ 
  \norm[W^{k,a}(\Omega)]{\xi_\eps}&\leq c(\eps)\norm[L^2(\Omega)]{\xi}
  \\
   \norm[W^{k,a}(Q)]{\phi_\eps}&\leq c\norm[C^k(\Omega)]{\xi}\norm[W^{k,a}(Q)]{\eta}\leq c(\eps)\norm[L^2(\Omega)]{\xi}\norm[W^{k,a}(Q)]{\eta}
 \end{align*}
 where the constant $c$ depends on the bounds of $\eta\in L^\infty([0,T];\mathcal{E}) \cap W^{1,2}([0,T];W^{1,2}(Q;\R^n))$ and the injectivity of $\eta$ only. The constant $c(\eps)$ depends additionally on $\eps$.
\end{proposition}

\autoref{lem:approxTestFcts} allows to pass to the limit with the discretized coupled PDE.
 
 \begin{proposition}[Limit-equation] \label{prop:QSELeq} 
 The limit pair $(\eta,v)$ as obtained in Proposition \ref{prop:QSconvergence} satisfies the following:
 \begin{align}
 \label{eq:ELnopres}
  0&= \int_0^T \inner[Q]{DE(\eta(t))}{\phi} + \inner[Q]{D_2R(\eta(t),\partial_t \eta(t))}{\phi} 
  + \nu \inner[\Omega(t)]{\nablasym v}{\nablasym \xi} \\ \nonumber &\qquad - \rho_f \inner[\Omega(t)]{f}{\xi} - \rho_s \inner[Q]{f\circ \eta}{\phi} dt
 \end{align}
 for all pairs $\phi \in L^2([0,T];W^{2,q}(Q;\R^n))$, $\xi \in L^2([0,T];W^{1,2}(\Omega;\R^n))$ which satisfy $\phi(t,.) = \xi \circ \eta(t)$ on $Q$ and $\diver \xi(t) = 0$ on $\Omega(t)$.
 \end{proposition}

 \begin{proof}
  First, we use the Minty method to show that $\inner[Q]{DE(\eta^\tau(t))}{\phi^\tau} \to \inner[Q]{DE(\eta(t))}{\phi}$. Fix $t\in [0,T]$ and pick $\psi\in C^\infty_0(Q;[0,1])$. Then, the pair $((\eta^{(\tau)}-\eta)\psi,0)$ fulfills the coupling condition for the discrete Euler-Lagrange equation and we have
 \begin{align*}
   &\phantom{{}={}}\inner{DE(\eta^{(\tau)})-DE(\eta)}{(\eta^{(\tau)}-\eta)\phi} \\
   &= -\inner{DE(\eta)}{(\eta^{(\tau)}-\eta)\phi} - \inner{
   D_2R(\eta^{(\tau)},\partial_t \tilde{\eta}^{(\tau)})}{(\eta^{(\tau)}-\eta)\phi} + \inner{f}{(\eta^{(\tau)}-\eta)\phi}
  \end{align*}
  As $\eta^{(\tau)}(t) \to \eta(t)$ weakly in $W^{2,q}(Q; \R^n)$ and strongly in  $C^{1,\alpha^-}(Q; \R^n)$ and moreover as  $\partial_t \tilde{\eta}^{(\tau)}(t)$ is uniformly bounded in $L^2 ([0,T]\times Q; \R^n)$ all three terms on the right hand side converge to $0$ when integrated in time and thus for almost all $t \in [0,T]$ by \autoref{prop:QSconvergence}.  Hence \autoref{ass:energy}, S6 implies the strong convergence of $\eta^{(\tau)}(t) \to \eta(t)$ in $W^{2,q}(Q)$ for almost all $t\in [0,T]$.

  By \autoref{lem:approxTestFcts}, it is enough to show the limit equation for $\xi \in C_0^\infty([0,T]\times \Omega ;\R^n)$, which is divergence free on a slightly larger set than the fluid domain. Fix such a $\xi$. Then since $\eta^{(\tau)}$ converges uniformly to $\eta$, $\diver \xi = 0$ on $\Omega^{(\tau)}(t)$ for all $\tau$ small enough. 
  
  Now we construct a matching $\phi^{(\tau)}(t,x) := \xi(t, \eta^{(\tau)}(t,x))$ and $\phi(t,x) := \xi(t,\eta(t,x))$. Then by \autoref{lem:W2qInvertible} and \autoref{prop:W2qIsomorph}, $\phi^{(\tau)} \in L^\infty([0,T];W^{2,q}(Q;\R^n))$ with uniform bounds. Thus we get by compactness and uniqueness of limits $\phi^{(\tau)}(t) \rightharpoonup \phi(t)$ in $W^{2,q}(\R^n)$.
  
  As constructed, the pairs $(\phi^{(\tau)}(t),\xi(t))$ are admissible in the respective Euler-Lagrange equations from \autoref{prop:QSdiscreteELequation} and we have
  \begin{align*}
   0&= \inner{DE(\eta^{(\tau)}(t))}{\phi^{(\tau)}(t)} + \inner{D_2R\left(\eta^{(\tau)}(t),\partial_t \tilde{\eta}^{(\tau)}(t)\right)}{\phi^{(\tau)}(t)} \\
   &+ \nu\inner[\Omega^{(\tau)}(t)]{\nablasym u^{(\tau)}(t)}{\nabla \xi(t)} - \rho_f\inner[\Omega^{(\tau)}(t)]{f}{\xi(t)} - \rho_s\inner[Q]{f\circ \eta^{(\tau)}(t) }{\phi^{(\tau)}(t)}
  \end{align*}
  for all $t\in[0,T]$ and $\tau$ small enough.

  Now we integrate this equation in time and check each of the terms for convergence. For the first term we note that by the strong convergence of $\eta^{(\tau)}$ in $W^{2,q}(Q)$ and \autoref{ass:energy}, S5 that $DE(\eta^{(\tau)}(t))$ converges strongly in $W^{-2,q}(Q;\R^n)$ for every fixed $t$. Since $\phi^{(\tau)}(t)$ converges weakly and both terms are uniformly bounded in their respective spaces, we get
  \begin{align*}
   \int_0^T \inner{DE(\eta^{(\tau)}(t))}{\phi^{(\tau)}(t)} dt \to \int_0^T \inner{DE(\eta(t))}{\phi(t)} dt.
  \end{align*}
  For the next term we find by \autoref{prop:QSconvergence} %
  and the continuity of $R$ in \autoref{ass:dissipation}, R1 that $D_2R\left(\eta^{(\tau)},\partial_t \tilde{\eta}^{(\tau)}\right)$ converges weakly in $L^2(0,T;W^{-1,2}(Q;\R^n))$ and $\phi^{(\tau)}$ converges strongly in $L^2(0,T;W^{1,2}(Q;\R^n))$ which implies that
  \begin{align*}
   \int_0^T \inner{D_2R\left(\eta^{(\tau)}(t),\partial_t \tilde{\eta}^{(\tau)}(t)\right)}{\phi^{(\tau)}(t)} dt \to \int_0^T \inner{D_2R\left(\eta(t),\partial_t \eta(t)\right)}{\phi(t)} dt.
  \end{align*}
  
  For the next terms, let us first deal with the variable domain by rewriting the terms using characteristic functions. Denoting the symmetric difference by $A\triangle B := A\setminus B \cup B\setminus A$ we have 
  \[ \int_0^T \norm[L^2(\Omega)]{\chi_{\Omega^{(\tau)}(t)}- \chi_{\Omega(t)}}^2 dt = \int_0^T \abs{\Omega^{(\tau)}(t) \triangle \Omega(t)} dt \to 0\]
  by the uniform convergence of the boundary and can thus conclude
 \begin{align*}
  &\phantom{{}={}}\int_0^T \inner[\Omega^{(\tau)}(t)]{\nabla u^{(\tau)}(t)}{\nabla \xi(t)} dt = \int_0^T \int_\Omega \chi_{\Omega^{(\tau)}(t)} \nabla u^{(\tau)}(t) : \nabla \xi(t) \,dy dt \\
  &\to \int_0^T \int_\Omega \chi_{\Omega(t)} \nabla u(t) : \nabla \xi(t) \, dy dt = \int_0^T \inner[\Omega(\cdot)]{\nabla u^{(\tau)}(t)}{\nabla \xi(t)} dt.
 \end{align*}
 as $u$ converges weakly in $L^2([0,T];W^{1,2}(\Omega;\R^n))$.

 The same approach also works for the forces on the fluid, where the domain is the only variable in $\tau$ and thus
 \begin{align*}
  \int_0^T \rho_f\inner[\Omega^{(\tau)}(t)]{f}{\xi(t)} dt \to \int_0^T \rho_f\inner[\Omega(t)]{f}{\xi(t)} dt 
 \end{align*}
 Finally we have the forces acting on the solid. Here both sides converge uniformly:
 \begin{align*}
  \int_0^T \rho_s\inner[Q]{f\circ \eta^{(\tau)}(t) }{\phi^{(\tau)}(t)} dt \to \int_0^T \rho_s\inner[Q]{f\circ \eta(t) }{\phi(t)} dt
 \end{align*}
 
 Collecting all the terms than concludes the proof.
 \end{proof}  

 \subsubsection*{Step 5: Construction of the pressure}
Take  $s\in (0,T)$ with. Since there is  a distance to the possible collision, we know that $\Omega(t)$ is a uniform Lipschitz domain with bounds on the energy for all $t\leq s$. Taking $\psi\in C_0^\infty(\Omega(t))$, such that $\int_{\Omega(t)} \psi dy =0$, we can use the Bogovski\u{\i}-operator $\bog_t$ defined on $\Omega(t)$ via \autoref{thm:sohr} to define
  \begin{align*}
   \tilde{P}(t)(\psi) =  \nu \inner[\Omega(t)]{\nablasym u}{\nablasym \bog_t \psi} - \rho_f \inner[\Omega(t)]{f}{\bog_t \psi} dt.
  \end{align*}
This then gives the estimate
  \begin{align*}
   \abs{\tilde{P}(t)(\psi)} = C \norm{\bog_t} \norm[L^2(\Omega(t))]{\psi}
  \end{align*}
  where $\norm{\bog_t}$ is the operator-norm of $\bog_t: \{\psi \in L^2(\Omega(t)): \int_{\Omega(t)} \psi = 0\, dy\} \to W^{1,2}(\Omega(t))$ which is bounded by the Lipschitz constants of $\Omega(t)$ by \autoref{thm:sohr}. 
  Now since $\{\psi \in L^2(\Omega(t)): \int_{\Omega(t)} \psi\, dy = 0\}$ is a Hilbert space we find a $\tilde{p}(t)$ in that space such that $\tilde{p}(t)\equiv \tilde{P}(t)$. 
  
  We can extend the operator to $L^2(\Omega(t))$ in the following way: 
  Take $\varphi(t)\in C^\infty_0(\Omega(t))$ and $\tilde{\varphi}(t)\in C^\infty_0(\Omega\setminus\Omega(t))$ fixed, such that $\int\varphi(t)\, dy=\int\tilde{\varphi}(t)\, dy=1$ for all $t\in [0,s]$. Since the change of domain in time is uniformly continuous, we may assume further that $\varphi,\tilde{\varphi}$ are $C^1$ smooth in time.  
Next we define $\bog$ to be the operator of Theorem~\ref{thm:sohr} with respect to the full domain $\Omega$. 
 
By taking the fixed pair of test functions \[
\xi_0(t) := \bog(\varphi(t)-\tilde{\varphi}(t)), \quad \phi_0(t,x) := \xi_0(t,\eta(t,x)),
\] we may define
 \begin{align*}\hat{p}(t,y)&=
 \Big(\inner[Q]{DE(\eta(t))}{\phi_0(t)} + \inner[Q]{D_2R(\eta(t),\partial_t \eta(t))}{\phi_0(t)} 
  + \nu \inner[\Omega(t)]{\nablasym u(t)}{\nablasym \xi_0(t)} \\ \nonumber &\qquad - \rho_f \inner[\Omega(t)]{f(t)}{\xi_0(t)} - \rho_s \inner[Q]{f(t)\circ \eta(t)}{\phi_0(t)}\Big)\varphi(t,y)
 \end{align*}
 which satisfies $\norm[{L^2([0,s];L^\infty(\Omega(t))}]{\hat{p}}\leq C$ with $C$ depending on the energy estimates only.
 But this allows to introduce the pressure. We define for $\psi\in L^1(\Omega(t))$, $c_\psi(t)=\int_{\Omega(t)}\psi(t)\, dy$. Now, if $\psi\in L^2([0,T],L^1(\Omega(t)))$ we find that $c_\psi\in L^2([0,s])$. Hence we may define
\[
P(\psi)=\int_0^T\inner{\tilde{p}}{\psi-c_\psi\varphi}\, dt+\int_0^T\int_{\Omega(t)}\hat{p}\, dy c_\psi \, dt\]

Thus $p\in L^\infty(0,s;L^2(\Omega(t))+L^2(0,s;L^\infty(\Omega(t))$ is well defined via that operator
 \begin{align*}
  \int_0^T \inner{\nabla p}{\xi} dt := P(\diver \xi),
 \end{align*}
 and satisfies the proposed regularity.

One can now check that it fulfills the right equations. For that it suffices to see, that
$\xi-\bog_t(\diver(\xi)-c_{\diver(\xi)}\varphi)-c_{\diver(\xi)}\bog(\varphi-\tilde{\varphi})=\xi-\bog_t(\diver(\xi)-c_{\diver(\xi)}\varphi)-c_{\diver(\xi)}\xi_0$ is divergence free over $\Omega(t)$. Hence \eqref{eq:QSEL} is satisfied by \eqref{eq:ELnopres} using the test-function $(\xi-\bog_t(\diver(\xi)-c_{\diver(\xi)}\varphi)--c_{\diver(\xi)}\xi_0,\phi-c_{\diver(\xi)}\phi_0)$.

 This finally allows us to conclude the Theorem.
 \begin{proof}[Proof of \autoref{thm:QSexistence}]
  For any injective $\eta_0$ there is a short interval $[0,T]$ such that for all $\tau$ small enough, such that all $\eta_k^{(\tau)}$ are injective according to \autoref{cor:QSshortTimeNoCollision}. But then we can take a subsequence and a limit $(\eta,v)$ and conclude by \autoref{prop:QSdiscreteELequation} that this limit is a solution on $[0,T]$.
  
  Now let $[0,T_{\max})$ be a maximal interval on which a solution $(\eta,v)$ constructed in the previous way exists. If $T_{\max} = \infty$, we are done. The same holds if $T_{\max}<\infty$ and $\liminf_{t\to T_{\max}} E(\eta(t)) = \infty$. Now assume that this is not the case. Then there exists a sequence of times $t_i \nearrow T_{\max}$ such that $E(\eta(t_i))$ is bounded and there exists a limit, which we will denote $\eta(T_{\max})$.
  
  Now let $E_0 := \liminf_{t\to T_{\max}} E(\eta(t)) \geq E(\eta(T_{\max}))$ due to lower semicontinuity. Following \autoref{lem:QSdiscreteAPrioriEstimates} and \autoref{lem:QSdiscreteHoelderEstimate}, there exists a minimal time $T$ on which the solution any solution starting with energy below $2E_0$ stays below energy $3E_0$ and is H\"{o}lder-continuous in time in that time interval. Due to the convergence, we can pick $t_i$ with $T_{\max} - t_i \leq T$ and $E(\eta(t_i)) \leq 2E_0$, but then the solution is H\"{o}lder-continuous right until $T_{\max}$ and thus $\lim_{t\nearrow T_{\max}} \eta(t) = \eta(T_{\max})$. But then we can use the short term existence to construct a solution starting from $\eta(T_{\max})$ and appending this to the previous solution yields a contradiction as $T_{\max}$ cannot be maximal.

 \end{proof}

\subsection{The example Energy-Dissipation pair}
\label{subsec:example}

Let us now consider the prototypical example we stated in the introduction in the form of \eqref{kelvinVoigt} and \eqref{st-venant}.
In particular, we will prove that those fulfill Assumptions \ref{ass:energy} and \ref{ass:dissipation}. Moreover, we comment a bit more on the meaning of those assumptions and on how those are important in the course of the construction. Effectively we will prove the following proposition.

\begin{proposition}
 The example energy and dissipation given in \eqref{kelvinVoigt} and \eqref{st-venant} fulfill the assumptions S1-S6 and R1-R4 respectively. In particular the resulting fluid-structure interaction problem has a weak solution, under the additional conditions given in \autoref{thm:QSexistence}.
\end{proposition}

Instead of proving the assumptions in ascending order or order of convenience, we will try to tackle them in the order as they appear in the proof of \autoref{thm:QSexistence}. Furthermore, we will roughly group them by some relevant subtopics.
\subsubsection*{The minimization problem (S1,S3-S4,R1-R2)}

We start with the definition of $\eta_{k+1}^{(\tau)}$ in the minimizing movements-scheme in \eqref{eq:QSdiscreteProblem}. In order to prove existence of minimizers, we need to invoke the direct method of the calculus of variation. Given a minimizing sequence, we find a converging subsequence and then show that the resulting limit has indeed a minimal value. In other words, we need to show compactness and lower semicontinuity, as well as a lower bound for the functional. 

The last one seems to be directly stated in S1 as well as in the quadratic homogeinity in R2. Of course for our example energy S1 immediately holds, as all terms are non-negative and R2 is similarly obvious, as $\partial_t \eta$ occurs as a quadratic factor. There is however some hidden difficulty in order to find a lower bound for the whole functional, which does not only include energy and dissipation, but also the force terms, which can indeed be negative. To counteract these, we actually use the proper quadratic growth of the dissipation, which is immediate for the fluid and a result of the Korn type inequality R3 for the solid. At this point though, as the first argument of the dissipation and the fluid domain are still fixed, there is no need yet, to use R3 to its full extent.%

Once a lower bound for our minimizing sequence is established, we need to consider compactness. Here the relevant topology for $\eta$ is the weak $W^{2,q}(Q;\Omega)$ topology and the relevant assumption for compactness is coercivity, in the form of S4. As we have bounded the other terms in the functional from below without involving the energy $E$, we know that this energy needs to be bounded from above and thus the coercivity allows us to use the Banach-Alaoglu theorem to extract compactness. In our example, the coercivity is obtained in the most simple way, as $\norm[L^q(Q)]{\nabla^2 \eta}^q$ is part of the energy. But it should be noted that we do not need a specific growth condition, since the energy will always be uniformly bounded throughout the paper. The precise form of the term will play a role for some of the other properties though.

As for(weak) lower semicontinuity, we need to verify assumptions S3 and R1 for the example case. First, note that the highest order term in the energy  $\norm[L^q(Q)]{\nabla^2 \eta}^q$ is weakly lower semi-continuous. Second, we find that $q>n$ allows us to pick another subsequence converging in $C^{1,\alpha}$ for some $\alpha<\frac{q-n}{q}$. This allows us to pass to the limit in the terms $\int_Q \frac{1}{8}|\nabla \eta^T \nabla \eta-I|_\mathcal{C} dx$ and $\frac{1}{(\det \nabla \eta)^a}$ in \eqref{st-venant}.

\subsubsection*{Converting between Lagrangian and Eulerian setting (S2)}

Note that as long as we were only discussing the minimization over the solid, the specific choice of $W^{2,q}(Q;\R^n)$ as a space was unimportant and choosing different terms in the integrand might as well have led us to a different space. It however becomes important when adding in the fluid, since it is prescribed w.r.t.\ Eulerian setting which is again determined by the solid deformation $\eta$.
The key here is the assumption S2, on the determinant. Not only does this result in physically reasonable injectivity (in conjunction with the Ciarlet-Ne\v{c}as condition), but it also allows us to convert between Eulerian and Lagrangian quantities as it actually implies that $\eta$ is a diffeomorphism with uniform bounds. In particular, this will then imply that the fluid domain has a regular enough boundary, which will play an important role in the hyperbolic case.

To prove this property we follow the ideas of~\cite{healeyInjectiveWeakSolutions2009} where a similar energy was studied.
  Define $f(x) := \det \nabla \eta$. If $E(\eta)$ is bounded, then $f$ is bounded in $W^{1,q}(Q)$ and $C^\alpha(Q)$. Now for a fixed $\epsilon_0$ assume that there is $x_0 \in Q$ with $f(x_0) = 2\epsilon_0$. Then
  \begin{align*}
   &\phantom{{}={}} E(\eta) \geq \int_{B_\delta(x_0)\cap Q} \frac{1}{f(x)^a} dx \geq c \frac{\delta^n}{ \left(\int_{B_\delta(x_0)\cap Q} f(x) dx \right)^a} \\
   &\geq c \frac{\delta^n}{\left(f(x_0) + \int_{B_\delta(x_0)\cap Q} \abs{f(x)-f(x_0)} dx \right)^a} 
   \geq c \frac{\delta^n}{\left(2\epsilon_0 + C \delta^\alpha \right)^a}
  \end{align*}
  However if $a\alpha >n$, the right hand side can be arbitrarily large if $\epsilon_0$ and $\delta$ are choosen small enough, which is a contradiction.
  
\subsubsection*{Uniform bounds (R3)}

It has been long known that there is a certain mismatch between physically reasonable and mathematically expedient dissipation functionals (see e.g.\ \cite{antmanPhysicallyUnacceptableViscous1998}). We would prefer the dissipations to be of the forms $\norm[W^{1,2}(Q)]{\partial_t \eta}^2$ and $\norm[W^{1,2}(\Omega)]{u}^2$. This would then almost directly lead to $L^2([0,T];W^{1,2}(Q;\R^n))$ and $L^2([0,T];W^{1,2}(\Omega;\R^n))$-bounds respectively for $\partial_t \eta$ and $v$ as well as their approximations. Instead, we are only given $R(\eta,\partial_t \eta)$ and $\norm[\Omega(t)]{\nablasym v}^2$. Thus we require Korn-inequalities to convert the bounds for latter into bounds for the former.

As the Korn inequality for the fluid is the classic one and the added difficulties due to the changing domain are overcome by \autoref{lem:globalKorn}, we only need to focus at the solid. For our example, this inequality and thus R3 follows from the main theorems in \cite{neffKornFirstInequality2002,pompeKornFirstInequality2003}. See also the discussion in \cite{mielkeThermoviscoelasticityKelvinVoigtRheology2019}, where these results are coupled with an energy similar to ours in the context of a thermoviscoelastic solid (but without a fluid).%

Observe that these inequalities require a certain regularity of the deformation $\nabla \eta$ itself. In fact we need the same properties that allow us to switch between Lagrangian and Eulerian settings, i.e. a uniform lower bound on the determinant $\det \nabla \eta$ and continuity of $\nabla \eta$, as otherwise there are known counterexamples for which the inequality fails.

\subsubsection*{Weak equations (S5, R4)}

Combined, the assumptions so far are enough to construct iterative minimizers and even to have a subsequence converge to a limit object $\eta,v$ in space-time by weak compactness. We are left to show that these function do satisfy a weak coupled PDE. This is where the assumptions S5 and R4 come in. Both of them are two-part in nature, requiring both the existence of a derivative as well as some form of continuity. Both are also immediately shown for the example by just doing the calculation. Let us start with the dissipation, namely
\begin{align*}
  \inner{D_2 R(\eta,b)}{\phi} = \int_Q 2(\nabla b^T \nabla \eta +\nabla \eta^T \nabla b)\cdot (\nabla \phi^T \nabla \eta + \nabla \eta^T \nabla \phi) \,dx.
\end{align*}
Since we have $C^{1,\alpha}(Q;\R^n)$-bounds on $\nabla \eta$, the $L^2(Q)$-regularity of $\nabla b (=\nabla \partial_t \eta)$ is enough to make sense of $D_2 R(\cdot,b)$ as an operator in $W^{-1,2}(Q;\R^n)$. Similarly, the uniform convergence in some H\"older space for $\nabla \eta$ is enough to give this derivative the required continuity with respect to both $b$ and $\eta$.

The calculation for the energy is a bit more involved. Restricting ourselves to deformations $\eta$ of finite energy and thus positive determinant, we get by a short calculation
\begin{align*}
 \inner{DE(\eta)}{\phi} &= \int_Q \frac{1}{4} \mathcal{C}(\nabla\eta^T \nabla\eta-I) \cdot (\nabla \phi^T \nabla \eta + \nabla\eta^T \nabla \phi) \\
 &- a \frac{\cof \nabla \eta}{(\det \nabla \eta)^a} \cdot \nabla \phi + \abs{\nabla^2 \eta}^{q-2} \nabla^2 \eta \cdot \nabla^2 \phi\, dx
\end{align*}
where the scalar products are to be understood over all tensorial dimensions. 

Again in order to pass to the limit with the energy, we need to make use of the uniform H\"{o}lder-continuity of $\nabla \eta$ to see that the first two terms in $DE(\eta)$ are well defined and continuous with respect to the corresponding convergence. Finally, the last term is well defined since $\eta \in W^{2,q}(Q;\R^n)$ uniformly, but to show that it is also continuous we need to show strong convergence using the convexity of the quantity. This then leads to the final assumption. %

\subsubsection*{Improved convergence (S6)}

As the usual compactness methods will only result in weak compactness, and S5 requires strong convergence we need a way to improve upon this. For this we rely on an idea that is most commonly attributed to Minty. While it is certainly not true that our energy is convex, the critical, second order term in its derivative $DE(\eta)$ is monotone and this allows us to improve convergence as desired.

Assume that as stated $\eta_l \rightharpoonup \eta$ in $W^{2,q}(Q;\R^n)$. Then after possibly another subsequence $\eta_l \to \eta$ in $C^{1,\alpha}(Q;\R^n)$ and in particular the first two terms of $DE(\eta_l)$ already converge (using the lower bound on $\det \nabla \eta$ given through S2). As a result, the stated conditions on convergence of $\inner{DE(\eta_l)-DE(\eta)}{(\eta_l-\eta)\psi} \to 0$ for all cutoffs $\psi \in C_0^\infty(Q;[0,1])$ are equivalent to those for
\begin{align*}
 \inner{\abs{\nabla^2 \eta_l}^{q-2} \nabla^2 \eta_l -\abs{\nabla^2 \eta}^{q-2} \nabla^2 \eta}{\nabla^2 ((\eta_l-\eta)\psi)} 
\end{align*}
Here the cutoff complicates things slightly, but expanding the right hand side yields terms of lower order ($(\eta_l-\eta) \otimes \nabla^2 \psi$ and $\nabla (\eta_l-\eta) \otimes \nabla \psi$) which already converge strongly to $0$ and one term of second order, which leaves us with
\begin{align*}
 \inner{\abs{\nabla^2 \eta_l}^{q-2} \nabla^2 \eta_l -\abs{\nabla^2 \eta}^{q-2} \nabla^2 \eta}{(\nabla^2\eta_l-\nabla^2\eta))\psi}
\end{align*}
where we now can send $\psi \to 1$ by approximation.
Now $\eta \mapsto \abs{\nabla^2 \eta}^{q-2} \nabla^2 \eta$ is a classic example of a monotone operator. Thus the term is bounded from below by $0$ and its convergence to $0$ implies strong convergence $\eta_l \to \eta$ in $W^{2,q}(Q;\R^n)$, by the fact that for $q\geq 2$ and $a,b\in \R^{n^3}$
\[
(\abs{a}^{q-2}a-\abs{b}^{q-2}b)\cdot (a-b)\geq c\abs{a-b}^q.
\]
\section{Minimizing movements for hyperbolic evolutions}
\label{sec:so}

In this section, we will introduce a general method for adding inertial effects to continuum mechanical problems, thereby turning them from parabolic to hyperbolic. We will demonstrate this in the purely Lagrangian case of a single viscoelastic solid, but as we will see in the next section, the method turns out to be flexible enough to allow even for problems which are of a mixed Lagrangian/Eulerian type such as fluid-structure interaction. Also note that while this section can be read independently from the previous one, at some places we will use a similar reasoning, which will thus be abridged slightly.

In particular we keep the notation from the previous section. Thus $\eta:Q \to \R^n$, $\eta \in \mathcal{E}$ is the deformation of the solid specimen and $E$ and $R$ are its elastic energy and dissipation. For simplicity we will also use the same set of assumptions, though many of them could be relaxed, as they are intended for interaction with the fluid, see in particular \autoref{rem:SOgeneralEnergies}. Also, as there are no changing domains within this section, we will suppress the dependence of the inner products and the resulting $L^2$-norms on $Q$.

The problem we thus want to solve is to find the deformation of the viscoelastic solid specimen moving inertially in space subject to an action of forces. In other words, we need to solve the balance of momentum (Newton's second law) that reads as
\begin{align}
\label{eq:balMomSolid}
 \rho \partial_t^2 \eta = D_2R(\eta,\partial_t \eta) + DE(\eta) - f \circ \eta
\end{align}
where $\rho= \rho_s$ is a constant density and $f$ some, not necessarily conservative, external force. Note that the assumption that $\eta\in \mathcal{E}$ implies that it satisfies given Dirichlet boundary conditions on $P$. On the other parts of the boundary $\partial Q\setminus P$ we assume here natural Neumann type (free) boundary conditions that will result from minimization.

 As usual, we translate this into the notion of a weak solution.

\begin{definition}[Weak solution to the inertial problem for solids] \label{def:SOweakSolution}
 We call $\eta \in L^\infty([0,T];\mathcal{E}) \cap W^{1,2}([0,T];W^{1,2}(Q;\R^n))$, such that $\partial_t\eta\in C^0_w([0,T];L^2((Q;\R^n))$ and $\eta(0) = \eta_0$ a weak solution to the inertial problem of the viscoelastic solid \eqref{eq:balMomSolid} if 
 \begin{align*}
\int_0^T \inner{DE(\eta)}{\phi} + \inner{D_2R(\eta,\partial_t \eta)}{\phi}-\inner{f\circ\eta}{\phi} + \rho \inner{\partial_t \eta}{\partial_t \phi} dt - \rho \inner[Q]{\eta'}{\phi(0)} = 0%
 \end{align*}
 for all $\phi \in C^\infty([0,T];C^\infty(Q;\R^n))$ with $\phi|_{[0,T]\times P} = 0$ such that $\phi(T) = 0$.
\end{definition}
Observe, that we can restrict the solution to the particular closed set $\mathcal{E}$ and thus will only work with injective deformations on $Q$. This will be of particular interest to us as this property is relevant for modelling fluid-structure interactions. 

The main goal of this section will be to prove the following theorem.

\begin{theorem}[Existence of solutions for solids] \label{thm:SOexistence}
 Assume that the conditions from \autoref{ass:energy} and \autoref{ass:dissipation} hold. Assume that the initial data $\eta_0 \in \mathcal{E} \setminus \partial \mathcal{E}$ with $E(\eta) < \infty$, that $\eta' \in L^2(Q;\R^n)$ and that $f \in C^0([0,\infty)\times\R^{n};\R^n)$.  Then there exists a weak solution to \eqref{eq:balMomSolid} according to \autoref{def:SOweakSolution} on $[0,T]$. Furthermore, $T> 0$ can be chosen in such a way that $T = \infty$ or $\eta(T) \in \partial \mathcal{E}$ (see also \autoref{rem:SOcollisions}).
\end{theorem}

Before we begin, let us give the general idea of the proof of \autoref{thm:SOexistence}. The main goal is to approximate the sought solution of the hyperbolic problem by solutions to suitably constructed parabolic problems. This will open up the pathways to using a wide array of methods developed for the parabolic situation. For us, this concerns particularly the \emph{method of minimizing movements}. The key is in discretizing the second time derivative in \eqref{eq:balMomSolid} by a difference quotient w.r.t.\ the acceleration scale $h$. We will thus first solve what we will call the time-delayed problem:
\begin{align}
\label{eq:timeDelayed}
 \rho\frac{\partial_t \eta^{(h)}(t)-\partial_t \eta^{(h)}(t-h)}{h} = -DR_2(\eta^{(h)}(t),\partial_t \eta^{(h)}(t)) - DE(\eta^{(h)}(t)) + f \circ \eta^{(h)}(t)
\end{align}
For any fixed $h$, \eqref{eq:timeDelayed} has the structure of a gradient flow, yet one with a nonlocality in time in form of the term $\partial_t \eta^{(h)}(t-h)$. Now the important observation is, that on the interval $[0,h]$, $\partial_t \eta^{(h)}(t-h)$ is not part of the solution but actually given as the initial data. Thus, on this interval, the problem can be solved using parabolic methods. But then, once we know the solution on $[0,h]$, we can use this as data for the problem on $[h,2h]$ and iterate.
To allow for an interation process, we in particular need to know that the solution obtained from the previous step is admissible to play the role of data in the next step. In other words, we need to assure that that $E(\eta^{(h)}(h))$ is bounded and that $\partial_t \eta^{(h)}$ possesses the necessary integrability. This is guaranteed by proving a suitable energy inequality, \emph{a key element of the proof}. Fundamentally, a gradient flow needs to be viewed in terms of energy and dissipation. In particular there is always an energy balance, which often only takes the form of an inequality. In our case, for the time delayed problem on $[0,h]$, the energy inequality will have the form
\begin{align*}
 &\phantom{{}={}} E(\eta^{(h)}(h)) + \frac{\rho}{2h} \int_0^h \norm{\partial_t \eta^{(h)}(t)}^2 dt + \int_0^h R(\eta^{(h)}(t),\partial_t \eta^{(h)}(t)) dt \\
 &\leq E(\eta^{(h)}(0)) + \frac{\rho}{2h} \int_0^h \norm{\partial_t \eta^{(h)}(t-h)}^2 dt + \int_0^h \inner{f\circ \eta^{(h)}}{\partial_t \eta^{(h)}} dt
\end{align*}
Let us elaborate the terms in this inequality: On the right-hand side, we have the potential energy $E$ of the initial data, as well as the averaged kinetic energy $\frac{\rho}{2} \fint_0^h \norm{\partial_t \eta^{(h)}}^2 dt$ of the ``previous step''. On the left hand side, we have the potential energy at the end of the step, as well as the averaged kinetic energy of the current step.

So not only have we bounded the initial data for the next step in terms of the initial data of the previous step and thus can iterate, we also have an estimate suitable to employ a telescope argument. Indeed, by summing up the estimate, over $l$ time intervals of length $h$, we will gain a uniform bound on the new endpoint $E(\eta^{(h)}(lh))$ and $\frac{\rho}{2} \fint_{(l-1)h}^{lh} \norm{\partial_t \eta^{(h)}}^2 dt$ only in terms of the given initial data and forces.
These uniform bounds on $\eta^{(h)}$ are independent of $h$, thus they allow us to take the limit $h \to 0$ and obtain a solution to the hyperbolic problem through first adapting the Aubin-Lions lemma to $\partial_t\eta$ and then using the Minty trick.  

Following this approach, we will show the existence of the time-delayed problem in detail in \autoref{subsec:SOtimeDelayed} before proving \autoref{thm:SOexistence} in \autoref{subsec:SOproof}. Finally, we will give some remarks on our method and possible variations of the proof in \autoref{subsec:SOremarks}.

\subsection{The time-delayed problem} \label{subsec:SOtimeDelayed}

For all of this subsection we will assume $h> 0$ to be fixed. In order to solve the time-delayed problem, we first need to give a precise definition of its weak formulation.

\begin{definition}[Weak solutions to the time-delayed equation for solids] \label{def:SOtimeDelayedSolution} Let $w\in L^2([0,h]\times Q;\R^n)$. 
We call $\eta \in L^\infty([0,h] \times Q; \mathcal{E}) \cap W^{1,2}([0,h] ;W^{k_0,2}(Q; \R^n))$ a weak solution to the time-delayed equation \eqref{eq:timeDelayed} if $\eta(0) = \eta_0$ and
 \begin{align} \label{eq:SOtimeDelayed}
  0&=\int_0^h \inner{DE_h(\eta)}{\phi} + \inner{D_2R_h(\eta,\partial_t \eta)}{\phi} - \inner{f\circ \eta }{\phi} + \frac{\rho}{h}\inner{\partial_t \eta - w}{\phi} dt.
 \end{align}
 for all $\phi \in C^\infty([0,h]\times Q;\R^n)$ with $\phi|_{[0,h]\times P} = 0$.
\end{definition}

In this definition $w$ will play the role of the given data $\partial_t \eta(t-h)$. In particular, as we assume $h>0$ to be a given constant throughout this subsection, so we will not track the $h$-dependence of any of the quantities. Note that in Definition \ref{def:SOtimeDelayedSolution} we used the regularized forms of the energy and dissipation potentials that read as
\begin{align}
E_h(\eta)=E(\eta)+h^{a_0}\norm{\nabla^{k_0}\eta}^2 \quad R_h(\eta,b) := R(\eta,b) + h \norm{\nabla^{k_0} b}^2,
\end{align} 
where we choose $k_0$ large enough, such that $k_0-\frac{n}{2}\geq 2-\frac{n}{q}$ which implies that $W^{k_0,2}(Q;\R^n)\subset W^{2,q}(Q;\R^n)$. This actually has no impact on the existence of time delayed solutions. Instead it is a mollifying strategy which will allow us to test the Euler-Lagrange equation with $\partial_t \eta$ in order to obtain the previously mentioned energy inequality (See also \autoref{rem:SOenergyInequality}).
A similar term will also help us with some regularity issues in the fluid-structure interaction problem later in
\autoref{lem:NSfullIterationAPriori}.

We also note that we have:
\begin{remark}[Properties of the regularizing energy and dissipation] \label{cor:regularDissipationProperties}
 For all $h> 0$, we find that $E_h$ fulfills the properties given in \autoref{ass:energy}, replacing $W^{2,q}$ with $W^{k_0,2}$ and $R_h$ fulfills the properties given in \autoref{ass:dissipation} replacing $W^{1,2}$ by  $W^{k_0,2}$ where we may replace $R2$ by 
 \[c\left(\norm{\nabla \lambda}^2+h \norm{\nabla^{k_0} \lambda}^2\right) \leq   R_h(\eta,\lambda) \leq C \left(\norm{\nabla \lambda}^2+h \norm{\nabla^{k_0} \lambda}^2\right).\]
\end{remark}

Now the bulk of this subsection will be devoted to proving the following theorem:

\begin{theorem}[Existence of time delayed solutions for solids] \label{thm:SOtimeDelayedExistence}
 Let $\eta_0 \in \mathcal{E} \cap W^{k_0,2}(Q;\R^n) \setminus \partial \mathcal{E}$, $w \in L^2([0,h] \times Q;\R^n)$ and $f \in C^0([0,h] \times Q; \R^n)$. Then there exists a weak solution to the time delayed equation \eqref{eq:timeDelayed} in the sense of \autoref{def:SOtimeDelayedSolution} or there exists a solution on a shorter interval $[0,h_{\max}]$ such that $\eta(h_{\max}) \in \partial \mathcal{E}$.\footnote{Note that a-posteriori (see \autoref{cor:QSshortTimeNoCollision})
it will be shown that (in dependence of $\eta_0$) there is always a minimal time-length $h_{\min}$ for which it can be guaranteed that $\eta(t)\notin \partial \mathcal{E}$ for $t\in [0,h_{\min}]$.}
\end{theorem}

Before we start, let us discuss how the time delayed problem can still be seen as a type of parabolic gradient flow. In particular, let us compare it to the classical parabolic gradient flow problem at the root of this, which is
\begin{align*}
 DE_h(\eta(t)) = - D_2R_h(\eta(t),\partial_t \eta(t)) + f \circ \eta(t).
\end{align*}
This problem consists of three components: Energy, dissipation and forces. Our goal is to identify the two additional terms in the time-delayed problem with those, in order to show that we still solve a similar problem.

Let us start with the delayed time derivative $\frac{\rho}{h} w(t) = \frac{\rho}{h} \partial_t \eta(t-h)$. As we work in the interval $[0,h]$, this is just a given function, not depending on the $\eta|_{[0,h]}$. But then any such function plays the role of a force. In fact, in contrast to the actual forces we consider in the problem, it is a force given in reference configuration and thus even easier to handle.

The other term, $\frac{\rho}{h} \partial_t \eta(t)$ can be seen as stemming from a quadratic dissipation potential
\begin{align*}
 \hat{R}(\eta,b) := \hat{R}(b) := \frac{\rho}{2h} \norm{b}^2,
\end{align*}
so that $D_2\hat{R}(\eta(t),\partial_t \eta(t)) = \frac{\rho}{h} \partial_t \eta(t)$. 

By this reasoning, we claim that in general, if there is a method to solve the parabolic gradient flow problem, then there the same method can solve the corresponding time delayed problem. We will refrain from trying to formalize this statement.

\begin{proof}[Proof of \autoref{thm:SOtimeDelayedExistence}]
The proof essentially follows the same lines as was done in the last section. We start by a time-discretization; i.e. we fix some time-step size $\tau$ by which we discretize the interval $[0,h]$. Given $\eta_k^{(\tau)}$, we recursively solve the following minimization problem to obtain $\eta_{k+1}^{(\tau)}$
 \begin{align} \label{eq:SOdiscreteProblem}
  & \text{Minimize} & & E_h(\eta) + \tau R_h\left(\eta_k^{(\tau)},\frac{\eta-\eta_k^{(\tau)}}{\tau}\right) - \tau \inner{f_k^{(\tau)} \circ \eta_k}{\frac{\eta-\eta_k^{(\tau)}}{\tau}}  + \tau \frac{\rho}{2h} \norm{\frac{\eta-\eta_k^{(\tau)}}{\tau}-w_k^{(\tau)}}^2 \\
	& \nonumber \text{subject to} & & \eta \in \mathcal{E}
 \end{align}
 where $w_k^{(\tau)} = \fint_{k\tau}^{(k+1)\tau} w \,dt\in L^2(Q;\R^n)$ and $f_k^{(\tau)} = \fint_{k\tau}^{(k+1)\tau} f dt\in L^2(Q;\R^n)$ are in-time averages. 
 
 Note that this is not quite in the form suggested by the previous discussion. Instead we deliberately wrote the last term as a quadratic difference, to give the problem a bit more structure. Note that when expanded, the last term is
 \begin{align*}
  \hat{R}\left(\frac{\eta-\eta_k^{(\tau)}}{\tau}\right) - \frac{\rho}{h}\inner{w_k}{\frac{\eta-\eta_k^{(\tau)}}{\tau}} + \frac{\rho}{2h} \norm{w_k^{(\tau)}}^2;
 \end{align*}
 so these two approaches only differ in a constant, which has no effect on the minimization.
 
 Now using the coercivity of $E$ similar to, but easier as in the proof of \autoref{prop:QSdiscreteELequation}, a (possibly non-unique) minimizer exists and a short calculation shows that it satisfies (assuming that $\eta_{k+1}\notin \partial \mathcal{E}$) the Euler-Lagrange equation
 \begin{align} \label{eq:SOdiscreteELequation}
  0 &= \inner{DE_h(\eta_{k+1}^{(\tau)})}{\phi} + \inner{D_2R_h\left(\eta_k^{(\tau)},\frac{\eta_{k+1}^{(\tau)}-\eta_k^{(\tau)}}{\tau}\right)}{\phi} - \inner{f_k^{(\tau)} \circ \eta_k^{(\tau)}}{\phi} \\ \nonumber &\qquad+ \frac{\rho}{2h} \inner{\frac{\eta_{k+1}^{(\tau)}-\eta_k^{(\tau)}}{\tau}-w_k^{(\tau)}}{\phi}
 \end{align}
 for all $\phi \in W^{2,q}(Q;\R^n)$ with $\phi|_P = 0$.
 
 Next we follow in the steps of \autoref{lem:QSdiscreteAPrioriEstimates} (see \autoref{rem:SOdissipation} for a discussion of some interesting differences) and derive a simple initial energy estimate by comparing the minimizer $\eta_{k+1}$ in \eqref{eq:SOdiscreteProblem} with the choice $\eta = \eta_k$:
 \begin{align} \label{eq:SOdicreteAPrioriEstimate}
  & E_h(\eta_{k+1}^{(\tau)}) + \tau R_h\left(\eta_k^{(\tau)},\frac{\eta_{k+1}^{(\tau)}-\eta_k^{(\tau)}}{\tau}\right) - \tau \inner{f_k^{(\tau)}}{\frac{\eta_{k+1}^{(\tau)}-\eta_k^{(\tau)}}{\tau}}  \\ \nonumber  &\qquad + \tau \frac{\rho}{2h} \norm{\frac{\eta_{k+1}^{(\tau)}-\eta_k^{(\tau)}}{\tau}-w_k^{(\tau)}}^2 
   \leq E_h(\eta_k^{(\tau)}) + \tau \frac{\rho}{2h} \norm{w_k^{(\tau)}}^2.
 \end{align}
 This estimate is can be summed so that, using the triangle and the weighted Young's inequality, we can derive for any $N$ such that $\tau N \leq h$
 \begin{align*}
  &\phantom{{}={}} E_h(\eta_N) + \sum_{k=0}^{N-1} \tau \left[ R_h\left(\eta_k^{(\tau)},\frac{\eta_{k+1}^{(\tau)}-\eta_k^{(\tau)}}{\tau}\right) + c \norm{\frac{\eta_{k+1}^{(\tau)}-\eta_k^{(\tau)}}{\tau}}^2 \right] \\
  & \leq E_h(\eta_0) + \tau  C \sum_{k=0}^{N-1} \left[ \norm{w_k^{(\tau)}}^2 + \norm{f_k^{(\tau)}}^2 \right] \\ &\leq E_h(\eta_0) + C \int_0^{h} \norm{w}^2 + \norm{f}^2 dt
 \end{align*}
 for some $C,c>0$ depending on $h$ but  independent of $\tau$. Further, in the last step we used Jensen's inequality to show
 \begin{align*}
  \tau \sum_{k=0}^{N-1} \norm{w_k^{(\tau)}}^2 = \tau \sum_{k=0}^{N-1} \norm{\fint_{k\tau}^{(k+1)\tau} w dt}^2 \leq \tau \sum_{k=0}^{N-1} \fint_{k\tau}^{(k+1)\tau} \norm{ w}^2 dt = \int_0^{N\tau} \norm{w}^2 dt
 \end{align*}
 and a similar estimate for $f$. In particular this also allows us to apply \autoref{prop:shortInjectivity} to show that $\eta_k$ is always injective and the Euler-Lagrange equation is well defined.
 
 If we now define the piecewise constant and piecewise affine approximations
 \begin{align*}
  \eta^{(\tau)}(t) &= \eta_k &&\text{ for } k\tau \leq t < (k+1)\tau\\
  \tilde{\eta}^{(\tau)}(t) &= \left(\frac{t}{\tau}-k\right) \eta_{k+1} + \left(k+1-\frac{t}{\tau}\right) \eta_k  &&\text{ for } k\tau \leq t < (k+1)\tau\\
  \intertext{ where in particular }
  \partial_t \tilde{\eta}^{(\tau)}(t) &= \frac{\eta_{k+1} - \eta_k}{\tau}  &&\text{ for } k\tau < t < (k+1)\tau
 \end{align*}
 our energy estimate turns into an uniform (in $\tau$ and $t$) bound on $E_h(\eta^{(\tau)}(t))$, as well as a uniform (in $\tau$) bound on $\int_0^{h}R_h(\eta^{(\tau)},\partial_t \tilde{\eta}^{(\tau)}) +c\norm{\partial_t \tilde{\eta}^{(\tau)}}^2 dt$. Now using the properties of energy and dissipation from our assumptions, this gives an uniform $L^\infty([0,h];W^{k_0,2}(Q;\R^n))$ bound on $\eta^{(\tau)}$ and $\tilde{\eta}^{(\tau)}$ as well as a uniform $L^2([0,h];W^{k_0,2}(Q;\R^n))$ bound on $\partial_t \tilde{\eta}^{(\tau)}$. 
 
Analogously to \autoref{prop:QSconvergence}, we may extract a converging subsequence and a single limit $\eta\in W^{1,2}(0,T;W^{k_0,2}(Q))\cap C^0([0,T];C^{1,\alpha}(Q))$. In particular we gain 
 \begin{align*}
   \tilde{\eta}^{(\tau)} &\rightharpoonup \eta \text{ in } W^{1,2}(0,T;W^{k_0,2}(Q;\R^n)) \\
     \eta^{(\tau)} &\rightharpoonup^* \eta \text{ in } L^\infty(0,T;W^{k_0,2}(Q;\R^n)) \\
    \tilde{\eta}^{(\tau)} &\to \eta\text{ in } L^\infty([0,T];C^{1,\alpha^-}(Q;\R^n))
      \\
   \eta^{(\tau)} &\to \eta\text{ in } L^\infty([0,T];C^{1,\alpha^-}(Q;\R^n))
 \end{align*}
 for all $0<\alpha^-< \alpha := 1-\frac{n}{q}$.
 
 This is already enough to pass to the limit in most of the terms in the Euler-Lagrange equation \eqref{eq:SOdiscreteELequation}. Only for the energy term, we need to show strong convergence of $\eta^{(\tau)}(t)$ in $W^{2,q}(Q;\R^n)$ using a Minty argument similar to the beginning of the proof of \autoref{prop:QSconvergence}. We thus again pick $\psi \in C_0^\infty(Q;\R^n)$ and test with $(\eta^{(\tau)}-\eta)\psi$ resulting in
 \begin{align*}
   &\phantom{{}={}}\inner{DE_h(\eta^{(\tau)})-DE_h(\eta)}{(\eta^{(\tau)}-\eta)\phi} = -\inner{DE_h(\eta)}{(\eta^{(\tau)}-\eta)\phi}\\
   &- \inner{
   D_2R_h(\eta^{(\tau)},\partial_t \tilde{\eta}^{(\tau)})}{(\eta^{(\tau)}-\eta)\phi} + \inner{f\circ \eta}{(\eta^{(\tau)}-\eta)\phi} + \frac{\rho}{h} \inner{\partial_t \tilde{\eta}^{(\tau)} - w }{(\eta^{(\tau)}-\eta)\phi}.
  \end{align*}
  Here the difference to the proof of \autoref{prop:QSconvergence} is only in the last term but this also converges to $0$, as $\eta^{(\tau)}-\eta\to 0$ strongly and $\partial_t\tilde{\eta}-w$ is uniformly bounded. Then, as before by assumption S6, $\eta^{(\tau)}(t)$ converges strongly in $W^{k_0,2}(Q;\R^n)$ for almost all $t\in[0,h]$ and thus we have
  \begin{align*}
   0&=\int_0^h \inner{DE_h(\eta^{(\tau)})}{\phi} + \inner{D_2R_h(\eta^{(\tau)},\partial_t \tilde{\eta}^{(\tau)})}{\phi} - \inner{f\circ\eta^{(\tau)}}{\phi} +\frac{\rho}{h} \inner{\partial_t \tilde{\eta}^{(\tau)} - w}{\phi} dt\\
   &\to\int_0^h \inner{DE_h(\eta)}{\phi} + \inner{D_2R_h(\eta,\partial_t \eta)}{\phi} - \inner{f\circ\eta}{\phi} +\frac{\rho}{h} \inner{\partial_t \eta - w}{\phi} dt
  \end{align*}
  which proves \autoref{thm:SOtimeDelayedExistence}.
 \end{proof}

\subsubsection*{Time-delayed energy inequality}

In the proof of \autoref{thm:SOtimeDelayedExistence} we already gave an initial, somewhat crude energy estimate on the discrete level. Now that we have a solution of the time-delayed equation, we can however give the much stronger, ``physical'' energy inequality, which will turn out to be crucial in what follows.

\begin{lemma}[Time-delayed energy inequality for the solid] \label{lem:SOtimeDelayedEnergyInequality}
Let the deformation $\eta \in L^\infty([0,h] \times Q; \mathcal{E}) \cap W^{1,2}([0,h] \times Q; \R^n)$ be a weak solution to the time-delayed equation \autoref{def:SOtimeDelayedSolution}. Then for all $t \in [0,h]$, we have
\begin{align*} 
  E_h(\eta(t)) + \frac{\rho}{2h} \int_0^t \norm{\partial_t \eta}^2 dt +  \int_0^t 2R_h(\eta,\partial_t \eta) dt  \leq E_h(\eta_0) + \frac{\rho}{2h} \int_0^t \norm{w}^2 dt + \int_0^t \inner{f\circ \eta }{\partial_t \eta}dt.
  \end{align*} 
\end{lemma}

\begin{proof}
 We use $\chi_{[0,t]}\partial_t \eta$ as a test function in the weak equation.\footnote{Note that this is the point where we rely on $R_h$, since to test $DE(\eta)$, we need $\phi \in L^2([0,T];W^{2,q}(Q;\R^n))$, but bounding $R(\eta,\partial_t \eta)$ only gives us a $L^2([0,T];W^{1,2}(Q;\R^n))$ bound. See also \autoref{rem:SOenergyInequality}.} From this we get
 \begin{align*}
  0&=\int_0^t \inner{DE_h(\eta)}{\partial_t \eta} + \inner{D_2R_h(\eta,\partial_t \eta)}{\partial_t \eta} - \inner{f\circ \eta}{\partial_t \eta} + \frac{\rho}{h}\inner{\partial_t \eta - w}{\partial_t \eta} dt \\
  &= E_h(\eta(t)) - E_h(\eta(0)) + \int_0^t  2R_h(\eta,\partial_t \eta) - \inner{f\circ \eta}{\partial_t \eta} + \frac{\rho}{h}\inner{\partial_t \eta - w}{\partial_t \eta} dt
 \end{align*}
 where we in particular used that $\inner{D_2R_h(\eta,\partial_t \eta)}{\partial_t \eta} = 2R_h(\eta,\partial_t \eta)$ by the quadratic nature of $R_h$. Finally we use Young's inequality on the last term in the form of
 \begin{align*}
  \inner{\partial_t \eta - w}{\partial_t \eta} = \norm{\partial_t \eta}^2 - \inner{w}{\partial_t \eta} \geq \norm{\partial_t \eta}^2 - \frac{1}{2}\norm{\partial_t \eta}^2 - \frac{1}{2}\norm{w}^2 = \frac{1}{2}\norm{\partial_t \eta}^2 - \frac{1}{2}\norm{w}^2.
 \end{align*}
 Reordering the terms then closes the proof.
\end{proof}

\subsection{Proof of \autoref{thm:SOexistence}} \label{subsec:SOproof}

We will start the proof of the theorem by directly using its two key ingredients, the two results from the previous section. First we iteratively use the existence of time-delayed solutions on the short intervals $[0,h]$ to construct a time-delayed solution on the longer interval $[0,T]$.

\subsubsection*{Iterated time-delayed solutions and energy estimates}

For fixed $h$ we start with given initial deformation $\eta_0 \in \mathcal{E}$ and we use the initial velocity as a constant right hand side $w_0(t) = \eta'$ for $t \in [0,h]$. Now given $\eta_l \in \mathcal{E}$ and $w_l \in L^2([0,h]\times Q;\R^n)$, we find a solution $\tilde{\eta}_{l+1} \in L^\infty([0,h] \times Q; \mathcal{E}) \cap W^{1,2}([0,h] ;W^{k_0,2}(Q; \R^n))$ to the time-delayed equation using \autoref{thm:SOtimeDelayedExistence}. We then set $\eta_{l+1} = \tilde{\eta}_{l+1}(h)$ and $w_{l+1} = \tilde{\eta}_{l+1}$ as data for the next step for which they are admissible by \autoref{lem:SOtimeDelayedEnergyInequality}.

From these ingredients we construct $\eta^{(h)}: [0,T] \times Q \to \R^n$ using
\begin{align*}
 \eta^{(h)}(t,x) := \tilde{\eta}_{l+1}(t-hl) \text{ for } hl \leq t \leq h(l+1).
\end{align*}

Thus, directly from the definition we see that $\eta^{(h)}$ fulfills
\begin{align} \label{eq:SOhWeakEquation}
 0 &= \int_0^T \inner{DE_h(\eta^{(h)}(t))}{\phi} + \inner{D_2R_h(\eta^{(h)}(t),\partial_t \eta^{(h)}(t))}{\phi} \\ \nonumber
 &- \inner{f \circ \eta^{(h)}(t)}{\phi} + \frac{\rho}{h}\inner{\partial_t \eta^{(h)}(t) - \partial_t \eta^{(h)}(t-h)}{\phi} dt.
\end{align}
Furthermore, exploiting the energy inequality (\autoref{lem:SOtimeDelayedEnergyInequality}) yields
\begin{align*}
  &\phantom{{}={}} E_h(\eta^{(h)}((l+1)h)) + \frac{\rho}{2} \fint_{lh}^{(l+1)h} \norm{\partial_t \eta^{(h)}}^2 dt +  \int_{lh}^{(l+1)h} 2R_h(\eta^{(h)},\partial_t \eta^{(h)}) dt  \\ \nonumber
  &\leq E_h(\eta^{(h)}(lh)) + \frac{\rho}{2} \fint_{(l-1)h}^{lh} \norm{\partial_t \eta^{(h)}}^2 dt + \int_{lh}^{(l+1)h} \inner{f\circ \eta}{\partial_t \eta^{(h)} }dt.
\end{align*}
Taking $t\in [lh,(l+1)h]$, we find
after summing the above over $1,...,l$,  and adding the energy inequality for $\tilde{\eta}^{l+1}$ in \autoref{lem:SOtimeDelayedEnergyInequality} the following crucial estimate:
\begin{align} \label{eq:SOhEnergyInequality}
  (E):&=E_h(\eta^{(h)}(t)) + \frac{\rho}{2} \fint_{t-h}^{t} \norm{\partial_t \eta^{(h)}}^2 dt +  \int_{0}^{t} 2R(\eta^{(h)},\partial_t \eta^{(h)}) dt  \\ \nonumber
  &\leq E_h(\eta_0) + \frac{\rho}{2} \norm{\eta'}^2 + \int_{0}^{t} \inner{f \circ \eta}{\partial_t \eta^{(h)} }dt
\end{align}
for all $t \in [0,T]$.

Now, as before, we need to estimate the force term using Young's inequality. This gives
\begin{align*}
(E)&\leq E_h(\eta_0) - E_{min} + \frac{\rho}{2} \norm{\eta'}^2 + \frac{t}{2\delta} \norm[L^\infty]{f}^2 + \frac{\delta}{2} \int_0^t \norm{\partial_t \eta^{(h)} }^2dt;
\end{align*}
here recall that $E_{min}$ is defined in Assumption (S1).
As all terms involving $t$ on the right hand side have a fixed sign, we extend to $t=T$ and find 
\begin{align*}
(E) \leq C_0 + C_1 \frac{T}{\delta} + \frac{\delta}{2} \int_{0}^{T} \norm{\partial_t \eta^{(h)} }^2ds
\end{align*}
for some constants $C_0,C_1$ resulting from the given data and independent of $h$. Dropping the positive terms involving $E$ and $R_h$ on the left-hand side, multiplying by $h$ and adding up implies
\begin{align*}
 \frac{\rho}{2} \int_0^T \norm{\partial_t \eta^{(h)}}^2 ds = \sum_{l=0}^N \frac{\rho}{2} \int_{lh}^{(l+1)h} \norm{\partial_t \eta^{(h)}}^2 ds \leq hN \left( C_0 + C_1 \frac{T}{\delta} + \frac{\delta}{2} \int_{0}^{T} \norm{\partial_t \eta^{(h)} }^2ds \right)
\end{align*}
for $hN = T$.\footnote{There is no need to assume that $T$ is a multiple of $h$, but we will do so for the sake of simplification.} Now choosing $\delta := \frac{\rho}{2 T}$  allows us to absorb the integral on the right hand side to the left and we end up with an uniform estimate of the form
\begin{align*}
 \frac{\rho}{4} \int_0^T \norm{\partial_t \eta^{(h)}}^2 ds \leq T C_0 + C_2' T^2,
\end{align*}
which implies also that
\[
(E)\leq T C_0 + C_2' T^2%
\]
Note that in contrast to the parabolic setup from the last section, up to this point there was no need to apply Korn's inequality. In particular, as we used the inertial term to estimate the force term, we obtain a uniform bound on the energy without exploiting the dissipative terms; i.e. we already know that $\sup_{t\in [0,T]} E(\eta^{h}(t))\leq  T C_0 + C_2' T^2$.  Now, using this estimate, we may apply \autoref{lem:globalKorn} without restrictions on the final time $T$, to find that 
\begin{align}
\label{eq:apriSOh}
&\sup_{t\in [h,T]}\bigg(\fint_{t-h}^h\norm{\partial_t\eta^{(h)}}^2\, ds+E(\eta^{(h)}(t))+ h^{a_0} \norm{\nabla^{k_0} \eta^{(h)}(t)}^2\bigg) \leq C  \text{and }\\ \nonumber
&\int_0^T\norm[W^{1,2}(Q)]{ \partial_t\eta^{(h)}}^2+h\norm[W^{k_0,2}(Q)]{\partial_t\eta^{(h)}}^2\, ds \leq C
\end{align}
are uniformly bounded with constant $C=C(T)$ independent of $h$. Moreover, it allows to conclude that $\eta^{(h)}(t)$ is always injective by \autoref{prop:shortInjectivity}. 

By the same arguments as used before, we can now choose a subsequence which converges to a limit function $\eta\in C_w(0,T;W^{2,q}(Q))\cap W^{1,2}(0,T;W^{1,2}(Q))\cap C^0([0,T];C^{1,\alpha}(Q))$. In particular we gain 
 \begin{align*}
   \eta^{(h)} &\rightharpoonup \eta \text{ in } W^{1,2}([0,T];W^{1,2}(Q;\R^n)) \\
     \eta^{(h)} &\rightharpoonup^* \eta \text{ in } L^\infty([0,T];W^{2,q}(Q;\R^n)) \\
    \eta^{(h)} &\to \eta\text{ in } C^0([0,T];C^{1,\alpha^-}(Q;\R^n))
 \end{align*}
 for all $0<\alpha^-< \alpha := 1-\frac{n}{q}$.
 Moreover, the weak lower semi-continuity implies that
 \begin{align}
\label{eq:apriSO}
\sup_{t\in [0,T]}\big(\norm{\partial_t\eta(t)}^2+E(\eta(t))\big) \leq C \text{ and }\int_0^T\norm[W^{1,2}(Q)]{\partial_t\eta}^2\, ds \leq C
\end{align}
with the same constant as before.

\subsubsection*{Improving convergence}

Our final goal is to prove convergence of the weak equation \autoref{def:SOtimeDelayedSolution} satisfied by the time-delayed approximation $\eta^{(h)}$ to the hyperbolic inertial equation we are interested in. The crucial term here is $DE(\eta^{(h)})$ which requires strong convergence of $\eta^{(h)}$ in $W^{2,q}(Q;\R^n)$. For this we want to use the Minty-type property of the Energy, which in turn requires us to improve convergence of the other terms in the equation. We achieve this by an Aubin-Lions type-lemma, but first we will need another estimate on the discrete difference quotient:

\begin{lemma}[Length $h$ bounds (solid)] \label{lem:SOh-m-estimate} %
 Fix $T > 0$. Then there exists a constant $C$ depending only on the initial data and $T$, such that for $k_0 > 2+\frac{(q-2)n}{2q}$ the following holds:
 \begin{align*}
  \int_0^T \norm[W^{-k_0,2}(Q)]{\frac{\partial_t \eta^{(h)}(t)-\partial_t \eta^{(h)}(t-h)}{h}}^2 dt \leq C
 \end{align*}
 where $\partial_t \eta$ is extended by $\eta'$ for negative times.
\end{lemma}

\begin{proof}
  Pick $\phi \in C_0^\infty(Q;\R^n)$. Then we have using the time-delayed equation
  \begin{align*}
   &\phantom{{}={}}\rho_s\abs{\inner[Q]{\frac{\partial_t \eta^{(h)}(t)-\partial_t \eta^{(h)}(t-h)}{h}}{\phi}}
   \leq \abs{\inner{DE(\eta^{(h)}(t))}{\phi}} + h^{a_0}\abs{\inner{\nabla^{k_0} \eta^{(h)}}{\nabla^{k_0} \phi}} \\
   &+ \abs{\inner{D_2R(\eta^{(h)}(t),\partial_t \eta^{(h)}(t))}{\phi}} + h\abs{\inner{\nabla^{k_0} \partial_t\eta^{(h)}}{\nabla^{k_0} \phi}} + \abs{\inner[Q]{f(t)}{\phi}} \\
   &\leq \left(\norm[W^{-2,q}(Q)]{DE(\eta^{(h)}(t))} + h^{a_0} \norm[Q]{\nabla^{k_0} \eta^{(h)}(t)} + \norm[\infty]{f}  \right) \norm[W^{k_0,2}(Q)]{\phi} \\
   &+ \left(\norm[W^{-1,2}(Q)]{D_2R(\eta^{(h)}(t),\partial_t \eta^{(h)}(t))} + h \norm[Q]{\nabla^{k_0} \partial_t \eta^{(h)}(t)}\right)\norm[W^{k_0,2}(Q)]{\phi}
  \end{align*}
  Now for the first set of terms, we note that they are uniformly bounded by \autoref{ass:energy}, S5 and \eqref{eq:apriSOh}. For the second set, we note that the quadratic growth of $R(\eta,.)$ in $W^{1,2}(Q;\R^n)$ implies a linear growth of $D_2R$ thus equally \eqref{eq:apriSOh} implies boundedness when integrated in time.
\end{proof}

Note that in the previous lemma the $h$, by which time is shifted, is the same $h$ as in the sequence. Thus even though $\partial_t \eta^{(h)}$ is already continuous, we can only ever compare at fixed distances in the form of multiples of $h$. This is an unavoidable consequence of the way we obtain this estimate using the equation. In particular this does not allow us to show that $\partial_t \eta^{(h)}$ converges strongly in $L^2$ as one would normally do. Instead we will show convergence of averages of Length $h$, which turn out to be much more natural in this context, in particular as they also occur in the energy inequality.

\begin{lemma}[Aubin-Lions (solid)] \label{lem:SOAubinLions} %
 Let $b^{(h)}: t \mapsto \fint_{t}^{t+h} \partial_t \eta^{(h)}ds$ We have (for a subsequence $h \to 0$)
 \begin{align*}
   b^{(h)} \rightarrow \partial_t \eta \text{ in }C^0([0,T];L^2(Q;\R^n)).
 \end{align*}
\end{lemma}

\begin{proof}
By the fundamental theorem of calculus we have
\begin{align*}
\partial_t b^{(h)} = \frac{\partial_t \eta^{(h)}(t+h)-\partial_t \eta^{(h)}(t)}{h}.
\end{align*}
Now $b^{(h)}$ is uniformly bounded in $L^\infty([0,T];W^{1,2}(Q;\R^n))$ by the energy estimate and $\partial_t b^{(h)}$ is uniformly bounded in $L^2([0,T];W^{-k_0,2}_0(Q;\R^n))$ by the previous lemma. Thus we can apply the classical Aubin-Lions lemma~\citation{simon1986compact}, yielding the existence of a converging subsequence in $C^0([0,T];L^2(Q;\R^n))$. 
It remains to associate the limit function with $\partial_t\eta$. For that take $h_0>0$ and $\phi\in C^\infty_0([h_0,T-h_0]\times Q)$, we find for all $h\in (0,h_0)$ by the weak convergence of $\partial_t\eta^{(h)}\rightharpoonup \partial_t\eta$ (and the Lebesgue point theorem) that
\begin{align*}
\int_0^T\inner[Q]{b^{(h)}}{\phi}\, dt 
&=\fint_{0}^{h}\int_0^T\inner[Q]{\partial_t\eta^{(h)}(t+s)}{\phi(t)}\, dt\, ds
=\fint_{0}^{h}\int_0^T\inner[Q]{\partial_t\eta^{(h)}(\tau)}{\phi(\tau-s)}\, d\tau\, ds
\\
&\to \int_0^T\inner[Q]{\partial_t\eta(\tau)}{\phi(\tau)}\, d\tau. \qedhere
\end{align*}
\end{proof}

Finally we will use a Minty-type argument to improve convergence a bit further.
\begin{lemma}[Minty-Trick]
\label{lem:SOMinty}
$\eta^{(h)}(t) \to \eta(t)$ strongly in $L^q(0,T;W^{2,q}(Q;\R^n))$ for almost all $t\in[0,T]$.
\end{lemma}
\begin{proof}
As in the last section we will rely on  \autoref{ass:energy}, S6. 
 Let $h_0>0$ and $h\in (0,h_0)$. Further take $\phi \in C_0^\infty((h_0,T-h_0)\times Q; \R^+)$ with $\dist(\supp(\phi),\partial Q)>h_0$. Accordingly we define for $\delta_h=h^{a_1}<h_0$
the approximation $\eta_h=(\eta\chi_{[0,T]\times Q})*\psi_{\delta_h}$, where $\psi_\delta$ is the standard convolution kernel in time-space. This implies that $(\eta^{(h)}-\eta_{\delta_h})\phi$ is a valid test function for \eqref{eq:SOhWeakEquation}. Moreover, we find that by the standard convolution estimates that
\begin{align}
\label{eq:convest}
\norm[W^{k_0,2}]{\eta_{\delta_h}\phi}\leq ch^{a_1(2-\frac{q}{n}-k_0+\frac{2}{n})}\norm[W^{2,q}(Q)]{\eta}\text{ and }\norm[W^{k_0,2}]{\partial_t\eta_{\delta_h}\phi}\leq ch^{(1-k_0)a_1}\norm[W^{1,2}(Q)]{\partial_t\eta},
\end{align}
and strong convergence in all norms in which $\eta$ is bounded. Now we calculate
\begin{align*}
0&\leq \limsup_{h\to 0} \int_0^T \inner{DE(\eta^{(h)}(t))-DE(\eta(t))}{(\eta^{(h)}-\eta)\phi}dt
\\
&=\limsup_{h\to 0}\int_0^T \inner{DE(\eta^{(h)}(t))}{(\eta^{(h)}-\eta_{\delta_h})\phi} \,dt+ \limsup_{h\to 0} \inner{\underbrace{DE(\eta^{(h)}(t))}_{\text{bd. in } W^{-2,q}(Q)}}{\underbrace{(\eta-\eta_{\delta_h})\phi}_{\to 0 \text{ in } W^{2,q}(Q)}} dt
\\
&=\limsup_{h\to 0}\int_0^T \inner{DE_h(\eta^{(h)}(t))}{(\eta^{(h)}-\eta_{\delta_h})\phi}
-2h^\frac{a_0}{2}\inner{h^\frac{a_0}{2}\nabla^{k_0}(\eta^{(h)}(t))}{\nabla^{k_0}(\eta^{(h)}-\eta_{\delta_h})\phi} dt
\\
&\leq \limsup_{h\to 0}\int_0^T \inner{DE_h(\eta^{(h)}(t))}{(\eta^{(h)}-\eta_{\delta_h})\phi} dt
+2h^{a_0}\int_0^T \norm{\nabla^{k_0}(\eta^{(h)}(t))}\norm[W^{k_0,2}]{\eta_{\delta_h}\phi} dt
\\
&\leq \limsup_{h\to 0}\int_0^T \inner{DE_h(\eta^{(h)}(t))}{(\eta^{(h)}-\eta_{\delta_h})\phi} + ch^{\frac{a_0}{2}-(2-\frac{q}{n}-k_0+\frac{2}{n})a_1} dt
\\
&=\limsup_{h\to 0}\int_0^T \inner{DE_h(\eta^{(h)}(t))}{(\eta^{(h)}-\eta_{\delta_h})\phi} dt
\end{align*}
by \eqref{eq:convest} and by choosing $a_1$ small enough.
The final term then can be estimated using the equation, as
\begin{align*}
 &\phantom{{}={}}\int_0^T \inner{DE_h(\eta^{(h)}(t))}{(\eta^{(h)}-\eta_{\delta_h})\phi} dt= - \inner{D_2R_h(\eta^{(h)}(t),\partial_t \eta^{(h)}(t))}{(\eta^{(h)}-\eta_{\delta_h})\phi} \\
 &+ \inner{f \circ \eta^{(h)}(t)}{(\eta^{(h)}-\eta_{\delta_h})\phi} + \frac{\rho_s}{h}\inner{\partial_t \eta^{(h)}(t) - \partial_t \eta^{(h)}(t-h)}{(\eta^{(h)}-\eta_{\delta_h})\phi} dt.
\end{align*}
Here all terms are converging.
 In particular observe that
\begin{align*}
& \inner{D_2R_h(\eta^{(h)}(t),\partial_t \eta^{(h)}(t))}{(\eta^{(h)}-\eta_{\delta_h})\phi}
=\inner{D_2R(\eta^{(h)}(t),\partial_t \eta^{(h)}(t))}{(\eta^{(h)}-\eta_{\delta_h})\phi}
\\
&\quad +2h\inner{\nabla^{k_0}\partial_t\eta^{(h)}}{\nabla^{k_0}((\eta^{(h)}-\eta_{\delta_h})\phi)}\to \inner{D_2R(\eta(t),\partial_t \eta(t))}{(\eta-\eta_{\delta_h})\phi}
\end{align*}
by the strong convergence of $\eta^{(h)}$ in $W^{1,2}(Q;\R^n)$, the weak convergence of $\partial_t\eta^{(h)}$ in $W^{1,2}(Q;\R^n)$ and since
\[
h\abs{\inner{\nabla^{k_0}\partial_t\eta^{(h)}}{\nabla^{k_0}((\eta^{(h)}-\eta_{\delta_h})\phi)}}\leq h^{\frac{1}{2}-\frac{a_0}{2}}\norm{\sqrt{h}\nabla^{k_0}\partial_t\eta^{(h)}}\norm{h^\frac{a_0}{2}\nabla^{k_0}(\eta^{(h)}-\eta_{\delta_h})\phi)}
\]
which converges to zero a.e.\ using the energy estimates and \eqref{eq:convest} by choosing $a_0<1$ and $a_1<1$ accordingly. The term including the right-hand side converges, since all terms involve converge strongly. For the last term, we do a discrete partial integration in time (i.e. shift the term involving $t-h$) to get
\begin{align*}
 &\phantom{{}={}}\int_0^T \frac{\rho_s}{h}\inner{\partial_t \eta^{(h)}(t) - \partial_t \eta^{(h)}(t-h)}{(\eta^{(h)}(t)-\eta_{\delta_h}(t))\phi(t)} dt \\
 &= -\rho_s 
 \int_0^T \inner{\partial_t \eta^{(h)}(t)}{\left(\frac{\eta^{(h)}(t+h)-\eta^{(h)}(t)}{h}-\frac{\eta_{\delta_h}(t+h)-\eta_{\delta_h}(t)}{h}\right)\phi(t+h)} dt
 \\
 &\quad -\rho_s 
 \int_0^T \inner{\partial_t \eta^{(h)}(t)}{\left(\eta^{(h)}(t)-\eta_{\delta_h}(t)\right)\frac{\phi(t+h)-\phi(t)}{h}} dt.
\end{align*}
Now note that the first difference quotient is equal to $w^{(h)}$ as it was defined in \autoref{lem:SOAubinLions} and thus converges strongly to $\partial_t \eta$ in $L^2([0,T]\times Q;\R^n)$, while the other difference quotients only involve constant functions and their mollifications and thus also converge in the same space. As a result, all the right hand sides converge strongly to $0$ in $L^2([0,T]\times Q;\R^n)$ and the left hand sides are bounded. Thus the total limit is 0 and via \autoref{ass:energy}, S6, we have $\eta^{(h)}(t) \to \eta(t)$ in $W^{2,q}(Q;\R^n)$ for almost all $t\in[0,T]$.
\end{proof}

\subsubsection*{Existence of the limit equation}

With this in hand, we can finally consider the weak equation \eqref{eq:SOhWeakEquation} for arbitrary test functions. For the first three terms we have, as before (for the regularizing terms vanish per the estimates in the last lemma)
\begin{align*}
 &\phantom{{}={}} \int_0^T \inner{DE_h(\eta^{(h)}(t))}{\phi} + \inner{D_2R_h(\eta^{(h)}(t),\partial_t \eta^{(h)}(t))}{\phi} + \inner{f \circ \eta^{(h)}(t)}{\phi} dt \\
 &\to \int_0^T \inner{DE(\eta(t))}{\phi} + \inner{D_2R(\eta(t),\partial_t \eta(t))}{\phi} + \inner{f \circ \eta(t)}{\phi} dt.
\end{align*}

This leaves us with the last term, where we shift the discrete derivative to the test function again and get
\begin{align*}
 &\int_0^T\frac{\rho}{h}\inner{\partial_t \eta^{(h)}(t) - \partial_t \eta^{(h)}(t-h)}{\phi} dt = 
 - \rho \int_0^T \inner{\partial_t \eta^{(h)}(t)}{\frac{\phi(t+h)-\phi(t)}{h}} dt 
 \\
 &\quad\to - \rho \int_0^T \inner{\partial_t \eta(t)}{\partial_t \phi} dt.
\end{align*}

From this, we get solutions on the interval $[0,T]$.

\subsubsection*{Continuation until collision}

Using the short term existence, we can now employ the usual continuation argument. Assume that $\eta: [0,T_{\max}) \to \mathcal{E}$ is a solution on a maximal interval. Then either $T_{\max} = \infty$ or we can use the energy inequality to show existence of a unique limit $\eta(T_{\max})$ similar to as we did at the end of \autoref{thm:QSexistence}. Then  $\eta(T_{\max}) \notin \partial \mathcal{E}$, would allow us to reapply the short time existence, which would be a contradiction. This finishes the proof of \autoref{thm:SOexistence}.

\subsection{Remarks} \label{subsec:SOremarks}

Let us close this section with some remarks on the preceeding proofs.

\begin{remark}[On the need for dissipation] \label{rem:SOdissipation}
It is noteworthy, that the a-priori estimates \eqref{eq:apriSO} are valid even in case of a purely elastic solid, which means in case $R\equiv 0$. In that case no strategy is known to deal with the non-linearity in $DE(\eta)$ without resorting to a relaxations of $\eta$. Even for the hyperbolic p-Laplacian $\partial_t^2\eta-\diver(\abs{\nabla \eta}^{p-2}\nabla \eta)=0$ the existence of weak solutions is a long standing open problem. Only in the case $p=2$, where the elastic energy is quadratic in highest order and weak convergence suffices, existence of solutions is known.

In our case, in order to apply the Minty method, we use the strong convergence of $b^{(h)}$ in an $L^2$ sense. This convergence can be derived from the Aubin-Lions lemma, in case we $\partial_t \eta^{(h)}(t)$ is in a space that compactly embeds into $L^2$. This is true in our set up due to the dissipation and would not be possible in the case the respective bounds would be missing.
A possible escape in the latter case could be the use of measure valued solutions. %
\end{remark}

\begin{remark}[On collisions and continuation afterwards] \label{rem:SOcollisions}
 For the time delayed problem we face the same difficulty as in the previous section. While the limit object $\eta$ exists for all times independently of any collision, to get to a limit equation, we need to have an appropriate Euler-Lagrange equation for all $\eta^{(\tau)}_k$. This we only have if $\eta^{(\tau)}_k \notin \partial \mathcal{E}$, i.e.\ if there is no self-touching of the solid.
 
 To progress past this point, a closer study of collisions and self-touching is needed. The key observation here is that the initial discrete minimization forms a classical obstacle problem. As a consequence, we have to replace the Euler-Lagrange equation with a variational inequality, since we are only able to test with directions $\phi$ which do not point ``out of'' $\mathcal{E}$. Showing full convergence for such an inequality is difficult, as the set of admissible directions $\phi$ can change along the sequence.
 
 Assuming that this can be done, the next question then is, if one can then continue from the time-delayed to the hyperbolic problem. As mentioned, the key to this is the appropriate energy-inequality. While as per definition the chosen test function $\partial_t \eta$ is admissible, as it can never point out of $\mathcal{E}$, the variational inequality has the opposite sign. Here one needs to show that it is in fact an equality in this special case. Only then, we are again guaranteed existence of a limit object $\eta$, but also again face the same difficulty in establishing the limit equation.
 
 Since the main focus of this paper is the fluid structure interaction, we will not go into further details on this at the time being. In particular, collisions of the solid lead to non-Lipschitz boundaries for the fluid and there is a general conjecture that an incompressible fluid will prevent collisions in the first place. 
 
 It should also be noted that recent progress into the quasi-stationary analogue of this has been made by \cite{kromerQuasistaticViscoelasticitySelfcontact2019} with a different approach, where instead of a variational inequality, they consider the associated Lagrange-multiplyier and its physical significance.
\end{remark}

\begin{remark}[On more general energies] \label{rem:SOgeneralEnergies}
As was already mentioned, our assumptions on the solid are specifically tailored to the general fluid-structure interaction problem handled in the next section. The method of solving hyperbolic problems we described in this section however has the potential for many more applications.
 
Of particular interest to elasticity is the problem of injectivity. In the approach of this paper, the injectivity properties are given through the Ciarlet-Nečas condition and the lower bound on the determinant in S2, and though generally desired in solid mechanics, are entirely optional for the flow of the proof. The method should thus be understood as working despite those assumptions, instead of because of them.
 
 In replacing those, there is however an obvious point where some attention needs to be taken.\footnote{There is also the less obvious point that for the dissipation used in our motivating example, the Korn-type inequality R2 is only true if $\det \nabla \eta$ is bounded away from $0$, but that is specific to the example.} This is the distance to the boundary of the allowed set $\partial \mathcal{E}$. In particular for any fixed $\eta \in \mathcal{E} \setminus \partial \mathcal{E}$, we need to make sure that there is a minimal existence interval, i.e.\ a short time in which no approximate solution can reach $\partial \mathcal{E}$, in order to have a full Euler-Lagrange equation. 
 
 There are two approaches to this. One is to show that the distance is large in terms of the metric derived from the dissipation. In our case we used the energy estimate combined with \autoref{prop:shortInjectivity} to simply show that reaching the boundary in short time simply is to expensive. The other is to show that the energy required to reach the boundary is to large. This in fact we also used, though its use was a bit hidden. Note that $\partial \mathcal{E}$ does not only consist of those configurations for which the solid touches itself, but also of those for which $\det \nabla \eta(x) = 0$ at some point. These we avoided, as they require infinite energy. But one should keep in mind that those configurations are even more difficult to handle than simple collisions. Not only can one not test in certain directions which would result in a negative determinant, but also for some reasonable energies $DE(\eta)$ might be ill-defined in terms of spaces. For details we refer to~\cite{ball2002some}. 
\end{remark}

\begin{remark}[On the proof of the energy inequality] \label{rem:SOenergyInequality}
In the proof of the energy inequality \autoref{lem:SOtimeDelayedEnergyInequality} we used a regularization term in the dissipation to simplify the proof. Since we will need that term later on in the fluid-structure interaction, this only seemed natural, but it should be noted that strictly speaking, it was not necessary. The same result is still true, if we only ever use $R$. To show this directly, one can use some techniques from the theory of minimizing movements, specifically the so called Moreau-Yosida approximation.

For this, in the proof of \autoref{thm:SOtimeDelayedExistence}, we add a third interpolation $\hat{\eta}^{(\tau)}$ where in each step, we solve the minimization problem \eqref{eq:SOdiscreteProblem} with $\sigma := t-k\tau$ in place of $\tau$. Integrating in time over the change of resulting minimum then results in an improved version of \eqref{eq:SOdicreteAPrioriEstimate}, with some additional terms depending on $\hat{\eta}^{(\tau)}$. This turns out to be the proper analogue of \autoref{lem:SOtimeDelayedEnergyInequality} for $\tau > 0$ and using the associated Euler-Lagrange equation allows us to take the limit $\tau \to 0$.

The complete proof is however to involved to carry out at this point, as we will not need it for our particular application.
\end{remark}

 \section{The unsteady fluid-structure interaction problem}
 \label{sec:full}
 
 We will now combine the methods developed in the last two sections and use their combination to show existence for a general fluid-structure interaction problem. In contrast to previous works (see the references in the introduction) we work in arbitrary dimension and consider a bulk solid with its full deformation. But most importantly, we consider the full nonlinear equation, both for the fluid in form of the incompressible Navier-Stokes equation with its transport-term as well as full nonlinear elasticity of the solid.

At this point some of the arguments behind the derivation of the Navier-Stokes equation come into play. The natural way to deal with inertia in a moving fluid domain is to transport it and the natural way to do so is along the flow of the fluid itself; thus giving rise to the desired transport term.
 
 This leads us to perhaps the main technique used of this section and why it is more than a simple combination of the ideas developed in the previous two. What we did in the last section was to introduce a time delayed equation, where effectively the velocity $\partial_t \eta(t)$ (or more technically the linear momentum) was compared with itself at time $t-h$. In this section we will do the same for the fluid. That means we need to compare $v(t)$ and $v(t-h)$ accordingly. However, not only is this not possible directly, as they are defined on different domains, comparing both at the same position in $\Omega$ is also invalid from a conceptional point of view.
 
Again the difference between Lagrangian and Eulerian description is essential here. For the solid, we work in the Lagrangian reference domain, where momentum is a locally conserved quantity. For the fluid, we are forced to consider the Eulerian, physical domain. Here momentum is not conserved locally, but along the flow. This leads us to the notion of a flow map.\footnote{Alternatively we could take the view-point of a series of flow determined, short-term Lagrangian reference configurations for the solid, on each of which, momentum is again conserved locally. In fact not only is this a valid way to consider the problem, it actually is the same calculation with a different interpretation of some of the terms. For details see \autoref{rem:NSlagrangianVsEulerian}
 }
 
 As before, let $\Omega(t)$ denote the fluid domain at a given time $t$. Now for a fixed $t_0$, assume we have a flow map $\Phi^{(t_0)}: [t_0,T] \times \Omega(t_0) \to \Omega$ generated by $v$, i.e. a map for which $\Phi^{(t_0)}(t_0,y) = y$ and $\partial_t \Phi^{(t_0)}(t,y) = v(t,\Phi^{(t_0)}(t,y))$. Existence of such a map is not guaranteed in general. In fact we will spend quite some work constructing it for the time-delayed problem and even then it will be the one object for which we cannot obtain convergence to the limit $h\to 0$.
 
However, if such a map would exist and would be regular enough, it would have certain very useful properties. First remember, that our boundary moves with $v$; then per definition, $\Phi^{(t_0)}|_{[t_0,T] \times \partial \Omega(t_0)}$ never leaves the boundary and in fact is a diffeomorphism between $\partial \Omega(t_0)$ and $\partial \Omega(t)$. Similarly, a short calculation reveals that $\partial_t \det \nabla \Phi^{(t_0)} = (\diver v) \circ \Phi^{(t_0)} = 0$ and $\det \nabla \Phi{(t_0)} (0,.) =  1$, so combining those, we get that $\Phi^{(t_0)}(t,.)$ is a volume preserving diffeomorphism between $\Omega(t_0)$ and $\Omega(t)$. This allows us to validly compare $v(t,\Phi^{(t_0)}(t,y))$ and $v(t_0,y)$ for any $y \in \Omega(t_0)$. But even more telling is what happens in the limit, where we have via the chain rule
\begin{align*}
 \lim_{t\searrow t_0} \frac{v(t,\Phi^{(t_0)}(t,y))-v(t_0,y)}{t-t_0} \to \partial_t v(t_0,y) + \nabla v(t_0,y) \cdot v(t_0,y)
\end{align*}
which is precisely the transport term we expect for the Navier-Stokes equation. So not only have we identified the right terms to construct the time-delayed equation, but they will automatically lead us to the right material time derivative in the Navier-Stokes equation.

Having explained the idea, the rest of this section will be spend with making it rigorous. In particular, much of the added difficulty in solving the time-delayed equation is in constructing the flow map $\Phi$, which needs to be done in parallel, as it is a part of the fluid inertial term in the equation itself. To do so, we will include yet another regularizing term to the dissipation, i.e.\ in addition to $E_h$ and $R_h$ as defined in the previous section, we will add a term of the form $h\norm[\Omega]{\nabla^k_0 u}^2$ to the dissipation of the fluid. For any $h>0$ this will give us some Lipschitz-regularity of $v$ and thus of $\Phi$ which will be crucial when sending the velocity scale $\tau \to 0$.

\subsection{An intermediate, time delayed model}
\label{subsec:NStimeDelayed}

As in the previous section, let us now start with deriving a time-delayed equation, similar to what we did in \autoref{subsec:SOtimeDelayed}.

\begin{definition}[Time delayed solution] \label{def:TDweakSolution}
Let $f\in C^0([0,h]\times \Omega;\R^n)$, $w \in L^2([0,h]\times \Omega;\R^n)$ and $\Omega_0 = \eta_0(Q)$. We call the pair $\eta : [0,h] \times Q \to \Omega, u: [0,h] \times \Omega \to \R^n$ a weak solution to the time delayed inertial equation if 
\begin{align} \label{eq:TDweakSolution} 
0&=\inner[Q]{DE_h(\eta)}{\phi} + \inner[Q]{D_2R_{h}(\eta,\partial_t \eta)}{ \phi} +\inner[Q]{\rho_s \frac{\partial_t \eta - w \circ \eta_0^{-1}}{h} }{ \phi} - \rho_s \inner[Q]{f\circ \eta}{\phi} \\ \nonumber
&+ \nu \inner[\Omega(t)]{\nablasym u}{\nabla \xi} +  h\inner[\Omega(t)]{\nabla^{k_0} u}{\nabla^{k_0}\xi} + \inner[\Omega_0]{\rho_f\frac{u\circ \Phi - w}{h}}{\xi \circ \Phi} - \rho_f \inner[\Omega(t)]{f}{\xi} 
\end{align}
for almost all $t\in[0,h]$ and all $\phi \in C^0([0,h]; W^{k_0,2}(Q;\R^n))$, $\xi \in C^0([0,h];W^{k_0,2}(\Omega;\R^n))$ satisfying $\diver \xi|_{\Omega(t)} =0$, $\xi|_{\partial \Omega} = 0$, $\phi|P = 0$ and the coupling conditions
\[\xi \circ \eta = \phi  \text{ and } u \circ \eta = \partial_t \eta \text{ in }  Q. \] Here we define $\Omega(t) = \Omega \setminus \eta(t,Q)$ and $\Phi:[0,h]\times \Omega_0 \to \Omega$ solves $\partial_t \Phi = u\circ \Phi$ and $\Phi_0(y) = y$.
\end{definition}

While most of this equation is similar to a combination of \autoref{def:QSWeakSol} and \autoref{def:SOtimeDelayedSolution}, the interesting addition here is that of the flow map $\Phi$. Note that in this subsection, this map will always start at $t_0=0$. What this does, is to allow us to take a slightly Lagrangian point of view, as $\Omega_0$ will play the role of a reference configuration for the fluid. 

In particular, in the way that $\Phi$ is linked to the equation, we already need to begin its construction in the discrete setting. This will also be another point where we use the regularized dissipation $R_{h}$. This will afford us easier controls on the boundary values of $v$, which in turn then greatly simplify the construction of $\Phi$. In this section we will prove the following existence theorem:

\begin{theorem}[Existence of time delayed solutions]\label{thm:TDexistence} Let $\eta_0 \in \mathcal{E} \cap W^{k_0,2}(Q;\R^n) \setminus \partial \mathcal{E}$, $w \in L^2([0,h] \times Q;\R^n)$ and $f \in C^0([0,h] \times Q; \R^n)$. Then there exists a solution $(\eta, v)$ to the time delayed equation \autoref{def:TDweakSolution} on the interval $[0,h]$, or there exists a solution on a shorter interval $[0,h_{\max}]$ such that $\eta(h_{\max}) \in \partial \mathcal{E}$.\footnote{Note that a-posteriori (see \autoref{cor:short-time-no-collision})
it will be shown that (in dependence of $\eta_0$) there is always a minimal time-length $h_{\min}$ for which it can be guaranteed that $\eta(t)\notin \partial \mathcal{E}$ for $t\in [0,h_{\min}]$.} Furthermore $\Phi(t,.)$ is a volume preserving diffeomorphism between $\Omega_0$ and $\Omega(t)$. %
\end{theorem}

Let us now begin with the proof of this theorem. The parts that are identical to one of the previous proofs, we will only sketch.

\subsubsection*{Proof of \autoref{thm:TDexistence}, step 1: Constructing an iterative approximation}

Fix a step-size $\tau >0$. We again proceed iteratively, this time constructing both the pair $\eta,v$ as well as $\Phi$. We start with the given $\eta_0$ and $\Phi_0 := id$. Now assume that we have $\eta_k:Q \to \Omega$ and a map $\Phi_k: \Omega_0 \to \Omega_k$, where as before $\Omega_k = \Omega \setminus \eta_k(Q)$. Define $w_k(y) = \fint_{k \tau}^{(k+1)\tau} w(t,y) dt$ for all $y \in \Omega$. We now solve the iterative problem
\begin{align} \label{eq:TDdiscreteProblem}
 (\eta_{k+1},v_{k+1}) &\in \argmin_{\eta,v} E_h(\eta) + \tau R_{h}\left(\eta_k, \frac{\eta- \eta_k}{\tau}\right)
 + \frac{\tau \rho_s}{2h} \norm{\frac{\eta-\eta_k}{\tau}\!\!-w_k\circ \eta_0}^2 - \tau \inner{f\circ \eta}{\frac{\eta-\eta_l}{\tau}} \nonumber \\ &+ \tau\frac{\nu}{2} \norm[\Omega_{k}]{\nablasym v}^2  + \frac{\tau h}{2} \norm[\Omega_k]{\nabla^{k_0} v}^2 + \frac{\tau \rho_f}{2h}\norm[\Omega_0]{v\circ \Phi_{k} -w_k}^2 - \tau\inner[\Omega_{0}]{f \circ \Phi_k}{v \circ \Phi_k}
\end{align}
with $\diver v_{k+1} = 0$ on $\Omega_{k}$ and the coupling condition
\[ v \circ \eta_k = \frac{\eta_{k+1} - \eta_k}{\tau} \text{ on $Q$}.\]
Finally we update $\Phi_k$ to $\Phi_{k+1}$ using
\begin{align*}
 \Phi_{k+1} := (id + \tau v_{k+1}) \circ \Phi_k.
\end{align*}
Note that at this point using the coupling condition we can immediately derive $\Phi_{k+1}(\partial \Omega_0) = \partial \Omega_{k+1}$ but we still need to show that a similar property holds in the interior. This will be done in step 2a of the proof. For now we can simply assume $v_{k+1}$ to be extended by $0$ in the definition of $\Phi_{k+1}$.

\begin{proposition}[Existence of iterative solutions]\label{prop:TDdiscreteELequation}
 The iterative problem \eqref{eq:TDdiscreteProblem} has a minimizer, i.e. $\eta_{k+1}$ and $v_{k+1}$ are defined. Furthermore the minimizers obey the following equation:
 \begin{align*}
&\phantom{{}={}} \inner{DE_h(\eta_{k+1})}{\phi} + \inner{D_2R_{h}\left(\eta_k,\frac{\eta_{k+1}-\eta_{k}}{\tau}\right)}{\phi} +   \frac{ \rho_s}{h}\inner[Q]{\frac{\eta_{k+1}-\eta_k}{\tau}-w_k \circ \eta_0}{\phi}
  \\
  & + \frac{\rho_f}{h}\inner[\Omega_0]{v_{k+1} \circ \Phi_{k} - w_k}{\xi \circ \Phi_{k}} + \nu\inner[\Omega_{k}]{\nablasym v_{k+1}}{\nablasym\xi} + h\inner[\Omega_k]{\nabla^{k_0}v_{k+1}}{\nabla^{k_0} \xi}
  \\
  &=\rho_f\inner[\Omega_0]{ f\circ \Phi_k}{\xi\circ \Phi_k}  +\rho_s\inner[Q]{f\circ\eta_{k}}{\frac{\eta_{k+1}-\eta_k}{\tau}}  
 \end{align*}
 where $\phi \in W^{2,q}(Q;\R^n)$, $\phi|_P = 0$ and $\xi \in W_0^{1,2}(\Omega;\R^n)$ such that \[\phi =\xi \circ \eta_{k} \text{ on $Q$ and } \diver\xi|_{\Omega_{k}} = 0.\]
\end{proposition}

\begin{proof}
 The proof differs from the quasistatic case in \autoref{prop:QSdiscreteELequation} only in the occurence of the additional terms for the effects of inertia. As both are non-negative, we still have a minimizing sequence $\tilde{\eta}_l,\tilde{v}_l$ with the same coercivity-properties. In particular due to the compact embeddings we can assume that for a subsequence both converge in an $L^2$ sense, under which the inertial terms are continuous. Thus this minimizing sequence will again converge to a minimizer. In fact giving the lower bound on the sequence is easier in this case, as the two force terms can now be estimated against the inertial terms directly, without having to resort to a potentially energy dependent Korn-inequality. (See the corresponding calculations the proof of \autoref{thm:SOtimeDelayedExistence} and \autoref{rem:SOdissipation} for more details.)
 
 Further, with regards to the Euler-Lagrange equation, we can treat the additional terms individually. Since both are quadratic functionals of $\eta$ and $v$ respectively, and neither involve any derivatives, this is straighforward. Note that again we are able to remove a factor of $\tau$ from the final term by scaling $\phi$ and $\xi$ differently than $\eta$ and $v$.
\end{proof}

Now equally as before, our minimization can be turned into a discrete energy inequality by comparing minimizers.

\begin{lemma}[Discrete energy inequality and estimates] 
\label{lem:TDdiscreteAPrioriEstimates}
We have
\begin{align*}
 & \phantom{{}={}}E_h(\eta_{k+1}) + \tau R_{h}\left( \eta_{k}, \frac{\eta_{k+1}- \eta_k}{\tau}\right) +  \tau \frac{ \rho_s}{2h} \norm[Q]{\frac{\eta_{k+1}-\eta_k}{\tau}-w_k\circ \eta_0}^2 
 \\&+ \tau \frac{\nu}{2}\norm[\Omega_{k}]{\nablasym v_{k+1}}^2 +\frac{\tau h}{2} \norm[\Omega_k]{\nabla^{k_0}v_{k+1}}^2 + \tau \frac{\rho_f}{2h} \norm[\Omega_0]{v_{k+1} \circ \Phi_{k} - w_k}^2 \\
 &\leq E_h(\eta_k) + \tau \frac{\rho_s}{2h}  \norm[Q]{w_k \circ \eta_0}^2+\tau \frac{\rho_f}{2h} \norm[\Omega_0]{w_k}^2 + \tau \rho_f\inner[\Omega_0]{ f\circ \Phi_k}{v \circ \Phi_k}  +\tau \rho_s\inner[Q]{f\circ\eta_{k}}{\frac{\eta_{k+1}-\eta_k}{\tau}}
\end{align*}
and there exist $c,C>0$ independent of $\tau$ and $N$ such that
\begin{align*}
 &\phantom{{}={}}E_h(\eta_N) + \sum_{k=1}^N \tau \left[R_{h}\left( \eta_{k-1}, \frac{\eta_{k}- \eta_{k-1}}{\tau}\right) + c \norm[Q]{\frac{\eta_{k}-\eta_{k-1}}{\tau}-w_k\circ \eta_0}^2 \right.\\ 
 & \left. +\nu \norm[\Omega_{k-1}]{\nablasym v_k}^2 + \frac{\tau h}{2} \norm[\Omega_{k-1}]{\nabla^{k_0}v_{k}}^2 + c \norm[\Omega_0]{v_k \circ \Phi_{k-1} - w_{k}}^2 \right] \\
 &\leq E_h(\eta_0) + C\left(\int_0^{h} \norm[Q]{w \circ \eta_0}^2 dt + \int_0^{h} \norm[\Omega_0]{w}^2 dt  + \norm[\infty]{f}^2\right).
\end{align*}
\end{lemma}

\begin{proof} 
 As in the quasistationary version (\autoref{lem:QSdiscreteAPrioriEstimates}), we compare the iterative minimizer $(\eta_{k+1},v_{k+1})$ in \eqref{eq:TDdiscreteProblem} with the pair $(\eta,v) = (\eta_k,0)$ to get the first inequality. For the second we sum up all those inequalites for $k\leq N-1$ to end up with 
 \begin{align*}
 &\phantom{{}={}}E_h(\eta_N) + \sum_{k=1}^N \tau \left[ R_{h}\left( \eta_{k-1},\frac{ \eta_k- \eta_{k-1}}{\tau}\right) + \frac{\rho_s}{2h} \norm[Q]{\frac{\eta_{k+1}-\eta_k}{\tau}-w_k\circ \eta_0}^2 \right.\\ 
 & \left. + \frac{\nu}{2} \norm[\Omega_{k-1}]{\nablasym v_k}^2 +\frac{\tau h}{2} \norm[\Omega_k]{\nabla^{k_0}v_{k+1}}^2 + \frac{\rho_f}{2h} \norm{v_k \circ \Phi_{k-1} - w_{k-1}}^2 \right] \\
 &\leq E_h(\eta_0) + \sum_{k=1}^N \tau \left[ \frac{\rho_s}{2h}  \norm[Q]{w_{k-1} \circ \eta_0}^2+  \frac{\rho_f}{2h} \norm[\Omega_0]{w_{k-1}}^2 +  \inner[\Omega_{k-1}]{f}{v_k} +  \inner[Q]{f \circ \eta_{k-1}}{\frac{\eta_k-\eta_{k-1}}{\tau}} \right]
 \end{align*}
 Now using the definition of $w_k$ we note that
 \begin{align*}
  &\phantom{{}={}} \sum_{k=1}^N \tau \norm[\Omega_0]{w_{k-1}}^2 =\sum_{k=1}^N \tau \norm[\Omega_0]{\fint_{\tau(k-1)}^{\tau k} w dt}^2
  \leq \sum_{k=1}^N \tau \fint_{\tau(k-1)}^{\tau k} \norm[\Omega_0]{ w }^2 dt = \int_0^{h} \norm[\Omega_0]{ w }^2 dt.
 \end{align*}
 The same can be done to show $\sum_{k=1}^N \tau  \norm[Q]{w_{k-1} \circ \eta_0}^2 \leq \int_0^{h} \norm[Q]{w \circ \eta_0}^2 dt$.
 We are left to estimate the force terms.
 
 As we did in \autoref{thm:SOtimeDelayedExistence}, we will not use the dissipation, in order to avoid relying on the Korn-inequality. Instead we will again opt to use the inertial terms. The inertial term for the solid we already estimated in that previous proof. Now we similarly estimate
 \begin{align*}
  \abs{\inner[\Omega_{0}]{f \circ \Phi_{k-1}}{v_k \circ \Phi_{k-1}}} \leq \frac{1}{2\delta} \norm[\Omega_0]{f \circ \Phi_{k-1}}^2  + \frac{\delta}{2} \norm[\Omega_{0}]{v_k \circ \Phi_{k-1}}^2 \leq \frac{1}{2\delta} \norm[\infty]{f}^2  + \frac{\delta}{2} \norm[\Omega_{0}]{v_k \circ \Phi_{k-1}}^2.
 \end{align*}
 Now for small $\delta$, the last term can be subsumed into the inertial term. From this the estimate follows.
\end{proof}

As before this immediately implies that for $h$ small enough all $\eta_k$ will be in $\mathcal{E} \setminus \partial \mathcal{E}$.

\subsubsection*{Proof of \autoref{thm:TDexistence}, step 2: Constructing interpolations}

Now we unfix $\tau$ and write the functions of the previous step as $\eta_k^{(\tau)}$, $v_k^{(\tau)}$ and $\Phi_k^{(\tau)}$ to prevent confusion. Using this, we define their interpolated counterparts:
\begin{align*}
 \eta^{(\tau)}(t,x) &= \eta^{(\tau)}_k(x) &\text{ for }& \tau k \leq t < \tau (k+1)\\
 \tilde{\eta}^{(\tau)}(t,x) &= \frac{\tau (k+1)-t}{\tau} \eta^{(\tau)}_k(x) + \frac{t-\tau k}{\tau} \eta^{(\tau)}_{k+1}(x) &\text{ for }& \tau k \leq t < \tau (k+1)\\
 u^{(\tau)}(t,y) &= v_k^{(\tau)}(y) &\text{ for }& \tau k \leq t < \tau (k+1), y \in \Omega_{k}\\
 u^{(\tau)}(t,y) &= \frac{(\eta^{(\tau)}_{k+1}-\eta^{(\tau)}_k) \circ \left(\eta^{(\tau)}_k\right)^{-1}}{\tau}  &\text{ for }& \tau k \leq t < \tau (k+1), y \in \Omega \setminus \Omega_{k}\\
 \Phi^{(\tau)}(t,y) &= \Phi_{k-1}^{(\tau)}(y) &\text{ for }& \tau k \leq t < \tau (k+1)\\
 \tilde{\Phi}^{(\tau)}(t,y) &=\frac{\tau (k+1)-t}{\tau} \Phi^{(\tau)}_{k-1}(x) + \frac{t-\tau k}{\tau} \Phi^{(\tau)}_{k}(x) &\text{ for }& \tau k \leq t < \tau (k+1)
\end{align*}
as well as $\Omega^{(\tau)}(t) = \Omega_{k}$ for $\tau k \leq t < \tau (k+1)$.

Now using the a-priori estimate \autoref{lem:TDdiscreteAPrioriEstimates}, we derive some bounds on those functions.

\begin{lemma}[Uniform bounds in $\tau$] \label{lem:TDuniformBounds}
 The following sequences are uniformly bounded in $\tau$:
\begin{align*}
 E_h(\eta^{(\tau)}(t,.)) & \in L^\infty([0,h])\\
 \eta^{(\tau)}, \tilde{\eta}^{(\tau)} &\in L^\infty([0,h]; W^{k_0,2}(Q;\Omega)) \\
 \partial_t \tilde{\eta}^{(\tau)} & \in L^2([0,h];W^{k_0,2}(Q;\R^n))\\
 u^{(\tau)} &\in L^2([0,h];W^{k_0,2}(\Omega;\R^n))\\
 u^{(\tau)} \circ \Phi^{(\tau)} &\in L^2([0,h]\times \Omega_0; \R^n)
\end{align*}
Furthermore we have per definition
\[\partial_t \tilde{\Phi}^{(\tau)} = u^{(\tau)} \circ \Phi^{(\tau)}\]
whenever $\Phi^{(\tau)}(t,y) \in \Omega^{(\tau)}(t)$ and $t \notin \tau \N$.
\end{lemma}

\begin{proof}
 First we note that the right hand side of the second estimate in \autoref{lem:TDdiscreteAPrioriEstimates} only depends on the initial data $\eta_0$ and $w$ as well as the force $f$. Then this gives us uniform bounds on $E_h(\eta_k)$ and thus an $L^\infty$ bound on $E_h(\eta^{(\tau)}(t,.))$. By the properties of the energy, \autoref{ass:energy} and its regularized version, this also results in a uniform bound on $\norm[W^{k_0,2}(Q)]{\eta_k}$ and thus in $L^\infty([0,h];W^{k_0,2}(Q;\R^n))$ bounds on $\eta^{(\tau)}$ and $\tilde{\eta}^{(\tau)}$. By the properties of the dissipation, \autoref{ass:dissipation} using the bound on the energy, we get
 \begin{align*}
  c_K \int_0^{h} \norm[Q]{\partial_t \nabla \tilde{\eta}^{(\tau)}}^2 + h \norm[Q]{\nabla^{k_0} \partial_t \tilde{\eta}^{(\tau)}}^2 dt &= \int_0^h  R \left(\eta^{(\tau)}_{k-1}, \partial_t \tilde{\eta}^{(h)}\right) dt + c_K\int_0^h h \norm[Q]{\nabla^{k_0} \partial_t \tilde{\eta}^{(\tau)}}^2 dt  \\ 
  \leq c \int_0^h  R_{h} \left(\eta^{(\tau)}_{k-1}, \partial_t \tilde{\eta}^{(h)}\right) dt  &\leq c\sum_{k=0}^N \tau R_{h} \left(\eta^{(\tau)}_{k-1},  \frac{\eta^{(\tau)}_k- \eta^{(\tau)}_{k-1}}{\tau}\right)
 \end{align*}
 where we know the right hand side to be bounded. Using Poincare's inequality %
 this then extends into an uniform $L^2([0,T];W^{k_0,2}(Q;\R^n))$ bound on $\partial_t \tilde{\eta}^{(\tau)}$. For the fluid, we use \autoref{prop:W2qIsomorph}, as well as the global Korn inequality \autoref{lem:globalKorn} to estimate for some constants
 \begin{align*}
  &\phantom{{}={}}\int_0^h C_{gK} \norm[W^{1,2}(\Omega)]{u^{(\tau)}}^2 + ch \norm[\Omega]{\nabla^{k_0} u}^2 dt \\
  &\leq \int_0^h R(\eta^{(\tau)},\tilde{\eta}^{(\tau)}) + \frac{\nu}{2} \norm[\Omega^{(\tau)}(t)]{ \nablasym u^{\tau}(t)}^2 dt + h \int_0^h \norm[Q]{\partial_t \nabla^{k_0}\tilde{\eta}^{(\tau)}}^2+ \norm[\Omega^{(\tau)}(t)]{\nabla^{k_0}u^{\tau}(t)}^2 dt
 \end{align*}
 which is uniformly bounded using the energy estimate again. The full estimate then follows by interpolating the missing derivatives.

 Finally, we consider the last term. Here we have
 \begin{align*}
  \int_0^{h} \norm[\Omega_0]{u^{(\tau)} \circ \Phi^{(\tau)} }^2 dt = \sum_{k=0}^N \tau \norm{u_k^{(\tau)} \circ \Phi_k^{(\tau)} }^2 \leq \sum_{k=0}^N \tau \frac{3}{2} \left(\norm{u_k^{(\tau)} \circ \Phi_k^{(\tau)} - w_k}^2 + \norm{w_k}^2\right)
 \end{align*}
 which again consists of two bounded sums.
\end{proof}

\subsubsection*{Proof of \autoref{thm:TDexistence}, step 2a: Bounds on $\Phi^{(\tau)}$}

We now arrive at one main difficulity in implementing the scheme, establishing the properties of and bounds on $\Phi^{(\tau)}$. The challenge here is that $\Phi^{(\tau)}$ is defined via concatenation of an unbounded (for $\tau \to 0$) number of functions and thus is highly nonlinear. As any linearizing would break the coupling properties needed, we will instead rely on using a high enough regularity for the constituting functions.

We will use this theorem to prove the following:
\begin{proposition}[$L^2C^{1,\alpha}$-bound for $v$] \label{prop:regularityOfV}
There is a $\tau_0>0$ and $\alpha > 0$, such that for all $\tau\in(0,\tau_0)$, we have that $\Phi_k^{(\tau)}: \Omega_0 \to \Omega_k$ is a diffeomorphism with $\frac{1}{2} \leq \det\nabla \Phi_k^{(\tau)} \leq 2$ for all $k < \frac{h}{\tau}$ and
 \[\sum_{k=1}^N \tau \norm[C^{1,\alpha}(\Omega_{k-1})]{v_k^{(\tau)}}^2 \leq \mathcal{K}\]
 for any $N< \frac{h}{\tau}$ where $\mathcal{K}$ and $\tau_0$ only depend on $w,h,E(\eta_0)$ and $f$. %
\end{proposition}

\begin{proof}

Here we use that $k_0$ is choosen large enough so that $W^{k_0,2}_0(\Omega;\R^n)$ embeds into $C^{1,\alpha}(\Omega;\R^n)$ for some $\alpha >0$. Thus
\begin{align*}
 \sum_{k=1}^N \tau \norm[C^{1,\alpha}(\Omega_{k-1})]{v_k^{(\tau)}}^2 \leq \int_0^h \norm[C^{1,\alpha}(\Omega)]{u^{(\tau)}}^2 \leq c \int_0^h \norm[W^{k_0,2}(\Omega)]{u}^2
\end{align*}
which is uniformly bounded by \autoref{lem:TDuniformBounds}.

Now we need to show the properties of $\Phi_N$. By chain rule, the multiplicative nature of the determinant and its expansion (\autoref{lem:detExpansion}) we have
\begin{align*}
 &\phantom{{}={}} \det \nabla \Phi_N^{(\tau)} = \prod_{k=1}^N \det\left(I + \tau \nabla v_k^{(\tau)} \right)|_{\Phi_{k-1}^{(\tau)}} = \prod_{k=1}^N \left[1 + \tau \underbrace{\tr\left(\nabla v_k^{(\tau)} \right)}_{=\diver v_k = 0}+ \sum_{l=2}^n \tau^l M_l\left(\nabla v_k^{(\tau)}\right) \right]_{y=\Phi_{k-1}^{(\tau)}} 
\end{align*} 
By the inequality between arithmetic and geometric mean, we then have
\begin{align*}
 \det \nabla \Phi_N^{(\tau)} \leq  \left( \sum_{k=1}^N \frac{1}{N} \left(1 + \sum_{l=2}^n \tau^l M_l\left(\nabla v_k^{(\tau)} \circ \Phi_{k-1}^{(\tau)} \right) \right) \right)^N \leq \left( 1 + \frac{1}{N}\sum_{k=1}^N \sum_{l=2}^n \tau^l c_l \Lip( v_k^{(\tau)})^l \right)^N
\end{align*}
Now as $(1+a/N)^N \to \exp(a)$ monotone increasing for $a>0$, we can further estimate
\begin{align*}
 \leq \exp\left(\sum_{k=1}^N \sum_{l=2}^n \tau^l c_l \Lip( v_k^{(\tau)})^l \right) = \exp\left(\sum_{l=2}^n c_l \tau^{l/2} \sum_{k=1}^N \left(\tau \Lip( v_k^{(\tau)})^2\right)^{l/2} \right)  \leq \exp\left(\sum_{l=2}^n c_l \tau^{l/2} \mathcal{K}^{l/2} \right)
\end{align*}
where we used that $l \geq 2$ and $\tau \Lip( v_{k_0}^{(\tau)})^2 \leq \sum_{k=1}^N \tau \Lip(v_k^{(\tau)})^2 \leq \tau \norm[C^{1,\alpha}(\Omega_{k-1})]{ v_k^{(\tau)}}^2 \leq \mathcal{K}$.

In a similar fashion, we can give a lower estimate
\begin{align*}
 \left(\det \nabla \Phi_N^{(\tau)}\right)^{-1} \leq \left( \sum_{k=1}^N \frac{1}{N} \left(1+ \sum_{l=2}^n \tau^l M_l\left(\nabla v_k^{(\tau)} \circ \Phi_{k-1}^{(\tau)} \right) \right)^{-1}  \right)^N \leq \exp\left( 2\sum_{l=2}^n c_l \tau^{l/2} \mathcal{K}^{l/2} \right)
\end{align*}
using $\frac{1}{1+a} \leq \frac{1}{1-\abs{a}} \leq 1+2\abs{a}$ for $\abs{a}$ small enough. Thus for $\tau_0$ small enough, we have that
 \begin{align*}
 \frac{1}{2}\leq \det(\nabla \Phi_N^{(\tau)})\leq 2.
 \end{align*}
 Now we know from the boundary condition that $\Phi_N|_{\partial \Omega_0}$ is an orientation preserving diffeomorpism, given by $\eta_N \circ \eta_0^{-1}$ and $id$ at the respective parts of the boundary. We also know that $\Omega_0$ and $\Omega_N$ are domains with the same topology as there were no collisions. But then $\Phi_N$ has to be a diffeomorphism by a simple degree argument.
\end{proof}

An immediate consequence of the last proof is the following:
\begin{corollary}[Regularity of $\Phi^{(\tau)}$] \label{lem:TDphiRegularity}
 The maps $\Phi^{(\tau)}(t,.)$ are uniformly Lipschitz, i.e. Lipschitz in $x$ such that the constants are bounded independently of $\tau$ and $t$. Furthermore in the limit we have
 \[\lim_{\tau \to 0} \det \nabla \Phi^{(\tau)} = 1.\]
\end{corollary}

\begin{proof}
By the estimates in the last proof, we find that
$\lim_{\tau \to 0} \det \nabla \Phi^{(\tau)} = 1$. What is left, is to prove the Lipschitz regularity.

Here we proceed in the same fashion:
 \begin{align*}
  \Lip(\Phi_N^{(\tau)}) &\leq \prod_{l=1}^N \Lip( \operatorname{id} + \tau v_l) \leq \prod_{l=1}^N (1 +\tau \Lip(v_l)) \leq \left( \frac{1}{N} \sum_{l=1}^N (1+ \tau \Lip(v_k)) \right)^{N} \\
  & =\left( 1  + \frac{1}{N}\sum_{l=1}^N \tau \Lip(v_l) \right)^{N} \leq \exp\left(\sum_{l=1}^N \tau \Lip(v_l)\right) \leq
  \exp\left(\sqrt{\sum_{l=1}^N \tau}\sqrt{\sum_{l=1}^N \tau \Lip(v_l)^2} \right)
  \\ 
  &\leq \exp\left(\sqrt{h}\sqrt{\mathcal{K} }\right) \qedhere
 \end{align*}
\end{proof}

\subsubsection*{Proof of \autoref{thm:TDexistence}, step 3: Convergence of the equation}

Now using compactness, we pick the usual subsequence and limits $\eta$, $v$, $\Phi$ such that
\begin{align*} %
 \eta^{(\tau)},\tilde{\eta}^{(\tau)} &\rightharpoonup^* \eta & \text{ in }& L^\infty([0,h]; W^{k_0,2}(Q;\R^n))\\
 \partial_t \tilde{\eta}^{(\tau)} & \rightharpoonup \partial_t \eta& \text{ in } & L^2([0,h];W^{k_0,2}(Q;\R^n))\\
 u^{(\tau)} & \rightharpoonup u &\text{ in } & L^2([0,h]; W^{k_0,2}(\Omega;\R^n))\\
 \Phi^{(\tau)} & \to \Phi & \text{ in }& C^0([0,h];C^\alpha(\Omega_0;\R^n))
\end{align*}
and we define $\Omega(t) = \Omega \setminus \eta(t,Q)$. In particular, due to \autoref{lem:TDphiRegularity} we know that $\Phi$ is Lipschitz with constant $\exp(\sqrt{Lh})$ and that $\det \nabla \Phi = 1$ almost everywhere. We also remark that $\Phi(t,.)|_{\partial \Omega_0}$ is injective as long as there is no collision in the solid (which we already excluded), and that again we also know that $\Phi(t,.):\Omega_0 \to \Omega(t)$ is a volume preserving diffeomorphism.

Finally we can conclude that
\begin{align*}
 \partial_t \Phi = \lim_{\tau \to 0} \partial_t \tilde{\Phi}^{(\tau)} = \lim_{\tau \to 0} u^{(\tau)} \circ \Phi^{(\tau)} = u \circ \Phi
\end{align*}
almost everywhere.

Then $\Phi$ has the properties required for a solution and $v$ and $\eta$ are coupled in the right way, as before. What is left is to show that these fulfill the weak equation.

In addition to the previous, let us also introduce the notation
\begin{align*}
 w^{(\tau)}(t) := w_k^{(\tau)} \text{ if } \tau k \leq t < \tau(k+1).
\end{align*}
Then in particular by the Lebesgue differentiation theorem $w^{(\tau)} \to w$ in $L^2([0,h] \times \Omega;\R^n)$ and $w^{(\tau)} \circ \eta_0^{-1} \to w \circ \eta_0^{-1}$ in $L^2([0,h] \times Q;\R^n)$, which is enough for our convergences. 

In order to achieve convergence of the energy-term in the weak equation we again need to improve the convergence of $\eta^{(\tau)}$ using the Minty-method, the proof of which is identical to the one used in the proof of \autoref{thm:SOtimeDelayedExistence}.

We now close the proof in the same way as we did in \autoref{prop:QSconvergence} and \autoref{thm:SOtimeDelayedExistence}, since the convergences and most of the terms are again identical.

As before, we use \autoref{lem:approxTestFcts} and pick a test function $\xi \in C_0^\infty([0,h]\times \Omega;\R^n)$ such that $\diver \xi =0$ in a neighborhood of the fluid domain. From this we can construct matching $\phi^{(\tau)} := \xi \circ \eta^{(\tau)}$ and use those to test the discrete Euler-Lagrange equation from \autoref{prop:TDdiscreteELequation}.

Most of the terms, including all those related to the solid, we have already dealt with in the previous iterations of this proof. What is left are the additional regularization term, the inertial effects of the fluid and the force term for the fluid which has been slightly modified from before.

We start with the latter, where we simply note that $\Phi^{(\tau)}$ and thus any concatenation with a uniformly continuous function such as $f \circ \Phi^{(\tau)}$ converges uniformly. Thus
\begin{align*}
 \int_0^h \inner[\Omega_0]{f\circ \Phi^{(\tau)}}{\xi \circ \Phi^{(\tau)}} dt \to \int_0^T \inner[\Omega_0]{f\circ \Phi}{\xi \circ \Phi} dt =\int_0^h \inner[\Omega(t)]{f}{\xi} dt
\end{align*}
where the last equality is true as $\Phi$ is volume preserving.

Of greater interest are the inertial terms. The term for the solid, we already discussed in \autoref{thm:SOtimeDelayedExistence} and the additionial $\tau$-dependence of $\phi$ does not change its convergence. For the fluid finally, we have
\begin{align*}
 \int_0^h \inner[\Omega_0]{u^{(\tau)} \circ \Phi^{(\tau)} - w^{(\tau)}}{\xi \circ \Phi^{(\tau)}} dt \to
 \int_0^h \inner[\Omega_0]{u \circ \Phi - w}{\xi \circ \Phi} dt
\end{align*}
as the right side converges uniformly and the left side at least weakly in $L^2([0,h]\times \Omega_0;\R^n)$.

Finally we note that as used before $\chi_{\Omega^{(\tau)}(t)} \nabla^{k_0} \xi \to \chi_{\Omega(t)} \nabla^{k_0} \xi$ in $L^2([0,h];L^2(\Omega;\R^n))$ and thus
\begin{align*}
 \int_0^h \inner[\Omega^{(\tau)}(t)]{\nabla^{k_0} u^{(\tau)}}{ \nabla^{k_0} \xi} dt \to\int_0^h \inner[\Omega(t)]{\nabla^{k_0} u}{ \nabla^{k_0} \xi} dt
\end{align*}
by the corresponding weak convergence of $u^{(\tau)}$. This finishes the proof. \qed

\subsubsection*{A posteriori energy inequality}

We close this section with an energy inequality corresponding to \autoref{lem:SOtimeDelayedEnergyInequality}. As before, this will be the central estimate that allows us to take the limit $h\to 0$ and converge to the Navier-Stokes equation.

\begin{lemma}[Energy inequality for time delayed solutions]\label{lem:TDaPosterioriEstimate}
 Assume that $(\eta,v)$ is a weak solution to the time delayed equation \eqref{eq:TDweakSolution}, as constructed in \autoref{thm:TDexistence}. Then we have the following energy inequality
 \begin{align*}
  &\phantom{{}={}}E_h(\eta(h)) + \int_0^{h} 2R_{h}(\eta, \partial_t \eta) + \nu \norm[\Omega(t)]{\nablasym v}^2 + h \norm[\Omega(t)]{\nabla^{k_0} v}^2 dt  +  \fint_0^h \frac{\rho_f}{2} \norm[\Omega(t)]{v}^2 + \frac{\rho_f}{2} \norm[Q]{\partial_t \eta}^2 dt\\
  &\leq E_h(\eta(0)) + \int_0^{h} \rho_f\inner[\Omega(t)]{f}{v} +\rho_s\inner[Q]{f\circ\eta}{\partial_t \eta}dt + \int_0^{h} \frac{\rho_f}{2h} \norm[\Omega_0]{w}^2 + \frac{\rho_f}{2h} \norm[Q]{w\circ \eta_0^{-1}}^2 dt.
 \end{align*}
\end{lemma}

\begin{proof}
 We test the equation with the pair $(\partial_t \eta,v)$. These have the correct coupling and boundary conditions. We need to be careful with regularity here. In the general framework, energy and dissipation require different regularities. This is where the regularized dissipation comes in. From \autoref{cor:regularDissipationProperties} we know that $\partial_t \eta \in L^2([0,h];W^{2,q}(Q;\R^n))$ and it thus is a valid test function for $DE(\eta)$. Testing the equation then gives
 \begin{align*}
0&=\int_0^{h} \inner[Q]{DE_h(\eta)}{\partial_t \eta} + \inner[Q]{D_2R_{h}(\eta,\partial_t \eta)}{ \partial_t \eta} +\inner[Q]{\rho_s \frac{\partial_t \eta - w \circ \eta_0^{-1}}{h} }{ \partial_t \eta} \\
&+ \nu \inner[\Omega(t)]{\nablasym v}{\nablasym v} + \inner[\Omega_0]{\rho_f\frac{v\circ \Phi - w}{h}}{v \circ \Phi} - \rho_f\inner[\Omega_0]{ f}{v} - \rho_s\inner[Q]{f\circ \eta}{\partial_t \eta} dt
 \end{align*}
 Now the first term is just the time derivative of the energy and thus its integral is $E_h(\eta(h))-E_h(\eta(0))$ while for the second term we remember that due to the quadratic scaling of the dissipation $\inner[Q]{D_2R_{h}(\eta,\partial_t \eta)}{ \partial_t \eta} = 2R_{h}(\eta,\partial_t \eta)$. Finally we estimate the inertial terms using Young's inequality in the form of $\inner{a-b}{a} = \abs{a}^2 - \inner{b}{a} \geq \frac{1}{2}\abs{a}^2- \frac{1}{2}\abs{b}^2$. Reordering terms according to sign then proves the estimate.
\end{proof}

\subsection{Proof of \autoref{thm:NSexistence}}
\label{subsec:proofNS}

Similar to what we did in \autoref{sec:so}, we will now use solutions to the time-delayed problem to approximate the full problem. The main added difficulty in the proof is in dealing with the inertial effects of the fluid. A particular problem there is that the flow map itself does not persist in the limit for $h \to 0$. However since all terms that involve the flow map only ever need it for a flow of length $h$, the goal is simply to find the right reformulation such that limit quantities still exist. In the same context we note that we only ever hope to obtain weak solutions to the Navier-Stokes subsystem. In particular, the material derivative $\partial_t v + v \cdot \nabla v$ turns out to be a problematic term. Finally we note that due to the changing domain, we generally use convergence of $u$ instead of $v$. But for this is is quite obvious that $\partial_t u$ is not a meaningful quantity in any sense (See also \autoref{rem:NSmaterialDerivative}).

With all this in mind, let us begin with the proof.

\subsubsection*{Proof of \autoref{thm:NSexistence}, step 1: Constructing another iterative approximation}

We now iteratively construct an approximative solution to the Navier-Stokes equation using solutions to our previous problem in the following sense:

For a fixed $h$ assume that $\eta_0$ with finite energy $E_h(\eta_0)$ and $v_0: \Omega_0 := \Omega \setminus \eta_0(Q) \to \R^n$, $\diver v_0 = 0$ and $\eta': Q \to \R^n$ are given. Set $w_0(t,y) = v_0(y)$ for $y\in \Omega_0$ and $w_0 = \eta' \circ \eta_0^{-1}$ otherwise.\footnote{Note that for this first step, $v_0$ and $\eta'$ do not need to fulfill a coupling condition $\eta' = v_0 \circ \eta_0$ on $\partial Q \setminus P$ yet. This is completely reasonable from a mathematical point of view, as initial values will only ever be taken in an $L^2$-sense, so there is no trace-theorem to make sense of this condition.}

Now for every step, assume that $\eta_l:Q\to \Omega$, $w_l : [0,h] \times \Omega$ and $\Omega_l := \Omega \setminus \eta_l(Q)$ are known. We use \autoref{thm:TDexistence} to construct a solution $\tilde{\eta}_{l+1},v_{l+1}, \Phi_{l+1}$ to the time-delayed problem \eqref{eq:TDweakSolution} on the interval $[0,h]$ with these as given data.
Observe in particular, that
\[
\Phi_{l+1}(s)(\Omega_l)=\Omega\setminus \tilde{\eta}_{l+1}(s,Q).
\]
 We define $\eta_{l+1} = \tilde{\eta}_{l+1}(h,.)$ and $\Omega_{l+1}=\Omega\setminus \eta_{l+1}(Q)$ and construct
\[
w_{l+1}:[0,h]\times \Omega_{l+1}\to \R^n,\quad w_{l+1}(t,.) = v_{l+1}(t,.) \circ \Phi_{l+1}(t,.) \circ \Phi_{l+1}(h,.)^{-1}.\]
Then $E_h(\eta_{l+1}) < \infty$ by the energy inequality \autoref{lem:TDaPosterioriEstimate} and since $\Phi_{l+1}$ is volume preserving $\int_0^{h} \norm{w_{l+1}}^2 dt = \int_0^{h} \norm[\Omega_k(t)]{v_{l+1}}^2 dt < \infty$ and so we can iterate this until we reach a collision or until $E(\eta_l)$ or $w_l$ diverge (as we will see by the next lemma, neither of the last two can happen in finite time).

Now we construct the $h$-approximation.

\begin{definition}[$h$-approximation]
 Let $(\tilde{\eta}_l)_l$ and $v_l$ as constructed before for fixed $h$. Then we define the approximations $\eta^{(h)}: [0,T] \times  Q \to \Omega$, $u^{(h)}: [0,T] \times \Omega \to \R^n$
 \begin{align*}
  \eta^{(h)}(t,x) &:= \tilde{\eta}_l(t-lh,x) & \text{ for }& t \in [lh,(l+1)h)
   \\
  \Omega^{(h)}(t) &:= \Omega_l(t-hl) & \text{ for }& t \in [lh,(l+1)h)
  \\
   v^{(h)}(t,y) &:= v_l(t-lh,y) & \text{ for }& t \in [lh,(l+1)h), y \in \Omega^{(h)}(t)\\
  u^{(h)}(t,y) &:= v^{(h)}(t,y) & \text{ for }& t \in [0,T), y \in \Omega^{(h)}(t) \\
  u^{(h)}(t,y) &:= \partial_t \eta^{(h)}(t,(\eta^{(h)}(t))^{-1}(y)) & \text{ for }& t \in [0,T), y \in \eta^{(h)}(t,Q)\\
  \rho^{(h)}(t,y) &:= \rho_f & \text{ for }& t \in [0,T), y \in \Omega^{(h)}(t) \\
  \rho^{(h)}(t,y) &:= \frac{\rho_s}{\det(\nabla \eta^{(h)}(t,(\eta^{(h)}(t))^{-1}(y)))} & \text{ for }& t \in [lh,(l+1)h), y \notin \Omega_l(t)
 \intertext{ Moreover for $y\in \Omega^{(h)}(t)$ and $s\in [-h,h]$ we define for $t\in [lh,(l+1)h)$}
  \Phi_s^{(h)}(t,.) &:= \Phi_{l}(t+s-lh) \circ (\Phi_{l}(t-lh))^{-1}& \text{ if }& t+s\in [lh,(l+1)h)\\
  \Phi_s^{(h)}(t,.) &:= \Phi_{l+1}(t+s-(l+1)h) \circ \Phi_{l}(h) \circ (\Phi_{l}(t-lh))^{-1}& \text{ if }& lh (l+1)h \leq t+s < (l+2)h
 \\
 \Phi_s^{(h)}(t,.) &:= \Phi_{l-1}(t+s-(l-1)h) \circ (\Phi_{l}(h))^{-1} \circ (\Phi_{l}(t-lh))^{-1}\!\!\!\!& \text{ if }& (l-1)h \leq t+s < lh.
 \end{align*}
 For $y\in \eta^{h}(t,Q)$ and $s\in [-h,h]$ we define 
 \[
  \Phi_s^{(h)}(t) := \eta^{(h)}(t+s) \circ (\eta^{(h)}(t))^{-1}
 \]
\end{definition}
 The map $\Phi_s^{(h)}(t)$ is a volume preserving diffeomorphism between $\Omega^{(h)}(t)$ and $\Omega^{(h)}(t+s)$.
 The following lemma shows that the map can be extended to a continuous function in space-time.
\begin{lemma}[The global flow map]
\label{lem:phi}
For all $h>0$ there is a flow map that is continuous in time-space satisfying 
\begin{align}
\label{eq:dsphi}
\partial_s \Phi_s^{(h)}(t,y) = u^{(h)}(t+s,\Phi_s^{(h)}(t,y)).
\end{align}
Moreover,
\begin{align*}
\det(\nabla \Phi_s^{(h)}(t,y))&=1 &&\text{ for }y\in \Omega^{(h)}(t) 
\quad \text{ and } \\
\det(\nabla \Phi_s^{(h)}(t,y))&=\frac{\rho^{(h)}(t+s,\Phi_s^{(h)}(t,y))}{\rho^{(h)}(t,y)} &&\text{ for }y\in \eta^{(h)}(t,Q)
\end{align*}
The inverse of the flow map is given by $(\Phi_s^{(h)}(t))^{-1}=\Phi_{-s}^{(h)}(t+s)$.
\end{lemma}
\begin{proof}
For all $y\in \Omega^{(h)}(t)\cup \eta^{(h)}(t,Q)$ we find (by chain rule and \autoref{thm:TDexistence}) that
that
\begin{align*}
\partial_s \Phi_s^{(h)}(t,y) = u^{(h)}(t+s,\Phi_s^{(h)}(t,y)).
\end{align*}
Since for $s=0$ the function $\Phi_0^{(h)}(t)=Id$ is trivially continuous over $\Omega$ and by the a-priori estimates also $u$ is uniformly Lipschitz continuous (in dependence of $h$). Hence by a standard argument for ordinary differential equations $ \Phi_s^{(h)}(t,y) $ is continuous over $\Omega$.

The identity of the determinant follows by \autoref{thm:TDexistence} for the fluid part and by chain rule and the defintion of $\rho^{(h)}$ for the solid part.
Furthermore the inverse of the flow map is given as the respective flow in the opposite direction, which is verified by considering $(\Phi_s^{(h)}(t))^{-1}=\Phi_{-s}^{(h)}(t+s)$.
\end{proof}

 It would also be possible to define $\Phi_s^{(h)}(t)$ for larger $s$, but for the remainder of the proof we only need $s \in [-h,h]$. (See also \autoref{rem:NSlagrangianVsEulerian} in regards to this).

The weak formulation of the approximate system is 
\begin{align}
\label{eq:weak-h}
  &\phantom{{}={}} \int_0^T \inner{DE_h(\eta^{(h)})}{\phi} + \inner{DR_h(\eta^{(h)},\partial_t \eta^{(h)})}{\phi} +\rho_s \inner{\frac{\partial_t \eta^{(h)}(t)-\partial_t \eta^{(h)}(t-h)}{h}}{\phi}  \\ \nonumber
 & + \inner[\Omega^{(h)}(t)]{\nablasym v^{(h)}}{\nablasym \xi} +  \rho_f \inner[\Omega^{(h)}(t-h)]{\frac{ v^{(h)}(t) \circ \Phi^{(h)}_h(t-h) - v^{(h)}(t-h)}{h}}{\xi(t)\circ \Phi_h^{(h)}(t-h)} dt. \\ \nonumber
 &= \int_0^T \rho_s \inner[Q]{f\circ \eta^{(h)}}{\phi} + \rho_f\inner[\Omega^{(h)}(t)]{f}{\xi} dt
\end{align}
for all $\phi \in C^0([0,T]; W^{k_0,2}(Q;\R^n))$, $\xi \in C^0([0,T];W^{k_0,2}(\Omega;\R^n))$ satisfying $\diver \xi|_{\Omega(t)} =0$, $\xi|_{\partial \Omega} = 0$, $\phi|P = 0$ and the coupling conditions
$\xi \circ \eta = \phi  \text{ and } u \circ \eta = \partial_t \eta \text{ in }  Q$.
Observe that by the definition of $\rho^{(h)}$ above, we find by a change of variables the following identity for the global momentum: for the same pair of testfunctions:
\begin{align}
\label{eq:weak-hu}
\begin{aligned}
&\inner[\Omega]{\frac{ \rho^{(h)}u^{(h)}(t) \circ \Phi^{(h)}_h(t-h) - \rho^{(h)}u^{(h)}(t-h)}{h}}{\xi(t)\circ \Phi_h^{(h)}(t-h)}
\\
&\quad = \rho_f \inner[\Omega^{(h)}(t-h)]{\frac{ v^{(h)}(t) \circ \Phi^{(h)}_h(t-h) - v^{(h)}(t-h)}{h}}{\xi(t)\circ \Phi_h^{(h)}(t-h)} dt
\\
&\qquad +\rho_s \inner{\frac{\partial_t \eta^{(h)}(t)-\partial_t \eta^{(h)}(t-h)}{h}}{\phi}.
\end{aligned}
\end{align}

The a posteriori-estimate \autoref{lem:TDaPosterioriEstimate} then leads us to an a-priori estimate for the new iteration:
\begin{lemma}[A-priori estimate (full problem)] \label{lem:NSfullIterationAPriori} We have for any $t \in [0,T]$
 \begin{align*}
  &\phantom{{}={}} E_h(\eta^{(h)}(t)) + \fint_{t-h}^{t} \frac{\rho_f}{2} \norm[\Omega^{(h)}(t)]{u^{(h)}}^2 + \frac{\rho_s}{2} \norm[Q]{\partial_t \eta^{(h)} }^2 dt \\
  &+ \int_0^{t}  R_h(\nabla \eta^{(h)}, \partial_t \eta^{(h)} ) + \nu\norm[\Omega^{(h)}(t)]{\nablasym u^{(h)}}^2 + h \norm[\Omega^{(h)}(t)]{\nabla^{k_0} u^{(h)}}^2 dt 
  \\ &\leq E_h(\eta_0) + \frac{1}{2} \norm[\Omega_0]{v_0}^2 + \int_0^{t} \rho_f \inner[\Omega^{(h)}(t)]{f}{u^{(h)}} + \rho_s \inner[Q]{f\circ\eta^{(h)}}{\partial_t\eta^{(h)}}dt,
 \end{align*}
 and moreover there exist $C,c>0$ independent of $h$ such that
  \begin{align*}
  &\phantom{{}={}} E_h(\eta^{(h)}(t)) + c \fint_{t-h}^{t} \norm[\Omega^{(h)}(t)]{u^{(h)}}^2 + \norm[Q]{\partial_t \eta^{(h)} }^2 dt \\
  &+ \int_0^{t}  R_h(\nabla \eta^{(h)}, \partial_t \eta^{(h)} ) + \nu\norm[\Omega^{(h)}(t)]{\nablasym u^{(h)}}^2 +h \norm[\Omega^{(h)}(t)]{\nabla^{k_0} u^{(h)}}^2 dt \leq C+ C t^2
\end{align*}
 In both these estimates take $u^{(h)}$ and $\partial_t \eta$ to be continued by their initial values for $t<0$.
\end{lemma}

\begin{proof}
 From \autoref{lem:TDaPosterioriEstimate} applied to the single step, from $l$ to $l+1$ we get
 \begin{align*}
  &\phantom{{}={}} E_h(\eta_{l+1}) + \fint_0^{s} \frac{\rho_f}{2} \norm[\tilde{\Omega}_l(t)]{v_{l+1}}^2  + \frac{\rho_s}{2} \norm[Q]{\partial_t \tilde{\eta}_{l+1}}^2 dt \\
  &+ \int_0^{s} R_h(\nabla \tilde{\eta}_{l+1},\partial_t \nabla \tilde{\eta}_{l+1})  + \nu\norm[\tilde{\Omega}_l(t)]{\nablasym v_{l+1}}^2 +h \norm[\Omega^{(h)}(t)]{\nabla^{k_0} v_{l+1}}^2 dt  dt \\
  &\leq E_h(\eta_{l}) + \fint_0^{s} \frac{\rho_f}{2} \norm[\Omega_{l}]{w_l}^2 + \frac{\rho_s}{2} \norm[Q]{w_l \circ \eta_l} dt + \int_0^{s} \inner[\Omega_l]{f}{v_{l+1}} + \rho_s\inner[Q]{f\circ \eta_l}{\partial_t\tilde{\eta}_l}dt
 \end{align*}
 for $s \in [0,h]$.
 Now per construction $\norm[\Omega_l]{w_l(t,.)} = \norm[\tilde{\Omega}_{l-1}(t)]{v_l(t,.)}$ and $\norm[Q]{w_l \circ \eta_l} = \norm[Q]{\partial_t \tilde{\eta}_l}$ thus we can use a telescope argument to get the first energy inequality as we did in \autoref{lem:SOtimeDelayedEnergyInequality}.
 
 Next we use Young's inequality on the two force terms to obtain
\begin{align*}
&\phantom{{}={}}\int_0^{t} \rho_f\inner[\Omega^{(h)}(t)]{f}{v^{(h)}} + \rho_s\inner[Q]{f\circ \eta_l}{\partial_t\eta^{(h)}}dt\\
&\leq \int_0^t \frac{1}{2\delta} \left(\rho_f\norm[\Omega^{(h)}(t)]{f}^2 + \rho_s\norm[Q]{f\circ \eta_l}^2\right) + \frac{\delta}{2} \left(\rho_f\norm[\Omega^{(h)}]{v^{(h)}}^2 +  \rho_s \norm[Q]{\partial_t \eta^{(h)} }^2 \right) dt\\
 &\leq \frac{C}{\delta} t \norm[\infty]{f}^2 +   \int_{0}^{t} \frac{\delta\rho_f}{2} \norm[\Omega^{(h)}(t)]{v^{(h)}}^2+\frac{\delta\rho_s}{2} \norm[Q]{\partial_t\eta^{(h)}}^2\, dt.
\end{align*}
Now dropping all other non-negative terms on the left hand side and extending the interval to $[0,T]$, we have
\begin{align*}
 &\phantom{{}={}} \fint_{t-h}^{t} \frac{\rho_f}{2} \norm[\Omega^{(h)}(t)]{v^{(h)}}^2  + \frac{\rho_s}{2} \norm[Q]{\partial_t \eta^{(h)}}^2 dt \\
 &\leq E(\eta_0) - E_{\min} + \frac{1}{2} \norm[\Omega_0]{v_0}^2 + \frac{C}{\delta} T \norm[\infty]{f}^2 +  \int_{0}^{T} \frac{\delta\rho_f}{2} \norm[\Omega^{(h)}(t)]{v^{(h)}}^2+\frac{\delta\rho_s}{2} \norm[Q]{\partial_t\eta^{(h)}}^2\, dt 
\end{align*}
and summing over intervals for $T = hN$ thus gives
\begin{align*}
 &\phantom{{}={}}\int_{0}^{T} \frac{\rho_f}{2} \norm[\Omega^{(h)}(t)]{v^{(h)}}^2+\frac{\rho_s}{2} \norm[Q]{\partial_t\eta^{(h)}}^2\, dt \leq \sum_{l=1}^N h \fint_{(l-1)h}^{lh} \frac{\rho_f}{2} \norm[\Omega^{(h)}(t)]{v^{(h)}}^2  + \frac{\rho_s}{2} \norm[Q]{\partial_t \eta^{(h)}}^2 dt \\
 &\leq hN \left(C + \frac{C}{\delta} T \norm[\infty]{f}^2 +  \int_{0}^{T} \frac{\delta\rho_f}{2} \norm[\Omega^{(h)}(t)]{v^{(h)}}^2+\frac{\delta\rho_s}{2} \norm[Q]{\partial_t\eta^{(h)}}^2\, dt \right)
\end{align*}
from which as before in \autoref{thm:SOexistence} choosing $\delta= \frac{1}{2T}$ yields the desired estimate.
\end{proof}

\begin{corollary}[Minimal no collision time]
\label{cor:short-time-no-collision}
 Assume that $\eta_0 \notin \partial \mathcal{E}$. Then there is a $T>0$ depending only on $\eta_0, v_0$ and $f$ such that $\eta_l$ and $\tilde{\eta}_l$ are injective on $\overline{Q}$ for all $h$ small enough and $h l \leq T$, i.e. we have $\eta^{h}(t) \notin \partial \mathcal{E}$ and thus there is no collision.
\end{corollary}

\begin{proof}
 From the final estimate in the proof of \autoref{lem:NSfullIterationAPriori} we get
 \[\norm[Q]{\eta_N-\eta_0}^2 = \norm[Q]{\sum_{l=0}^{N-1}  \int_0^{h} \partial_t \tilde{\eta}_l dt}^2 \leq  \int_0^{T} \norm[Q]{\partial_t \tilde{\eta}_l}^2 dt \leq T  C(1+T^2).\]
 Using this bound for small enough $T$ then allows us to apply the short-distance injectivity result \autoref{prop:shortInjectivity}.
\end{proof}
As a direct consequence of the uniform bounds of $\det(\nabla \eta^{(h)})$, the definition of $R_h$ and \autoref{lem:globalKorn}, we find that

\begin{corollary}[Korn-type estimate]
\label{cor:NSkorn}
There is a constant just depending on the energy estimate in \autoref{lem:NSfullIterationAPriori}, such that
\begin{align*}
&\sup_{t\in[0,T-h]}\fint_{t}^{t+h}\norm[\Omega]{u(s)}^2\, ds+\int_0^T \norm[W^{1,2}(Q)]{\partial_t\eta^{(h)}}^2+\norm[W^{1,2}(\Omega)]{u^{(h)}}^2 dt\leq C,
\\
&\sup_{t\in [0,T]} h^\frac{a_0}{2}\norm[W^{k_0,2}(Q)]{\eta^{(h)}}+\sqrt{h} \sqrt{\int_0^T \norm[W^{k_0,2}(Q)]{\partial_t\eta^{(h)}}^2 + \norm[W^{k_0,2}(\Omega)]{u^{h}}^2}\leq C.
\end{align*}
\end{corollary}
\subsubsection*{Proof of \autoref{thm:NSexistence}, step 2: The weak time-derivative}
The weak time-derivative in the frame-work of incompressible fluids is a delicate issue due the fat that it is merrely in the dual of functions that are solenoidal on the fluid-domain. We introduce here a subtle estimate which is essential in order to pass to the limit with the convective term which, ever since the days of Leray, has been the major challange in the existence theory for PDE's involving Navier-Stokes equations; in particular in the field of fluid-structure interactions~(see for instance~\cite{lengelerStokestypeSystemArising2015}). 

In the following, we may understand $\partial_t \eta^{(h)}$ and $u^{(h)}$ to be extended by their initial values for $t \in [-h,0]$.

\begin{lemma}[Length $h$ bounds (fluid)] \label{lem:NSh-m-estimate}
Fix $T > 0$. Then there exists a constant $C$ depending only on the initial data, such that the following holds:
\begin{enumerate}
  \item $\int_0^T \norm[H^{-m}(Q)]{\frac{\partial_t \eta^{(h)}(t)-\partial_t \eta^{(h)}(t-h)}{h}}^2 dt \leq C$
  \item
 $\norm[\Omega]{\xi(t) -\xi(t-s_0) \circ \Phi_{-s_0}^{(h)}(t)} \leq C h \Lip(\xi)$ for all $\xi \in C_0^\infty([0,T]\times \Omega)$, $s_0\in[-h,h]$
 \item $
  \norm[\Omega]{\xi-\xi \circ \Phi_{s_0}^{(h)}(t-h)} \leq C h\Lip(\xi)$ for all $\xi \in C_0^\infty(\Omega)$, $s_0\in[0,h]$
\end{enumerate}
\end{lemma}

\begin{proof} 
The first estimate is shown in almost the same way as \autoref{lem:SOh-m-estimate}. While we need to construct a matching $\xi$, this can be set to $0$ on the fluid-domain and thus have no impact.

For the second let $\xi \in C_0^\infty([0,T]\times \Omega;\R^n)$ and calculate
 \begin{align*}
  &\phantom{{}={}}\int_{\Omega} \abs{\xi(t) - \xi(t-s_0)\circ \Phi_{-{s_0}}^{(h)}(t) }^2 dy \\
  &=\int_{\Omega} \abs{ \int_{-s_0}^{0} \partial_s \left(\xi(t+s) \circ \Phi^{(h)}_s(t)\right) ds }^2 dy\\
  &= \int_{\Omega} \abs{ \int_{-s_0}^0 \left[\nabla\xi(t+s) \cdot u^{(h)}(t+s)  + \partial_t \xi(t+s) \right]_{\Phi^{(h)}_s(t)}  ds }^2 dy  \\
  &\leq s_0\int_{-s_0}^0\int_{\Omega} \abs{ \left[\nabla\xi(t+s) \cdot u^{(h)}(t+s) + \partial_t \xi(t+s) \right]_{\Phi^{(h)}_s(t)}  }^2 dy ds \\
  &\leq h^2 \Lip(\xi)^2 \fint_{t-h}^{t}\int_\Omega \det(\nabla \phi^{(h)}_{-s}(t+s))\abs{u^{(h)}(s,.)}^2+1\, ds 
  \\
  &\leq Ch^2 \Lip(\xi)^2 \fint_{t-h}^{t}\left(\norm[\Omega]{u^{(h)}(s,.)}+1\right)^2 ds 
  \end{align*}
using the uniform bounds of $\det(\nabla \phi^{(h)}_{-s}(t+s))$ (see the characterization in \autoref{lem:phi}) and the velocity \autoref{cor:NSkorn}. This implies (2). The third assertion follows by the very same arguments.
 \end{proof}
 The next proposition estimates the weak time-derivative of the momentum equation. 
 \begin{proposition} \label{prop:NSh-m-estimateFluid}
There is a $m\geq k_0$ and a constant just depending on the uniform energy bound, such that for all
$\xi \in C^0(0,T; W^{m,2}_0(\Omega;\R^n)$ with $ \{\diver \xi|_{\Omega(t)} = 0\}$ (for all $t\in [0,T]$, then 
\begin{align*}
\int_0^T \abs{\inner[\Omega]{ \frac{(\rho^{(h)}  u^{(h)})(t)-(\rho^{(h)}  u^{(h)})(t-h)}{h}}{\xi(t)} }dt\leq C\norm[L^2({[a,b]};W^{m,2}(\Omega))]{\xi} 
\end{align*}
for all $0\leq a<b\leq T$.
 \end{proposition} 
 \begin{proof}
 Let $\xi \in C^0([0,T]\times\Omega;\R^n)$ with $\diver \xi(t) = 0$ on $\Omega^{(h)}(t)$ for all $t\in [0,T]$ and define $\phi := \xi \circ \eta^{(h)}$. 
 Let us first split the integrand into two 
  along the flow map.
 \begin{align*}
  &\phantom{{}={}}\inner[\Omega]{ \frac{(\rho^{(h)}  u^{(h)})(t)-(\rho^{(h)}  u^{(h)})(t-h)}{h}}{\xi(t)} \\
  &=\inner[\Omega]{ \frac{(\rho^{(h)}  u^{(h)})(t)-(\rho^{(h)}  u^{(h)})(t-h) \circ \Phi^{(h)}_{-h}(t)}{h}}{\xi(t)} \\
  &- \inner[\Omega]{ \frac{(\rho^{(h)}  u^{(h)})(t-h)-(\rho^{(h)}  u^{(h)})(t-h) \circ \Phi^{(h)}_{-h}(t)}{h}}{\xi(t)}=: J_1(t) - J_2(t)
 \end{align*}
 Now we estimate the first of those using the weak equation \eqref{eq:weak-h}
 \begin{align*}
  &\phantom{{}={}}\int_0^T \abs{J_1(t)} dt 
   = \int_0^T \bigg|\rho_s \inner{\frac{\partial_t \eta^{(h)}(t)-\partial_t \eta^{(h)}(t-h)}{h}}{\phi}
  \\
  & + \rho_f \inner[\Omega^{(h)}(t-h)]{\frac{ u^{(h)}(t) \circ \Phi^{(h)}_h(t-h) - u^{(h)}(t-h)}{h}}{\xi(t)\circ \Phi_h^{(h)}(t-h)}\bigg|dt 
  \\
  &\leq
  \int_0^T \abs{ \inner{DE(\eta^{(h)})}{\phi}} +h^{a_0}\abs{\inner[Q]{\nabla^{k_0}\eta^{(h)}}{\nabla^{k_0}\phi}}
  + \abs{\inner{D_2R_h(\eta^{(h)},\partial_t \eta^{(h)})}{  \phi}}   \\
  & + h\abs{\inner[Q]{\nabla^{k_0}\partial_t\eta^{(h)}}{\nabla^{k_0}\phi}}+
  \nu \abs{ \inner[\Omega(t)]{\nablasym u^{(h)}}{\nablasym \xi}}  + h \abs{\inner[\Omega(t)]{\nabla^{k_0} u^{(h)}}{\nabla^{k_0}\xi}} 
  \\
  & + \rho_f\abs{  \inner[\Omega(t)]{f}{\xi}} + \rho_s\abs{ \inner[Q]{f\circ \eta}{\phi}}dt \\
  &\leq c\int_0^T \norm[W^{-2,q}(Q)]{DE(\eta^{(h)})} \norm[W^{2,q}(\Omega)]{\xi(t)}\norm[W^{2,q}(Q)]{\eta^{(h)}(t)} \\
  &+ \norm[W^{-1,2}(Q)]{D_2R(\eta^{(h)},\partial_t \eta^{(h)})} \norm[W^{2,q}(\Omega)]{\xi(t)}\norm[W^{2,q}(Q)]{\eta^{(h)}(t)} 
\\
& +\norm[W^{k_0,2}(Q)]{\eta^{(h)}(t)}\left[h^{a_0}\norm{\nabla^{k_0}\eta^{(h)}}\norm[C^{k_0}(\Omega)]{\xi(t)}
+h\norm[W^{k_0,2}(Q)]{\partial_t\eta^{(h)}(t)}\norm[C^{k_0}(\Omega)]{\xi(t)} \right]
  \\
&+ h\norm[W^{k_0,2}(Q)]{u^{(h)}(t)}\norm[W^{k_0,2}(Q)]{\xi} + \norm[\Omega(t)]{\nablasym u^{(h)}} \norm[\Omega(t)]{\nablasym \xi} +\norm[\infty]{f} \norm[\Omega(t)]{\xi}dt
 \end{align*}
 where we used that by \autoref{prop:W2qIsomorph} we have $\norm[W^{2,q}(Q)]{\phi(t)} \leq c \norm[W^{2,q}(\Omega)]{\xi(t)}\norm[W^{2,q}(Q)]{\eta^{(h)}(t)}$ and $\norm[W^{k_0,2}(Q)]{\phi(t)} \leq c \norm[C^{k_0}(\Omega)]{\xi(t)}\norm[W^{k_0,2}(Q)]{\eta^{(h)}(t)}$.
 Now from the energy estimate we know that $\norm[W^{2,q}(Q)]{\eta^{(h)}(t)}$ and $h^{a_0/2}\norm[W^{k_0,2}(Q)]{\eta^{(h)}(t)}$ are uniformly bounded in $h$ and $t$. Thus every term is a product of a quantity which has (at least) a uniform $L^2([0,T])$-bound using the energy estimate and a term which can be estimated against $\norm[C^{k_0}(\Omega)]{\xi(t)}$. Choosing $m$ such that $W^{m,2}(\Omega)$ embeds into $C^{k_0}(\Omega)$ then gives us
 \begin{align*}
  \int_0^T \abs{J_1(t)} dt \leq C \norm[L^2({[0,T]};W^{m,2}(\Omega))]{\xi}
 \end{align*}

 For the other term we first note that by the density preserving nature of $\Phi$ we have
 \begin{align*}
  \inner[\Omega]{(\rho^{(h)}  u^{(h)})(t-h) \circ \Phi_{-h}^{(h)}(t)}{\xi(t)} = \inner[\Omega]{(\rho u)^{(h)}(t-h)}{\xi(t) \circ \Phi_{h}^{(h)}(t-h)}
 \end{align*}
  and thus
 \begin{align*}
  J_2(t) = \inner[\Omega]{(\rho^{(h)}  u^{(h)})(t-h)}{\frac{\xi(t)-\xi(t) \circ \Phi_h^{(h)}(t-h)}{h}} \leq \norm[\Omega]{(\rho u^{h})(t-h)} C \Lip(\xi(t))
 \end{align*}
 using \autoref{lem:NSh-m-estimate} (3), which is estimated by the uniform $L^\infty$ bounds of $\rho^{(h)}$ and \autoref{cor:NSkorn}.

 \end{proof}

\subsubsection*{Proof of \autoref{thm:NSexistence}, step 3: Convergence to the limit}

Now we again use our uniform bounds, this time in $h$ to converge to a limit equation for $h\to 0$. By the usual compact embeddings we conclude the following:
\begin{lemma}
There exists a sub-sequence (still written $h\to 0$) and limit functions $\eta : [0,T] \times Q \to \Omega$, $v:\Omega \to \R^n$ such that
\begin{align*}
 \eta^{(h)} &\rightharpoonup \eta & \text{ in }& C_w([0,T];W^{2,q}(Q;\R^n))\\
 \partial_t \eta^{(h)} & \rightharpoonup \partial_t \eta & \text{ in }& L^2([0,T];W^{1,2}(Q;\R^n))\\
 u^{(h)} & \rightharpoonup u & \text{ in }& L^2([0,T];W^{1,2}(\Omega;\R^n))
\end{align*}
\end{lemma}

\begin{proof}
 Proceeding like in previous proofs, using the estimates from \autoref{lem:NSfullIterationAPriori} we have bounds on $E(\eta^{(h)}(t))$ independent of $t$ and $h$ and on $\int_0^T R(\eta^{(h)},\partial_t \eta^{(h)})dt$ independent of $h$. Thus by the assumption on energy and dissipation, the sequence $(\eta^{(h)})_h$ is uniformly bounded in $L^\infty([0,T];W^{2,q}(Q;\R^n)) \cap W^{1,2}([0,T];W^{1,2}(Q;\R^n))$ which allows us to pick a subsequence and a limit $\eta$ in the same space, fullfilling the first set of convergences.
 
 Finally we use the global Korn-inequality \autoref{lem:globalKorn} to show that $\int_0^T \norm[W^{1,2}(\Omega)]{u^{(h)}}^2 dt$ is uniformly bounded and extract a limit $u$ (after possibly another subsequence).
\end{proof}

Exactly by the same argument as was done in \autoref{lem:SOAubinLions} and \autoref{lem:SOMinty}
we get:
\begin{corollary}[Aubin-Lions \& Minty (coupled solid)] \label{cor:NSAubinLionsSolid}
 Let $w^{(h)}: t \mapsto \fint_{t-h}^{t} \partial_t \eta^{(h)}ds$. We have (for a subsequence $h \to 0$)
 \begin{align*}
   w^{(h)} \rightarrow \partial_t \eta \text{ in }L^2([0,T];L^2(Q;\R^n))\text{ and }\eta^{(h)}\to \eta \text{ in } L^q(0,T;W^{2,q}(Q)).
 \end{align*} 
In particular for almost all $t\in [0,T]$ we have
 \[DE(\eta^{(h)}(t)) \to DE(\eta(t)) \text{ in } W^{-2,q}(Q;\R^n).\]
\end{corollary}

We now want to prove a similar result for the Eulerian velocity $u^{(h)}$. While we have an estimate on the time derivative of $\fint_{-h}^0 \rho^{(h)}u^{(h)} (t+s) ds$ in the form of \autoref{prop:NSh-m-estimateFluid}, this estimate is in a dual space of functions which are divergence free on the fluid domain and thus in a $t$ and $h$-dependent space. As a consequence, we are no longer in the realm of classic Aubin-Lions type theorems and instead need to prove a similar result directly.

\begin{lemma}[Aubin-Lions (fluid)] \label{lem:NSAubinLions}
 For each $t \in [0,T]$, $h> 0$ we define $\tilde{u}^{(h)}(t) \in L^2(\Omega;\R^n)$ by
 \begin{align*}
  \tilde{u}^{(h)}(t) := \fint_{-h}^0(\rho^{(h)} u^{(h)})(t+s) ds.
 \end{align*}
 For all $\delta>0$ (sufficiently small) there exists a subsequence $h\to 0$ such that 
 for all $A \in C_0^\infty([0,T]\times \Omega;\R^{n\times n})$
  \begin{align*}
  \int_0^T \inner[\Omega]{(u^{(h)})_\delta}{A\tilde{u}^{h}} dt \to\int_0^T \inner[\Omega]{(u)_\delta}{A\rho u} dt,
 \end{align*}
 where $(\cdot)_\delta$ is the operator defined in \autoref{lem:approxTestFcts}.
\end{lemma}

\begin{proof} 
We begin the proof with a couple of observations. First as the operator $(\cdot)_\delta$ introduced in \autoref{lem:approxTestFcts} is (bounded) and linear, we find that (for a subsequence) $(u^{(h)})_\delta\weakto (u)_\delta$ with $h\to 0$ in $L^2(0,T;W^{1,2}(\Omega))$.

Next, as $\tilde{u}^{(h)}$ is uniformly bounded in $L^\infty(0,T;L^2(\Omega))$ we find that it (after choosing a subsequence) possesses a weak* limit $\tilde{u}\in L^\infty(0,T;L^2(\Omega))$. Since for $\xi\in C^\infty_0([0,T]\times\Omega)$ we have
\begin{align*}
&\int_0^T\inner[\Omega]{\tilde{u}^{(h)}}{\xi}\, dt
=\fint_{-h}^0\int_0^T\inner[\Omega]{(\rho^{(h)} u^{(h)})(t+s)}{\xi(t)}\, dt \,ds
\\
&\quad =\int_0^T\inner[\Omega]{(\rho^{(h)} u^{(h)})(t)}{\fint_{-h}^0\xi(t-s)} \to \int_0^T\inner[\Omega]{\rho u}{\xi}\, dt.
\end{align*}
we also know that $\tilde{u} = \rho u$ almost everywhere.
Next take a sequence $h_i\to 0$ which satisfies all the above convergences, including the strong convergence of \autoref{cor:NSAubinLionsSolid}. Next fix $\epsilon>0$. We aim to show that there is a $N_\epsilon$, such that  for another (non-relabled) subsequence and all $j>i>N_\epsilon$
  \begin{align}
  \label{eq:aim}
  \abs{\int_0^T \inner[\Omega]{(u^{(h_i)}(t))_\delta}{A(\tilde{u}^{(h_i)}(t)-\tilde{u}^{(h_j)}(t))}dt} \leq c\epsilon,
 \end{align}
 which implies the result. 
  Our strategy follows the approach introduced in~\cite[Theorem~6.1]{muhaExistenceRegularityWeak2019}. This means we will choose the regularizing parameter $\delta$ in dependence of $\epsilon$, the time step $\sigma$ in terms of $\delta$ and the number $N_\epsilon$ in dependence of $\sigma$.
  However, in contrast to the approach there, we first use the uniform convergence of $\eta^{(h)}$ that allows for any given $\delta>0$ to take $h$ small enough ($N_\epsilon$ large enough), such that for $\hat{\Omega}_\delta(t)=\bigcap_{i\geq N_\epsilon}\Omega^{(h_i)}(t)$ and $\check{\Omega}_\delta(t)=\bigcup_{i\geq N_\epsilon}\Omega^{(h_i)}(t)$ satisfy
 \begin{align}
 \label{delta1}
 \sup_{t\in [0,T]}\sup_{i\geq N_\epsilon}\Big(\abs{\Omega^{(h_i)}(t)\setminus \hat{\Omega}_\delta(t)}+\abs{\check{\Omega}_\delta(t)\setminus \Omega^{(h_i)}(t)}\Big)\leq \delta
 \end{align}
 Next we may use the approximation introduced in \autoref{lem:approxTestFcts} for  ${u}^{(h_i)}$. The regularity of the domain allows to assume that %
 \begin{align*}
  (\diver((u^{(h_i)})_\delta(t))=0 \text{ in }\check{\Omega}_\delta(t)
  \end{align*}
Moreover \autoref{lem:approxTestFcts} implies that for almost every $t$ and every $m\in \mathbb{N}$ %
\begin{align}
\label{delta2}
\begin{aligned}
   \norm[W^{m,2}(\Omega)]{(u^{(h_i)}(t))_\delta}&\leq c(\delta,m)\norm[W^{1,2}(\Omega)]{u^{(h_i)}(t)}
   \\
    \norm[L^2({[0,T]};W^{1,2}(\Omega))]{(u^{(h_i)})_\delta}&\leq c\norm[L^2({[0,T]};W^{1,2}(\Omega))]{u^{(h_i)}}
   \\
  \norm[L^2({[0,T]};L^2(\Omega))]{(u^{(h_i)})_\delta-u^{(h_i)}}
  &\leq c\delta^\frac{2}{2+n}\norm[L^2({[0,T]};W^{1,2}(\Omega))]{u^{(h_i)}}.
  \end{aligned}
 \end{align}
 The $\delta$ will be chosen later depending on $\epsilon$.
 Furthermore we choose (in dependence of $\delta$) $\sigma > 0$, $N \in \N$ such that $T = N\sigma$. Now for any $k \in \{0,...,N\}$
 \begin{align*}
  \norm[\Omega]{\tilde{u}^{(h)}(\sigma k)}^2 \leq  \fint_{k\sigma-h}^{k\sigma} \norm[\Omega]{(\rho^{(h)}u^{(h)})(t)}^2 dt \leq C\norm[\infty]{\rho^{(h)}} \leq C \rho_{\max}
 \end{align*}
 by the volume density-preserving nature of $\Phi$ and where $C$ is constant derived from \autoref{lem:NSfullIterationAPriori} and $\rho_{\max}$ is a uniform upper bound on the density in fluid and solid, which can be easily derived from the energy bounds. As usual we continue $u$ and $\partial_t \eta$ by their initial data. We can thus use the compact embedding to find a subsequence $h_i\to 0$ such that $\tilde{u}^{(h_i)}(\sigma k)$ converges strongly in $\dot{H}^{-1}(\Omega;\R^n)$ for all $k\in \{0,...,N-1\}$. In particular we can choose the $N_\epsilon$ in such a way that for all $i,j\geq N_\epsilon$ and all $k\in \{0,...,N-1\}$
 \begin{align*}
 \norm[H^{-1}(\Omega)]{\tilde{u}^{(h_i)}(\sigma k)-\tilde{u}^{(h_j)}(\sigma k)}\leq \epsilon.
 \end{align*}
 Now we estimate for $t\in [\sigma k,\sigma(k+1))$
 \begin{align*}
 &\phantom{{}={}}\inner[\Omega]{(u^{(h_i)}(t))_\delta}{A(\tilde{u}^{(h_i)}(t)-\tilde{u}^{(h_j)}(t))}
 \\
 &= %
  \inner[\Omega]{(u^{(h_i)})_\delta(t)}{A\tilde{u}^{(h_i)}(t) - A\tilde{u}^{(h_i)}(\sigma k)}
   +\inner[\Omega]{(u^{(h_i)})_\delta(t)}{A(\tilde{u}^{(h_i)}(\sigma k)-\tilde{u}^{(h_j)}(\sigma k))}
   \\
    &\phantom{{}={}}
+ \inner[\Omega]{(u^{(h_i)})_\delta(t)}{A\tilde{u}^{(h_j)}(\sigma k)- A\tilde{u}^{(h_j)}(t) }
=:I(t)+II(t)+III(t).
 \end{align*}
 Now for $i,j\geq N_\epsilon$ we find using the
bounds on the approximability and the strong convergence in $H^{-1}$ that
 \[
\int_0^T II(t)\, dt\leq C\epsilon.
\]
The other two terms are estimated using the continuity in time of $ \tilde{u}^{(h_i)}$.  Indeed we find that
\begin{align*}
 \partial_\theta \tilde{u}^{(h_i)}(\theta,y)  &= \partial_\theta \left(\fint_{-h}^0(\rho^{(h_i)} u^{(h_i)})(\theta+s) \circ \Phi_s^{(h_i)}(\theta)   ds \right) \\
 &=  \frac{1}{h}\partial_\theta \left(\int_{\theta-h}^\theta(\rho^{(h_i)} u^{(h_i)})(s) ds \right) 
 =  \frac{ (\rho^{(h_i)}u^{(h)})(\theta) -(\rho^{(h_i)}u^{(h_i)})(\theta-h) }{h}
\end{align*}
and thus
\begin{align*}
\abs{\int_0^T I(t)\, dt} &\leq \sum_{k}\int_{\sigma k}^{(\sigma+1)k} \int_{\sigma k}^t \abs{\inner[\Omega]{(u^{(h_i)}(t))_\delta}{A \frac{ (\rho^{(h)}u^{(h)})(\theta) -(\rho^{(h)}u^{(h)})(\theta-h) }{h}}}d \theta dt
\\
&=\sum_{k}\int_{\sigma k}^{\sigma(k+1)} \int_\theta^{\sigma (k+1)} \abs{\inner[\Omega]{(u^{(h_i)}(t))_\delta }{A \frac{ (\rho^{(h)}u^{(h)})(\theta) -(\rho^{(h)}u^{(h)})(\theta-h) }{h} }} dt \,d \theta \\
&\leq \sum_{k}\int_{\sigma k}^{\sigma(k+1)} \int_0^{\sigma} \abs{\inner[\Omega]{(u^{(h_i)}(\theta+s))_\delta }{A \frac{ (\rho^{(h)}u^{(h)})(\theta) -(\rho^{(h)}u^{(h)})(\theta-h) }{h} }} \,ds\,d \theta  \\
&\leq \norm[\infty]{A} \int_{0}^{\sigma} \int_0^T \abs{ \inner[\Omega]{(u^{(h_i)}(\theta+s))_\delta }{ \frac{ (\rho^{(h)}u^{(h)})(\theta) -(\rho^{(h)}u^{(h)})(\theta-h) }{h} }}\,dt \, ds \\
&\leq \norm[\infty]{A} \int_0^\sigma \norm[L^2({[0,T]};W^{m,2}(\Omega))]{(u^{(h_i)}(.+s))_\delta} ds \leq \norm[\infty]{A} C_\delta \sigma \norm[L^2({[0,T]};W^{1,2}(\Omega))]{u^{(h_i)}}
\end{align*}
using \autoref{prop:NSh-m-estimateFluid}.

Since an analogous estimate on $(III)(t)$ is possible, we find \eqref{eq:aim} by choosing $\sigma$ small enough in dependence of $\epsilon$.  
\end{proof}
Observe that in particular, due to the strong convergence of $\partial_t\eta^{(h)}$ (and consequently $u$ on $\eta^{h}(t,Q)$)
 \begin{align}
 \label{cor:NSAubin}
  \int_0^T \inner[\Omega^{h}(t)]{(u^{(h)})_\delta}{A\tilde{u}^{h}} dt \to \rho_f\int_0^T \inner[\Omega(t)]{(u)_\delta}{A u} dt,
 \end{align}
 for all $A\in C^\infty_0(\Omega)$.

Finally we show that the limit is a weak solution to the full system \autoref{def:NSweakSol} and thus prove \autoref{thm:NSexistence}:

\subsubsection*{Proof of \autoref{thm:NSexistence}, Step 3a: Passing to the limit with the coupled PDE}
In the following we assume that $T$ is small enough, such that a sequence of approximate solutions $(\eta^{(h)},u^{(h)})$ exist on the interval $[0,T+h]$. Later it will be discussed how to prolong the solution up to the point of contact.

As before in \autoref{prop:QSconvergence} and \autoref{thm:TDexistence}, we use \autoref{lem:approxTestFcts} to restrict ourselves to test functions $\xi \in C_0^\infty(\Omega;\R^n)$ with $\diver \xi = 0$ in a neighborhood of $\Omega(t)$. We then construct $\phi^{(h)} := \xi \circ \eta^{(h)}$ and proceed to the limit with the terms of time discrete approximations. As before we get
\begin{align*}
 &\phantom{{}={}}\int_0^T \inner{DE(\eta^{(h)}(t))}{\phi^{(h)}} + \inner{DR_h(\eta^{(h)}(t),\partial_t\eta^{(h)})}{\phi^{(h)}} dt \\
 &\to \int_0^T \inner{DE(\eta(t))}{\phi} + \inner{DR(\eta(t),\partial_t\eta)}{\phi} dt
\end{align*}
as well as
\begin{align*}
 \int_0^T \rho_s \inner[Q]{f \circ \eta^{(h)}}{\phi^{(h)}} + \rho_f \inner[\Omega^{(h)}(t)]{f}{\xi} dt \to 
 \int_0^T \rho_s \inner[Q]{f \circ \eta}{\phi} + \inner[\Omega(t)]{f}{\xi} dt
\end{align*}
and 
\begin{align*}
 \int_0^T \nu \inner[\Omega^{(h)}(t)]{\nablasym u^{(h)}}{\nablasym \xi} dt \to \int_0^T \nu \inner[\Omega(t)]{\nablasym u}{\nablasym \xi} dt.
\end{align*}

What is left are the inertial terms. For the solid as in \autoref{thm:SOexistence}, we transfer the difference quotient onto the test function to get
\begin{align*}
 &\phantom{{}={}} \int_0^T  \inner{\frac{\partial_t \eta^{(h)}(t) - \partial_t \eta^{(h)}(t-h)}{h}}{\phi^{(h)}(t)} dt\\
 &=- \int_0^T \inner{\partial_t \eta^{(h)}(t)}{\frac{\phi^{(h)}(t+h)-\phi^{(h)}}{h}} dt \to - \int_0^T \inner{\partial_t \eta(t)}{\partial_t \phi(t)} dt.
\end{align*}

For the fluid we do the same, but we have to take into account the flow map $\Phi^{(h)}$. In particular, we rewrite
\begin{align*}
 &\phantom{{}={}}\int_0^T \inner[\Omega^{(h)}(t-h)]{\frac{u^{(h)}(t) \circ \Phi_h^{(h)}(t-h) - u^{(h)}(t-h)}{h} }{\xi(t) \circ \Phi_h^{(h)}(t-h)} dt\\
 &=-\int_0^T \inner[\Omega^{(h)}(t)]{u^{(h)}(t)}{\frac{\xi(t+h) \circ \Phi_h^{(h)}(t)- \xi(t)}{h}} dt  \\
 &=-\int_0^T \inner[\Omega^{(h)}(t)]{(u^{(h)})_\delta(t)}{\frac{\xi(t+h) \circ \Phi_h^{(h)}(t)- \xi(t)}{h}} dt \\
 &+ \int_0^T \inner[\Omega^{(h)}(t)]{(u^{(h)})_\delta(t)-u^{(h)}(t)}{\frac{\xi(t+h) \circ \Phi_h^{(h)}(t)- \xi(t)}{h}} dt=: -I^{\delta,h}+II^{\delta,h}
 \end{align*}
 where $(u^{(h)})_\delta$ is a regularization in space, as used in \autoref{lem:NSAubinLions} (defined in \autoref{lem:approxTestFcts}). As the right part in the scalar product of $II^{\delta,h}$ is uniformly bounded in $L^\infty([0,T];L^2(\Omega(t);\R^n))$ using \eqref{delta2} we know that this term vanishes for $\delta \to 0$ (uniformly in $h$). For the first term we expand
 \begin{align*}
&\phantom{{}={}} I^{\delta,h} = \int_0^T \inner[\Omega^{(h)}(t)]{(u^{(h)})_\delta(t)}{\fint_0^{h} \partial_s \left(\xi(t+s) \circ \Phi_s^{(h)}(t)\right) ds } dt \\
 &=  \int_0^T \inner[\Omega^{(h)}(t)]{(u^{(h)})_\delta(t)}{\fint_0^{h} \left(\partial_t \xi(t+s) -u^{(h)}(t+s) \cdot \nabla \xi(t+s) \right) \circ \Phi_s^{(h)}(t) ds } dt \\
 &=  \int_0^T \fint_0^{h} \inner[\Omega^{(h)}(t+s)]{(u^{(h)})_\delta(t)  \circ \Phi_{-s}^{(h)}(t+s)}{ \partial_t \xi(t+s) -u^{(h)}(t+s) \cdot \nabla \xi(t+s) }ds\, dt \\
 &=  \int_0^T \fint_0^{h} \inner[\Omega^{(h)}(t+s)]{(u^{(h)})_\delta(t)}{ \partial_t \xi(t+s) -u^{(h)}(t+s) \cdot \nabla \xi(t) }ds\, dt \\
 &+\int_0^T \fint_0^{h} \inner[\Omega^{(h)}(t+s)]{(u^{(h)})_\delta(t)}{u^{(h)}(t+s) \cdot \nabla (\xi(t) - \xi(t+s))}ds\, dt \\
 &+  \int_0^T \fint_0^{h} \inner[\Omega^{(h)}(t+s)]{(u^{(h)})_\delta(t)  \circ \Phi_{-s}^{(h)}(t+s)-(u^{(h)})_\delta(t)}{ \partial_t \xi(t+s) -u^{(h)}(t+s) \cdot \nabla \xi(t+s) }\!\!\!\!\!ds\, dt 
\end{align*}
Since $\norm[L^\infty(\Omega)]{\nabla (\xi(t) - \xi(t+s))}\leq
 ch\norm[{L^\infty([0,T]\times\Omega)}]{\partial_t\nabla\xi} $ the second term converges to zero with $h\to 0$.
For the third term above we may use \autoref{lem:NSh-m-estimate} to show that the $L^2$-norm of the left hand side is bounded by $h \Lip((u^{(h)})_\delta(t)) \leq h C_\delta \norm[W^{1,2}]{u^{(h)}(t)}$ which proves that the term vanishes for $h\to 0$. For the left term we wish to use \autoref{lem:NSAubinLions}. For that we take $A_\delta\in C^0(0,T;C^\infty_0(\hat{\Omega}_\delta)$, such that $A_\delta(t)\to \chi_{\Omega(t)}$ almost everywhere. 
Hence we find by \eqref{cor:NSAubin} that
\begin{align*}
\lim_{h\to 0} I^{\delta,h}&=\lim_{h\to 0}\int_0^T \fint_0^{h} \inner[\Omega^{(h)}(t+s)]{(u^{(h)})_\delta(t)}{ \partial_t \xi(t+s) -u^{(h)}(t+s) \cdot \nabla \xi(t) }ds\, dt
\\
&= \int_0^T\inner[\Omega(t)]{(u)_\delta}{\partial_t\xi-u\cdot \nabla \xi A_\delta(t)}\, dt
\\
&\quad  + \lim_{h\to 0}\int_0^T \fint_0^{h} \inner[\Omega]{(u^{(h)})_\delta(t)}{ u^{(h)}(t+s) \cdot \nabla \xi(t)(A_\delta(t)-\chi_{\Omega^{(h)}}(t+s)) }\, ds\, dt.
\end{align*}
The last term is estimated by H\"older's inequality and Sobolev embedding. Indeed, for $a<\frac{n}{n-2}$ we find by \eqref{delta2} 
 \begin{align*}
&\abs{\int_0^T \fint_0^{h} \inner[\Omega]{(u^{(h)})_\delta(t)}{ u^{(h)}(t+s) \cdot \nabla \xi(t)(A_\delta(t)-\chi_{\Omega^{(h)}}(t+s)) }\, ds\, dt}
\\
&\quad \leq \int_0^T\norm[L^{2a}(\Omega)]{(u^{(h)})_\delta(t)}\fint_0^h\norm{u^{(h)}(t+s)}\norm[L^{2a'}(\Omega)]{(A_\delta(t)-\chi_{\Omega^{(h)}}(t+s))}
\\
&\quad \leq c\norm[{L^2([0,T];W^{1,2}(\Omega))}]{(u^{(h)})_\delta(t)}\sup_{t\in T}\bigg(\fint_0^h\norm{u^{(h)}(t+s)}^2\,ds\bigg)^\frac12 
\\
&\qquad \times \bigg(\fint_0^h\norm[{L^2([0,T];L^{2a'}(\Omega))}]{(A_\delta(t)-\chi_{\Omega^{(h)}}(t+s))}^2\, ds\bigg)^\frac12
\\
&\quad \leq c\bigg(\fint_0^h\norm[{L^2([0,T];L^{2a'}(\Omega))}]{(A_\delta(\cdot)-\chi_{\Omega^{(h)}}(\cdot+s))}^2\, ds\bigg)^\frac12.
\end{align*}
Since by the strong convergence of $\eta^{(h)}$ and $\Omega^{(h)}$ we find that 
\[
\lim_{h\to 0}\bigg(\fint_0^h\norm[{L^2([0,T];L^{2a'}(\Omega))}]{(A_\delta(\cdot)-\chi_{\Omega^{(h)}}(\cdot+s))}^2\, ds\bigg)^\frac12=\norm[{L^2([0,T];L^{2a'}(\Omega))}]{(A_\delta-\chi_{\Omega}))}
\]
we finally gain by passing with $\delta \to 0$ that 
\begin{align*}
\lim_{\delta\to 0}\lim_{h\to 0}(-I^{\delta,h}+II^{\delta,h})
&=
-\int_0^T\inner[\Omega(t)]{u}{\partial_t\xi-u\cdot \nabla \xi}\, dt.
\end{align*}
Thus we have shown that we obtain the right equation in the limit:
\begin{align}
\label{eq:limit-div}
\begin{aligned}
& \int_0^T -\rho_s \inner[Q]{\partial_t \eta}{\partial_t \phi} - \rho_s \inner[\Omega(t)]{v}{\partial_t \xi- v \cdot \nabla \xi} + \inner{DE(\eta)}{\phi} + \inner{D_2R(\eta,\partial_t \eta)}{\phi} + \nu \inner[\Omega(t)]{\nablasym v}{\nablasym \xi} dt\\
  &\,\,= \int_0^T  \rho_s \inner[Q]{f\circ \eta}{\phi} + \rho_f \inner[\Omega(t)]{f}{\xi} dt - \rho_s \inner[Q]{\eta_*}{\phi(0)} - \rho_f \inner[\Omega(0)]{v_0}{\xi(0)}.
\end{aligned}
\end{align}

\subsubsection*{Proof of \autoref{thm:NSexistence}, Step 3b: Reconstruction of the pressure}

As we do not want to consider the time-derivatives of the operator $\bog_t$ we cannot go along the same lines as in the proof of \autoref{thm:QSexistence}. Instead we have to proceed in a global manner.
We construct the pressure as a distribution.
 
 Let $\psi\in C^\infty_0([0,T]\times \Omega)$.%
 Take $\bog$ to be the operator of Theorem~\ref{thm:sohr} with respect to the domain $\Omega$. To apply this operator to $\psi$, we need to normalize its mean by picking a $\tilde{\psi} \in C_0^\infty([0,T]\times \Omega)$ with $\supp(\tilde{\psi}(t)) \cap \Omega(t) = \emptyset$ and $\int_\Omega \psi(t) dy = -\int_\Omega \tilde{\psi}(t) dy$ for all $t\in [0,T]$.
 
 Now let $\xi(t) := \bog(\psi(t)+\tilde{\psi}(t))$, $\phi(t,x) := \xi(t,\eta(t,x))$ and define a linear operator
 by
 \begin{align*}
  P(\psi)&:=  \int_0^T \inner[Q]{DE(\eta(t))}{\phi} + \inner[Q]{D_2R(\eta(t),\partial_t \eta(t))}{\phi} 
  + \nu \inner[\Omega(t)]{\nablasym u}{\nablasym \xi} \\ \nonumber &\qquad - \rho_f \inner[\Omega(t)]{f}{\xi} - \rho_s \inner[Q]{f\circ \eta}{\phi} - \rho_s\inner[Q]{\partial_t \eta}{\partial_t \phi} - \rho_f \inner[\Omega(t)]{u}{ \partial_t \xi-u\cdot \nabla \xi} dt.
 \end{align*}
Note that $P(\psi)$ is independent of the choice of $\tilde{\psi}$: Assume that $\tilde{\psi}_1$ and $\tilde{\psi}_2$ are two such choices and $\xi_1$ and $\xi_2$ the corresponding functions. Then $\xi_1-\xi_2 = \bog(\tilde{\psi}_1-\tilde{\psi}_2)$ has divergence $0$ on $\Omega(t)$ and thus the above integral is the same because of \eqref{eq:limit-div}. In particular if $\supp(\psi(t))\subset \eta(t,Q)$ (for all $t$), we may choose $\tilde{\psi}\equiv \psi$ which implies (by the linearity of $\bog$) that $P(\psi)=0$.
Hence $\supp(P)\subset [0,T]\times \overline{\Omega(t)}$.
 
 Furthermore it can be estimated by first noting that
 \begin{align*}
  &\quad\int_0^T \inner[Q]{DE(\eta(t))}{\phi} + \inner[Q]{D_2R(\eta(t),\partial_t \eta(t))}{\phi} 
  + \nu \inner[\Omega(t)]{\nablasym u}{\nablasym \xi}  - \rho_f \inner[\Omega(t)]{f}{\xi} - \rho_s \inner[Q]{f\circ \eta}{\phi} dt\\
  &\leq T \sup_{t\in[0,T]} \norm[W^{-2,q}(Q)]{DE(\eta(t))} \norm[{L^1([0,T],W^{2,q}(Q))}]{\phi(t)}  \\
  &+ \int_0^T \norm[W^{-1,2}(Q)]{D_2R(\eta(t),\partial_t \eta(t))} \norm[W^{1,2}(Q)]{\phi} + \norm[\Omega(t)]{\nablasym u} \norm[\Omega(t)]{\nablasym \xi}
  + c\norm[\infty]{f} (\norm[Q]{\phi} + \norm[\Omega(t)]{\xi}) dt \\
  &\leq C  \norm[{L^1([0,T],W^{2,q}(Q))}]{\phi} + \norm[{L^2([0,T];W^{1,2}(\Omega))}]{\xi}
 \end{align*}
 via the known bounds on the terms in the weak equation. Finally using \autoref{prop:W2qIsomorph} we know that $\norm[W^{2,q}(Q)]{\phi} \leq C \norm[W^{2,q}]{\xi}$ and we can use the properties of the Bogovski\u{\i}-operator to estimate this by
 \begin{align*}
  &\leq C \norm[{L^1([0,T],W^{2,q}(Q))}]{\psi + \tilde{\psi}} + C\norm[{L^2([0,T];L^2(\Omega))}]{\psi+\tilde{\psi}} \\ &\leq C \norm[{L^1([0,T],W^{2,q}(Q))}]{\psi} + C\norm[{L^2([0,T];L^2(\Omega))}]{\psi}
 \end{align*}
 where for the last inequality we note that $\tilde{\psi}(t)$ can be chosen as a multiple of a fixed $C^\infty_0$-function and thus its norm only needs to depend on $\abs{\int_{\Omega} \psi(t) dy} \leq c\norm[\Omega]{\psi(t)}$. 
 Additionally for the other remaining terms we have
 \begin{align*}
  \abs{\int_0^T \inner[Q]{\partial_t \eta}{\partial_t \phi}} &\leq \norm[L^2({[0,T]}\times Q)]{\partial_t \eta} \norm[W^{1,2}({[0,T]};L^2(Q))]{\phi} \\
  \abs{\int_0^T \inner[\Omega(t)]{u}{\partial_t \xi}} &\leq \norm[L^2({[0,T]}\times \Omega)]{u} \norm[W^{1,2}({[0,T]};L^2(\Omega))]{\xi}\\
  \abs{\int_0^T \inner[\Omega(t)]{u}{u \cdot \nabla \xi}} &\leq %
   \norm[L^{a}({[0,T]};L^{b}(\Omega))]{u}^2 \norm[L^{a'}\infty({[0,T]};W^{1,b'}(\Omega))]{\xi}
 \end{align*}
 where $a,b\in (1,\infty)$ are chosen in such a way that $\abs{u}^2\in L^{a}([0,T],L^b(\Omega))$, which is possible since $\abs{u}^2\in L^{\infty}([0,T],L^1(\Omega))\cap L^{1}([0,T],L^p(\Omega))$ (with $p=\frac{n}{n=2}$ for $n>2$ and arbitrarily large for $n=2$). Now bounding the norms of $\xi$ and $\phi$ in terms of $\psi$ as before proves that $P \in \mathcal{D}' ([0,T]\times\Omega)$.
 Thus $p$ is well defined via that operator and expanding
 \begin{align*}
  \int_0^T \inner{\nabla p}{\xi} dt = P(\diver \xi)
 \end{align*}
 proves that it fulfills the right equations for $\xi\in C^\infty([0,T]\times \Omega)$. %
 Moreover, it can be decomposed into 
 \begin{align}
 \label{eq:press}
 p\in L^\infty([0,T],W^{-1,q}(\Omega))+L^2([0,T]\times \Omega)+W^{-1,2}([0,T],W^{-1,2}(\Omega))\cap L^{a'}([0,T],W^{-1,b'}(\Omega)).
 \end{align}

\subsubsection*{Proof of \autoref{thm:NSexistence}, Step 4: Energy inequality \& maximal interval of existence}

Above, we have shown existence of coupled weak solutions $u,\eta$ on $[0,T]$ for some $T>0$. As before we can now pick a maximal interval $[0,T_{\max})$ and use the energy bounds to conclude that either $T_{\max} = \infty$ or there exists a limit $\eta(T_{\max}) \in \partial \mathcal{E}$. 

Finally, we observe that \eqref{eq:energ} follows by\autoref{lem:NSfullIterationAPriori}. This lemma combined with \autoref{cor:NSkorn} also implies the following a-priori estimates for the solution,
  \begin{align}
  \label{eq:apri}
  &\phantom{{}={}} E(\eta(t)) + c  \norm[\Omega^{(h)}(t)]{u^{(h)}}^2 + \norm[Q]{\partial_t \eta(t)}^2 + \int_0^{t}  \norm[W^{1,2}(Q)]{\partial_t \eta} + \norm[W^{1,2}(\Omega(s))]{v}^2\, ds \leq C,
\end{align}
for all times $t$ up to the point of collision.

\subsection{Remarks on the full problem}

\begin{remark}[On Lagrangian and Eulerian formulation] \label{rem:NSlagrangianVsEulerian}
 In the preceding proofs, we switched between the Lagrangian and the Eulerian point of view several times. While in the end, for the final equation we have to treat solid as Lagrangian and fluid as Eulerian, we are free to change from one to the other as long as $h>0$ during the proof.
 
 We used this prominently in defining the global Eulerian velocity $u$, as doing so made it easier to talk about convergence. But it should be noted that the same can be equally done in reverse. In the proof of the time-delayed problem \autoref{thm:TDexistence}, every time we used $v$, we could have similarly considered $\partial_t \Phi$ or its difference quotient respectively. This way the whole proof can be rewritten in terms of $\Phi$, eliminating the need $v$ and $u$ completely.
 
 This actually continues as long as $h>0$. It is only at the very last moment, where we take the limit $h\to0$ that we are no longer guaranteed existence of a flow map $\Phi$ and have to introduce an Eulerian velocity to conclude the proof.
 
 In particular those two points of view are not mutually exclusive and both can be used at the same time. This should be kept in mind when trying to extend the proof by coupling the system with additional equations, which might prefer one viewpoint or the other.
\end{remark}

\begin{remark}[On the material derivative] \label{rem:NSmaterialDerivative}
 Comparing an Eularian quantity defined on $\Omega$ with itself at different times but at the same point is not physical, and as a direct consequence so are time derivatives as $\partial_t u$ on their own. Instead one has to take into account their corresponding flow. In the case of $u$ this means that the correct term is the material derivative $\partial_t u + u\cdot \nabla u$.
 
 For most of the proof, following this principle lead to a much more natural mathematical treatment as well. However in the end, we had to make one deviation when proving the Aubin-Lions Lemma \ref{lem:NSAubinLions}. Here a more natural quantity would be to consider the average velocity of a particle instead, i.e.
 \begin{align*}
  \hat{u}^{(h)}(t) := \fint^0_{-h} u^{(h)}(t+s) \circ \Phi_{s}(t) \,ds.
 \end{align*}
 
 This is not only hinted at by the preceding \autoref{prop:NSh-m-estimateFluid}, where some additional correction is needed to turn the natural estimate on $u(t) - u(t-h) \circ \Phi_{-h}$ into a less natural estimate on $u(t) - u(t-h)$, but it is also the quantity that actually occurs at the end of Step 3a of the proof of \autoref{thm:NSexistence} and where we thus need to make a similar correction to undo the previous one.
 
 However while seemingly more natural, $\hat{u}^{(h)}$ has less regularity, as all its derivatives in space and in time involve terms of the form $\nabla \Phi_s$ for which we are unable to obtain useful bounds.

\end{remark}

\appendix
\section{}

\subsection{Some technical lemmata}

The following lemmata are not surprising, but technical. Likely most of them are known already, but we were unable to find a good source for the specific versions required.

\begin{lemma}[Expansion of the determinant] \label{lem:detExpansion}
 Let $A \in \R^{n\times n}$. Then
 \begin{align*}
  \det(I + \tau A) = 1 + \tau \tr A + \sum_{l=2}^n \tau^l M_l(A)
 \end{align*}
  where $M_l(A)$ is a homogeneous polynomial of degree $l$ in the entries of $A$. Note that this is a finite sum.
\end{lemma}

\begin{proof}
 Consider the Leibniz formula
 \begin{align*}
  \det(I + \tau A) = \sum_{\pi \in S_n} \operatorname{sgn} (\pi) \prod_{i=1}^n (\delta_{i,\pi(i)} + \tau A_{i,\pi(i)})
 \end{align*}
 where $S_n$ is the set of permutations of $\{1,...,n\}$. We expand the product and order the terms by the exponent of the factor $\tau^l$ and thus by the number of terms $\tau A_{i,\pi(i)}$ that are taken while expanding the product. This will directly yield the homogeneous polynomial $M_l(A)$.
 
 For $\tau^0$ and $\tau^1$, the only non-zero terms occur for $\pi = id$, otherwise there will be at least one factor $\delta_{i,\pi(i)}$ for $i\neq \pi(i)$. For $\tau^0$ this means we only choose the $\delta_{i,i}$ terms and for $\tau^1$ we can choose any one $\tau A_{i,i}$-term. Thus $M_0(A) = 1$ and $M_1(A) = \tr A$.
\end{proof}

\begin{lemma}[Invertible maps] \label{lem:W2qInvertible}
 Let $\eta \in W^{2,q}(Q;\R^n)$ be injective, such that $\det \nabla \eta > \epsilon_0 > 0$ for some $\varepsilon < 1$ and $\eta|_P = \gamma$. Then $\eta^{-1} \in W^{2,q}(\eta(Q);\R^n)$ and $\norm[W^{2,q}(\eta(Q))]{\eta^{-1}} \leq c \frac{\norm[W^{2,q}(Q)]{\eta}^{2n-1}}{\epsilon_0^2}$ where $c$ depends only on $q,Q,\gamma$ and $n$.
\end{lemma}

\begin{proof}
 Due to the condition on the determinant $\nabla \eta$ is invertible and furthermore, we have the well known formula
 \begin{align*}
  \nabla (\eta^{-1}) = (\nabla \eta)^{-1} \circ \eta^{-1} = \frac{(\cof \nabla \eta)^T}{\det \nabla \eta} \circ \eta^{-1}
 \end{align*}
 Now we take the derivative of $\nabla (\eta^{-1}) \circ \eta$ to get
 \begin{align*}
  (\nabla^2 (\eta^{-1})) \circ \eta \cdot \nabla \eta = \nabla \left(\nabla (\eta^{-1}) \circ \eta\right) = \frac{\nabla (\cof \nabla \eta)^T}{\det \nabla \eta} - \frac{(\cof \nabla \eta)^T \otimes (\cof \nabla \eta)}{(\det \nabla \eta)^2} \cdot \nabla^2 \eta
 \end{align*}
 Integrating then yields
 \begin{align*}
  &\phantom{{}={}} \int_{\eta(Q)} \abs{(\nabla^2 (\eta^{-1}))}^q dy = \int_Q \abs{(\nabla^2 (\eta^{-1})) \circ \eta}^q \det \nabla \eta dx\\
  &= \int_Q \abs{\frac{\nabla (\cof \nabla \eta)^T}{\det \nabla \eta} - \frac{(\cof \nabla \eta)^T \otimes (\cof \nabla \eta)}{(\det \nabla \eta)^2} \cdot \nabla^2 \eta}^q \det \nabla \eta dx
  \intertext{Now the determinants in the denominators can be estimated by $\epsilon_0$, while the numerators all consist of one second derivative multiplied with a number of first derivatives, which we can estimate by their supremum.}
  &\leq \int_Q C \left( \frac{\abs{\nabla^2 \eta} \norm[\infty]{\nabla \eta}^{n-2}}{\epsilon_0^{1-\frac{1}{q}}} + \frac{\norm[\infty]{\nabla \eta}^{2n-2} \abs{\nabla^2 \eta} }{\epsilon_0^{2-\frac{1}{q}}} \right)^q dx \leq C \frac{\norm[L^q(\eta)]{\nabla^2 \eta}^q \norm[\infty]{\nabla \eta}^{q(2n-2)} }{\varepsilon^{2q}}
 \end{align*}
 Using the Morrey embedding $\norm[\infty]{\nabla \eta} \leq \norm[C^\alpha]{\nabla \eta}\leq C \norm[W^{2,q}(Q)]{\eta}$ and collecting the terms then shows
 \begin{align*}
  \norm[L^q(\eta(Q)]{\nabla^2 (\eta^{-1})} \leq C\frac{\norm[W^{2,q}]{\eta}^{2n-1}}{\epsilon_0^2}
 \end{align*}
 Finally, as we have partially known boundary values $\eta^{-1}|_{\gamma(P)} = \gamma^{-1}$, the lower order estimates follow from a Poincaré-inequality.
\end{proof}

For the next result we an interpolation.
We begin by recalling the following result, which follows for instance from the interpolation estimate in~\cite[Theorem 2.13]{triebelFunctionSpacesLipschitz2002} which implies combined with the usual Sobolev embeddings~\cite[Theorem~2.5.1 and Remark~2.5.2]{Ziemer89book} that for all $m \in [0,\infty)$, $\alpha \in [1,\infty)$ and all Lipschitz domains $\Omega$ satisfy
\[
  m\leq l  \qquad\text{and}\qquad  \frac{1}{\alpha}-\frac{m}{n}\geq \frac{1}{\gamma}-\frac{l}{n} = \frac{k-l}{ka}+\frac{l}{2k} - \frac{l}{n}
\]
the estimate
\begin{align}
\label{eq:triebel}
\norm[{W^{m,\alpha}}]{g}\leq C\norm[{W^{k,2}}]{g}^\frac{l}{k}\norm[{L^{a}}]{g}^\frac{k-l}{k} .
\end{align}
\begin{lemma}
\label{lem:horriblelemma}
Let $Q\subset \R^n$ be a bounded Lipschitz domain, $q>n$  and the number $k\in \N$ be defined as 
\begin{align}
\label{eq:k}
k=2+\frac{n+1}{2}\text{ if } n \text{ is odd}, \quad k=3+\frac{n}{2}\text{ if } n \text{ is even}.
\end{align}
For every $\eta\in W^{2,q}(Q)\cap W^{k,2}(Q)$, there is a constant $c$ depending on $\Omega,n,k$ and $\norm[W^{2,q}(Q)]{\eta}$ such that
\begin{align*}
\sum_{l=1}^k\sum_{a\in \{1,...,n\}^l}\norm{\nabla^{k-l} \Pi_{i=1}^{l}\partial_{a_i}\eta}
\leq c\norm[W^{k,2}(Q)]{\eta}.
\end{align*}
\end{lemma}
\begin{proof}

 Observe, that since $\nabla \eta$ is uniformly bounded by the $W^{2,q}(Q)$ norm, we find that
\begin{align*}
\sum_{a\in \{1,...,n\}^l}\norm{\nabla^{k-l} \Pi_{i=1}^{l}\partial_{a_i}\eta}
\leq c\sum_{\beta\in \N^{l}_0,\ \abs{\beta}=k-l}\norm{\Pi_{i=1}^{l}\abs{\nabla^{\beta_i}\nabla \eta}}.
\end{align*}
The estimate for $l=1$ is direct. Next assume, that $l\geq 2$ and $\beta\in \N^{l}_0,\ \abs{\beta}=k-l$ such that
all $\beta_i\neq 1$. Now by H\"older's and Young's inequality
\begin{align*}
\Pi_{i=1}^{l}\norm{\nabla^{\beta_i}\nabla \eta}\leq c\sum_{\beta_i>1}
\norm[\frac{2(k-l)}{\beta_i}]{\nabla^{\beta_i-1}\nabla^2 \eta}^\frac{k-l}{\beta_i}
\end{align*}
Next we seek to interpolate $\nabla^2\eta$ in between $W^{2,q}$ and $W^{k-2,2}$. For that we wish to use \eqref{eq:triebel}. Hence we have to prove that
\begin{align}
\label{eq:interpol}
\frac{\beta_i}{2(k-l)}\geq \frac{k-1-\beta_i}{q(k-2)}+\frac{\beta_i-1}{2(k-2)}.
\end{align}
Since $l\geq 2$ we find (by multiplying \eqref{eq:interpol} with $k-2$) that \eqref{eq:interpol} holds true whenever
\[
\frac{1}{2}\geq \frac{k-\beta_i-1}{n}\Leftrightarrow n\geq 2(k-\beta_i-1),
\]
which is satisfied by the definition of $k$ as long as $\beta_i\geq 2$.

Hence we may use \eqref{eq:interpol}
\begin{align*}
\norm[\frac{2(k-l)}{\beta_i}]{\nabla^{\beta_i-1}\nabla^2 \eta}^\frac{k-l}{\beta_i}\leq c\norm[L^q(Q)]{\nabla^2\eta}^{\frac{k-l}{\beta_i}\frac{k-1-\beta_i}{k-2}}\norm[W^{k-2,2}(Q)]{\nabla^2\eta}^{\frac{k-l}{\beta_i}\frac{\beta_i-1}{k-2}}
\leq c \norm[W^{k-2,2}(Q)]{\nabla^2\eta},
\end{align*}
using that $\frac{k-l}{\beta_i}\frac{\beta_i-1}{k-2}\leq 1$.

The last case is proved inductively. First with no loss of generality we take $\beta_1=1$. Then $\sum_{i=2}^l\beta_i\leq k-l-1$ and using H\"older's inequality and Sobolev embedding implies
\begin{align*}
\norm{\abs{\nabla^2\eta}\Pi_{i=1}^{l-1}\abs{\nabla^{\beta_i}\eta}}\leq \norm[n]{\nabla^2\eta}\norm[\frac{2n}{n-2}]{\Pi_{i=1}^{l-1}\abs{\nabla^{\beta_i}\eta}}
\leq c\norm[W^{1,2}(Q)]{\Pi_{i=1}^{l-1}\abs{\nabla^{\beta_i}\eta}}.
\end{align*}
If now $\beta_i\neq 1$ for all $i>1$, the estimate follows by the above case for the pair $\nabla^{k-(l-1)}\Pi_{i=1}^{l-1}\partial_{a_i}\eta$. If not, we may assume that $\beta_2=1$ and can repeat the argument again. After at most $l$ steps (in which case $k\geq 2l$), we get the result.

\end{proof}

\begin{proposition}[Space isomorphisms] 
\label{prop:isos}

\label{prop:W2qIsomorph}
 Let $\eta \in W^{2,q}(Q;\R^n)$ such that $\det \nabla \eta > \epsilon_0 > 0$ and $\eta|_P = \gamma$. Then the map
 \[\eta^\#: \xi \mapsto \xi \circ \eta; W^{2,q}(\eta(Q);\R^n) \to W^{2,q}(Q;\R^n)\]
 is a linear vector space-isomorphism with operator-norm $\norm{\eta^\#} \leq C\norm[W^{2,q}(Q)]{\eta}^2/\epsilon_0^{1/q}$ where $c$ does only depend on $q,Q,\gamma$ and $n$.
 Moreover, if $q>n$ and additionally $\eta\in W^{k,2}(Q)$ and $\xi\in C^k(\eta(Q))$, for $k$ defined in \eqref{eq:k}. Then,
 \[
 \norm[W^{k,2}(Q)]{\xi\circ\eta}\leq c\norm[W^{k,2}(Q)]{\eta}\norm[C^k(Q)]{\xi},
 \]
 where the constant depends on $\Omega,n,k$ and $\norm[W^{2,q}(Q)]{\eta}$ only. 
\end{proposition}

\begin{proof}
 Linearity follows immediately from the definition. Now we calculate
 \begin{align*}
  \norm[L^q(Q)]{\nabla^2 (\xi \circ \eta)} &= \norm[L^q(Q)]{((\nabla^2 \xi) \circ \eta .\nabla \eta).\nabla \eta + (\nabla \xi) \circ \eta . \nabla^2 \eta} \\
  &\leq C \left(\norm[\infty]{\nabla \eta}^2 \norm[L^q(Q)]{(\nabla^2 \xi) \circ \eta} + \norm[\infty]{\nabla \xi} \norm[L^q(Q)]{\nabla^2 \eta} \right).
 \end{align*}
 and use
 \begin{align*}
  &\phantom{{}={}}\epsilon_0 \norm[L^q(Q)]{ (\nabla^2 \xi) \circ \eta}^q \leq \int_Q \abs{(\nabla^2 \xi) \circ \eta}^q \det \nabla \eta dx = \norm[L^q(\eta(Q))]{\nabla^2 \xi}^q
 \end{align*}
 to estimate the first term. Then using a Poincaré-inequality and the usual Morrey-embeddings we get
 \begin{align*}
  \norm[W^{2,q}(Q)]{\xi \circ \eta} \leq C \norm[W^{2,q}(\eta(Q)]{\xi} \frac{\norm[W^{2,q}(Q)]{\eta}^2}{\varepsilon^{1/q}}
 \end{align*}
 which proves that $\eta^\#$ is a vector space-homomorphism with given operator-norm. Now as $(\eta^\#)^{-1} = (\eta^{-1})^\#$ we conclude that it is also an isomorphism by the previous lemma.
 
 For the second estimate we observe that
 \[
 \norm{\nabla^k(\xi\circ\eta)}\leq c\sum_{l=1}^k\sum_{a\in \{1,...,n\}^l}\norm[C^l(\eta(Q)]{\xi}\norm{\nabla^{k-l} \Pi_{i=1}^{l}\partial_{a_i}\eta},
 \]
 which finishes the proof by \autoref{lem:horriblelemma}
\end{proof}

\subsection{Proof of \autoref{lem:approxTestFcts}}
\label{app:testfcts}
The proof is split in two parts. The first part constructs an extension of the solenoidality. The second part shows how this extension can than be convoluted. We also will need the following Poncar\'{e} type lemma:

\begin{lemma}[Poincar\'{e}-lemma for thin regions] \label{lem:thinPoincare}
 Let $S_0 \subset \R^n$ be an $(n-1)$-dimensional rectifiable set and $\Phi: S_0 \times [0,\eps_0] \to \R^n$ a injective $L$-bi-Lipschitz function such that $\Phi(.,0) = \operatorname{id}$.  Define $S_\eps = \Phi(S_0,[0,\eps])$ for $\eps \in [0,\eps_0]$. Then for all $f \in W^{1,a}(S_\eps)$ with $f|_{S_0} = 0$ in the trace sense we have
\begin{align}
 \norm[L^a(S_\eps)]{f} \leq c \eps \norm[W^{1,a}(S_\eps)]{f} \text{ for all $f \in W^{1,a}(S_\eps(t);\R^n)$ with $f|_{\Omega(t)} = 0$.}
\end{align}
where $c$ is independent of $\eps$.
\end{lemma}

\begin{proof}
 By the usual density arguments it is enough to prove the theorem for smooth functions. Now for $z\in S_0$ and $s_0\in [0,\eps_0]$ we find:
 \begin{align*}
  \abs{f(\Phi(z,s_0))} &= \abs{f(\Phi(z,s_0))-f(\Phi(z,0))} = \abs{\int_0^{s_0} \partial_s f(\Phi(z,s)) ds} \\
   &\leq \int_0^{s_0} \abs{(\nabla f)(\Phi(z,s))}\abs{\partial_s \Phi(z,s)} ds \leq \int_0^{s_0} L \abs{(\nabla f)(\Phi(z,s))} ds
 \end{align*}
 But then integrating over the whole domain gets us
 \begin{align*}
  \int_{S_\eps} \abs{f(y)}^a dy &= \int_S \int_0^\eps \abs{f(\Phi(z,s_0))}^a \abs{ J(z,s_0) } ds_0 dz \\
  &\leq \int_{S_0} \int_0^\eps \left(\int_0^{s_0} L \abs{\nabla f(\Phi(z,s))}ds \right)^a \abs{ J(z,s_0) } ds ds_0 dz \\
  &\leq \int_{S_0} \int_0^\eps \fint_0^{\eps} \eps^a L^a \abs{\nabla f(\Phi(z,s))}^a \abs{ J(z,s_0) } ds ds_0 dz \\
  &= \eps^a L^a \int_{S_0} \int_0^{\eps} L^a \abs{\nabla f(\Phi(z,s))}^a \abs{ J(z,s) } \norm[\infty]{J}\norm[\infty]{J^{-1}} ds dz \\
  &= L^a \eps^a  \norm[\infty]{J}\norm[\infty]{J^{-1}}\int_{S_\eps} \abs{\nabla f(y)}^a dy 
 \end{align*}
 where $J(z,s)$ is the Jacobian of $\Phi$ which is bounded from above and away from zero, as $\Phi$ is bi-Lipschitz.
\end{proof}
\begin{lemma}[Extension of the solenoidal region]
 \label{lem:diverapprox}
Fix a function 
\[
\eta \in L^\infty([0,T];\mathcal{E}) \cap W^{1,2}([0,T];W^{1,2}(Q;\R^n))%
\text{ with } \sup_{t\in T} E(\eta(t)) < \infty,\]
such that $\eta(t) \notin \partial \mathcal{E}$ for all $t \in [0,T]$.
As before we set $\Omega(t)=\Omega \setminus \eta(t,Q)$. 
 
Let $\xi \in L^2([0,T];W^{1,2}_0(\Omega;\R^n))$ such that $\diver \xi(t) = 0$ on $\Omega(t)$. Then there exists $\varepsilon_0 > 0$ such that for all $\varepsilon >0$, there exists $\xi_\varepsilon$ such that $\diver \xi(t,y) = 0$ for all $y\in \Omega$ with $\dist(y,\Omega(t)\cup \partial \Omega) < \varepsilon$ and there are constants $c$ independent of $\xi$ such that for a.e. $t\in [0,T]$
\begin{align*}
 \norm[W^{1,2}(\Omega)]{\xi_\varepsilon} &\leq c\norm[W^{1,2}(\Omega)]{\xi}  &\text{ and } \qquad
  \norm[L^2(\Omega)]{\xi_\eps-\xi}
  &\leq c\eps^\frac{2}{n+2}\norm[W^{1,2}(\Omega)]{\xi}.
\end{align*}
Additionally for any $k \in \N$ and $a\in (1,\infty)$ such that $\xi \in L^2([0,T];W^{k,a}(Q;\R^n))$ we also have
\begin{align*}
  \norm[{L^2([0,T];W^{k,a}(\Omega))}]{\xi_\eps-\xi}&\to 0\text{ for }\eps\to 0
\end{align*}
and similarly if $\xi \in W^{1,2}([0,T];W^{1,\infty}(\Omega;\R^n))$ then also
\begin{align}
\label{eq:epstime}
\norm[{L^2([0,T];W^{1,2}(\Omega))}]{\partial_t (\xi-\xi_\eps)}\to 0\text{ for }\eps\to 0.
\end{align}
\end{lemma}

\begin{proof}
 
We begin by defining
\[
S_\eps(t):=\{y\in \Omega\, |\, \dist(y,\eta(t,\partial Q))\leq \eps\}
\]
and introduce the usual cutoff function
$\psi_\eps:[0,T]\times \Omega\to [0,1]$, such that
\[
\chi_{S_\eps(t)}\leq \psi_\eps(t)\leq \chi_{S_{2\eps}(t)}\text{ and }\norm[C^{l}]{\psi_\eps(t)}\leq \frac{c}{\eps^l}\text{ for }l\in \N.
\] 
Due to the regularity of $\eta$, we may assume, that $\partial_t\psi_\eps$ is uniformly bounded, such that
\begin{align}
\label{eq:psitime}
\norm[{L^2([0,T]\times \Omega)}]{\partial_t\psi_\eps}\to 0\text{ with }\eps\to 0.
\end{align}

We also pick $\tilde{\psi} \in C_0^\infty([0,T];\Omega;\R^n)$ such that $\supp \tilde{\psi}(t) \cap S_{\eps_0}(t) = 0$ for some $\eps_0>0$ and $\int_\Omega \tilde{\psi}(t) dy = 1$ for all $t$. Using this we then define
\begin{align*}
 \xi_\varepsilon(t) := \xi(t) - \bog\left( \psi_\eps(t) \diver \xi(t) - b_\eps(t)  \tilde{\psi}(t)\right)
\end{align*}
where $\bog$ is the Bogovski\u{\i}-operator on $\Omega$ and $b_\eps(t) := \int_\Omega\psi_\eps(t) \diver \xi(t)\, dy$ is used to keep the mean. Then per definition
\begin{align*}
 \diver \xi_\eps(t) = (1-\psi_\eps(t)) \diver \xi(t) - b_\eps(t) \tilde{\psi}(t)
\end{align*}
has no support on $S_\eps(t)$, as required and 
\begin{align*}
 &\phantom{{}={}}\norm[W^{k,a}(\Omega)]{\xi-\xi_\eps} = \norm[W^{k,a}(\Omega)]{\bog\left( \psi_\eps(t) \diver \xi(t) - b_\eps(t)  \tilde{\psi}(t)\right)} \\
 &\leq c\norm[W^{k-1,a}(\Omega)]{ \psi_\eps(t) \diver \xi(t) - b_\eps(t)  \tilde{\psi}(t)} \leq c\norm[W^{k-1,a}(\Omega)]{ \psi_\eps(t) \diver \xi(t)} + c\abs{b_\eps(t)}
\end{align*}
is the main quantity we need to estimate.

Let us begin with the special case $k=0,a=2$. Here we use that $L^\frac{2n}{2+n}(\Omega;\R^n)\subset W^{-1,2}(\Omega;\R^n)$ and apply H\"older's inequality to show that
\begin{align*}
&\phantom{{}={}}\norm[L^2(\Omega)]{\xi-\tilde{\xi}_\eps}\leq c\norm[W^{-1,2}(\Omega)]{\psi_\eps\diver(\xi)}+c\abs{b_\eps} \leq c\norm[{L^\frac{2n}{n+2}(\Omega)}]{\psi_\eps\diver(\xi)}
\\
& \leq c \norm[L^{n}(S_{2\eps}(t))]{\psi_\eps} \norm[L^2(\Omega)]{\diver \xi}
\leq c\abs{S_{2\eps}}^\frac{1}{n}\norm[W^{1,2}(\Omega)]{\xi}\leq c\eps^\frac{1}{n}\norm[W^{1,2}(\Omega)]{\xi}.
\end{align*}

For $k \geq 1$ we first note that $\abs{b_\eps(t)} \leq c \norm[L^2(S_\eps(t))]{\diver \xi} \to 0$ for each fixed $\xi$ and that furthermore
\begin{align*}
 \norm[W^{k-1,a}(\Omega)]{ \psi_\eps(t) \diver \xi(t)} &\leq c\sum_{l=0}^{k-1} \norm[C^{k-1-l}(\Omega)]{\psi_\eps(t)} \norm[L^a(S_{2\eps}(t))]{\nabla^l \diver \xi} \\
 &\leq c\sum_{l=0}^{k-1} \eps^{-(k-1-l)} \norm[L^a(S_{2\eps}(t))]{\nabla^l \diver \xi}.
\end{align*}
In particular for $k=1,a=2$ we have $k-1-l=0$ so this immediately proves $\norm[W^{1,2}(\Omega)]{\xi_\eps} \leq c \norm[W^{1,2}(\Omega)]{\xi}$ independently of $\xi$. For $k>1$ we will apply the Poincar\'{e} inequality \autoref{lem:thinPoincare}. For this we make use of the fact that $S_{\eps_0} \setminus \Omega(t)$ is a small neighborhood of an uniform Lipschitz boundary and thus can be written in the required way using $\eta$ itself. Furthermore for any $l<k$ we have $\nabla^l \diver \xi = 0$ on $\Omega(t)$ and thus also on $\partial \Omega(t)$ in the trace sense. Now this then gives us $\norm[L^a(S_{2\eps}(t))]{\nabla^l \diver \xi} \leq c \eps^{k-1-l} \norm[W^{k,a}(S_{2\eps}(t))]{\xi}$ which is enough to finish the estimate.

Finally let us consider the time-derivative. As $\bog$ is a linear operator, we have
\begin{align*}
 &\phantom{{}={}}\norm[W^{1,2}(\Omega)]{\partial_t (\xi_\eps-\xi)} 
 = c\norm[L^2(\Omega)]{\partial_t \left(\psi_\eps(t) \diver \xi(t) - b_\eps(t)  \tilde{\psi}(t)\right)} \\ &\leq c \norm[L^2(S_{2\eps}(t))]{\partial_t (\psi_\eps(t) \diver \xi(t))} + c\abs{\partial_t b_\eps(t)} \norm[L^2(\Omega)]{\tilde{\psi}(t)} +c\abs{b_\eps(t)} \norm[L^2(\Omega)]{\partial_t \tilde{\psi}(t)} 
\end{align*}
For the last term we have already shown that $\abs{b_\eps(t)} \to 0$ and $\tilde{\Psi}$ does not depend on $\eps$. For the second to last term we note that
\begin{align*}
 \abs{\partial_t  b_\eps(t)} = \abs{ \int_\Omega \partial_t(\psi_\eps(t) \diver \xi(t)) dy} \leq \norm[L^2(\Omega)]{\partial_t(\psi_\eps(t) \diver \xi(t))}
\end{align*}
which is the same as the first term and for which we use the estimate
\begin{align*}
 \norm[L^2(\Omega)]{\partial_t(\psi_\eps(t) \diver \xi(t))} \leq \norm[L^2(\Omega)]{\partial_t \psi_\eps(t)}\norm[W^{1,\infty}(S_{2\eps}(t))]{\xi(t)} + \norm[L^2(\Omega)]{\psi_\eps(t)}\norm[W^{1,\infty}(S_{2\eps}(t))]{\partial_t \xi(t)}
\end{align*}
which implies \eqref{eq:epstime} by \eqref{eq:psitime} and H\"older's inequality.

\end{proof}

\begin{proof}[Proof of \autoref{lem:approxTestFcts}]

First we apply \autoref{lem:diverapprox} to find a function $\hat{\xi}$ with $\hat{\xi}=0$ on $\Omega(t)$ and an $\eps$-neighborhood of $\partial (\Omega \setminus \Omega(t))$. Thus taking a convolution with $\gamma_{\eps^2}$ does not intervene with the zero boundary values (if $\eps$ is small enough). 

We will now apply \autoref{lem:diverapprox} again to $\hat{\xi}_\eps*\gamma_{\eps^2}$ and call the result $\xi_\eps$, a function
which is smooth by Theorem~\ref{thm:sohr}. Moreover since all operations are linear we find that $\xi_\eps\in C^\infty_0(\Omega)$ is divergence free in $\Omega(t)\cup S_\eps$. And by collecting all the properties of the approximation, we find that
\begin{align*}
&\norm[W^{l,a}(\Omega)]{\xi-{\xi}_\eps}\to 0
\end{align*}
for $l\leq k-1$. Moreover,
\begin{align*}
&\norm[W^{1,2}(\Omega)]{\xi-{\xi}_\eps}
\leq c\norm[W^{1,2}(\Omega)]{\xi}.
\end{align*}
and
\begin{align*}
&\norm[L^2(\Omega)]{\xi-\xi_\eps}\leq c\eps^\frac{2}{n+2}\norm[W^{1,2}(\Omega)]{\xi}
\end{align*}
Next we turn to the estimates for
\[
\phi_\eps:=\xi_\eps\circ \eta.
\]
They follow by \autoref{prop:isos}, and standard convolution estimates. First, for $k>2$ we find
\[
\norm[W^{k,2}(Q)]{\phi_\eps}\leq c\norm[C^k(\Omega)]{\xi_\eps}\norm[W^{k,2}(Q)]{\eta}\leq c(\eps)\norm[L^2(\Omega)]{\xi}\norm[W^{k,a}(Q))]{\eta}.
\]
Second, in case $\xi\in L^\infty([0,T];W^{2,a}(\Omega))$, we find
\[
\norm[W^{2,a}(Q)]{\phi_\eps-\phi}\leq c\norm[W^{2,a}(\Omega)]{\xi_\eps-\xi}\to 0 \text{ with } \eps\to 0.
\]
Finally, for the time derivative in case $\partial_t \xi\in L^\infty([0,T];W^{1,2}(\Omega))$ and $\xi\in L^\infty([0,T];W^{3,a}(\Omega))$ with $a>n$ we find by Sobolev embedding that:
\begin{align*}
\norm[W^{1,2}(Q)]{\partial_t(\phi_\eps-\phi)}%
&
\leq c\norm[W^{1,2}(\Omega)]{\partial_t(\xi_\eps-\xi)}+c\norm[W^{1,\infty}(\Omega)]{\nabla(\xi_\eps-\xi)}\norm[W^{1,2}(Q)]{\partial_t\eta}
\\
&\leq c\norm[W^{1,2}(\Omega)]{\partial_t(\xi_\eps-\xi)}+c\norm[W^{3,a}(\Omega)]{\xi_\eps-\xi}\norm[W^{1,2}(Q)]{\partial_t\eta}
\end{align*}
which implies the assertions for the time-derivatives by \eqref{eq:epstime}.

\end{proof}

\subsection*{Acknowledgments}

\noindent
All authors thank the support of the Primus research programme PRI\-MUS/19/SCI/01 and the University Centre UNCE/SCI/023 of Charles University. Moreover they thank for the support of the program GJ19-11707Y of the Czech national grant agency (GA\v{C}R).

 \bibliographystyle{alpha}
 \bibliography{biblio}

\begin{thebibliography}{DEGLT01b}

\bibitem[AGS05]{ambrosioGradientFlows2005}
L.~Ambrosio, N.~Gigli, and G.~Savare.
\newblock {\em Gradient Flows in Metric Spaces and in the Space of Probability
  Measures}.
\newblock Lectures in Mathematics. ETH Z{\"u}rich. Birkh{\"a}user Basel, 2005.

\bibitem[Ant98]{antmanPhysicallyUnacceptableViscous1998}
Stuart~S. Antman.
\newblock Physically unacceptable viscous stresses.
\newblock {\em Zeitschrift f\"ur angewandte Mathematik und Physik},
  49(6):980--988, 1998.

\bibitem[Bal02]{ball2002some}
John~M Ball.
\newblock Some open problems in elasticity.
\newblock In {\em Geometry, mechanics, and dynamics}, pages 3--59. Springer,
  2002.

\bibitem[BS90]{borchersEquationsRotDiv1990}
Wolfgang Borchers and Hermann Sohr.
\newblock On the equations $\operatorname{rot} v=g$ and $\operatorname{div} u=
  f$ with zero boundary conditions.
\newblock {\em Hokkaido Mathematical Journal}, 19(1):67--87, 1990.

\bibitem[BS18]{breitCompressibleFluidsInteracting2018}
Dominic Breit and Sebastian Schwarzacher.
\newblock Compressible fluids interacting with a linear-elastic shell.
\newblock {\em Archive for Rational Mechanics and Analysis}, 228(2):495--562,
  May 2018.

\bibitem[CN87]{ciarletInjectivitySelfcontactNonlinear1987}
Philippe~G. Ciarlet and Jind{\v r}ich Ne{\v c}as.
\newblock Injectivity and self-contact in nonlinear elasticity.
\newblock {\em Archive for Rational Mechanics and Analysis}, 97(3):171--188,
  1987.

\bibitem[CN17]{CheSar17}
Nikolai~V. Chemetov and {\v{S}}{\'a}rka Ne{\v{c}}asov{\'a}.
\newblock The motion of the rigid body in the viscous fluid including
  collisions. {G}lobal solvability result.
\newblock {\em Nonlinear Anal. Real World Appl.}, 34:416--445, 2017.

\bibitem[CNM19]{chemetovWeakstrongUniquenessFluidrigid2019}
Nikolai~V. Chemetov, {\v S}{\'a}rka Ne{\v c}asov{\'a}, and Boris Muha.
\newblock Weak-strong uniqueness for fluid-rigid body interaction problem with
  slip boundary condition.
\newblock {\em Journal of Mathematical Physics}, 60(1):011505, 2019.

\bibitem[DE99]{desjardinsExistenceWeakSolutions1999}
Benoit Desjardins and Maria~J. Esteban.
\newblock Existence of weak solutions for the motion of rigid bodies in a
  viscous fluid.
\newblock {\em Archive for rational mechanics and analysis}, 146(1):59--71,
  1999.

\bibitem[DEGLT01a]{DesGra01}
B.~Desjardins, M.~J. Esteban, C.~Grandmont, and P.~Le~Tallec.
\newblock Weak solutions for a fluid-elastic structure interaction model.
\newblock {\em Revista Matem\'{a}tica Complutense}, 14(2):523--538, 2001.

\bibitem[DEGLT01b]{desjardinsWeakSolutionsFluidelastic2001}
Benoit Desjardins, Mar{\'i}a~J. Esteban, C{\'e}line Grandmont, and Patrick
  Le~Tallec.
\newblock Weak solutions for a fluid-elastic structure interaction model.
\newblock {\em Revista Matem\'atica Complutense}, 14(2):523--538, 2001.

\bibitem[Dem00]{demouliniWeakSolutionsClass2000}
Sophia Demoulini.
\newblock Weak solutions for a class of nonlinear systems of viscoelasticity.
\newblock {\em Archive for Rational Mechanics and Analysis}, 155(4):299--334,
  2000.

\bibitem[DG93]{degiorgiNewProblemsMinimizing1993}
Ennio De~Giorgi.
\newblock New problems on minimizing movements.
\newblock {\em Ennio de Giorgi: Selected Papers}, pages 699--713, 1993.

\bibitem[DSET01]{demouliniVariationalApproximationScheme2001}
Sophia Demoulini, David M.~A. Stuart, and Athanasios E.~Tzavaras.
\newblock A {{Variational Approximation Scheme}} for {{Three}}-{{Dimensional
  Elastodynamics}} with {{Polyconvex Energy}}.
\newblock {\em Archive for Rational Mechanics and Analysis}, 157(4):325--344,
  May 2001.

\bibitem[Fei03]{Fei03}
Eduard Feireisl.
\newblock On the motion of rigid bodies in a viscous compressible fluid.
\newblock {\em Archive for Rational Mechanics and Analysis}, 167(4):281--308,
  2003.

\bibitem[FG06]{friedTractionsBalancesBoundary2006}
Eliot Fried and Morton~E. Gurtin.
\newblock Tractions, balances, and boundary conditions for nonsimple materials
  with application to liquid flow at small-length scales.
\newblock {\em Archive for Rational Mechanics and Analysis}, 182(3):513--554,
  2006.

\bibitem[FJM06]{frieseckeHierarchyPlateModels2006}
Gero Friesecke, Richard~D. James, and Stefan M{\"u}ller.
\newblock A hierarchy of plate models derived from nonlinear elasticity by
  gamma-convergence.
\newblock {\em Archive for rational mechanics and analysis}, 180(2):183--236,
  2006.

\bibitem[Gal13]{Gal13}
Giovanni~P. Galdi.
\newblock On time-periodic flow of a viscous liquid past a moving cylinder.
\newblock {\em Archive for Rational Mechanics and Analysis}, 210(2):451--498,
  2013.

\bibitem[Gal16]{Gal16bf}
Giovanni~P. Galdi.
\newblock On bifurcating time-periodic flow of a {N}avier-{S}tokes liquid past
  a cylinder.
\newblock {\em Archive for Rational Mechanics and Analysis}, 222(1):285--315,
  2016.

\bibitem[GGH13]{GGH13}
Matthias Geissert, Karoline G\"otze, and Matthias Hieber.
\newblock {$L^p$}-theory for strong solutions to fluid-rigid body interaction
  in {N}ewtonian and generalized {N}ewtonian fluids.
\newblock {\em Transactions of the American Mathematical Society},
  365(3):1393--1439, 2013.

\bibitem[GH16]{grandmontExistenceGlobalStrong2016}
Céline Grandmont and Matthieu Hillairet.
\newblock Existence of global strong solutions to a beam–fluid interaction
  system.
\newblock {\em Archive for Rational Mechanics and Analysis}, 220(3):1283--1333,
  2016.

\bibitem[GK09]{galdiSteadyFlowNavier2009}
Giovanni~P. Galdi and Mads Kyed.
\newblock Steady flow of a {{Navier}}–{{Stokes}} liquid past an elastic body.
\newblock {\em Archive for rational mechanics and analysis}, 194(3):849--875,
  2009.

\bibitem[GM12]{gigliVariationalApproachNavier2012}
Nicola Gigli and Sunra~JN Mosconi.
\newblock A variational approach to the {{Navier}}–{{Stokes}} equations.
\newblock {\em Bulletin des Sciences Mathématiques}, 136(3):256--276, 2012.

\bibitem[Gra02]{grandmontExistenceThreeDimensionalSteady2002}
C.~Grandmont.
\newblock Existence for a {{Three}}-{{Dimensional Steady State
  Fluid}}-{{Structure Interaction Problem}}.
\newblock {\em Journal of Mathematical Fluid Mechanics}, 4(1):76--94, February
  2002.

\bibitem[GRRT08]{GalRan08}
Giovanni~P. Galdi, Rolf Rannacher, Anne~M. Robertson, and Stefan Turek.
\newblock {\em Hemodynamical flows}, volume~37 of {\em Oberwolfach Seminars}.
\newblock Birkh\"{a}user Verlag, Basel, 2008.
\newblock Modeling, analysis and simulation, Lectures from the seminar held in
  Oberwolfach, November 20--26, 2005.

\bibitem[GS07]{GalSil07}
Giovanni~P. Galdi and Ana~L. Silvestre.
\newblock The steady motion of a {N}avier-{S}tokes liquid around a rigid body.
\newblock {\em Archive for Rational Mechanics and Analysis}, 184(3):371--400,
  2007.

\bibitem[GS09]{GalSil09}
Giovanni~P. Galdi and Ana~L. Silvestre.
\newblock On the motion of a rigid body in a {N}avier-{S}tokes liquid under the
  action of a time-periodic force.
\newblock {\em Indiana University Mathematics Journal}, 58(6):2805--2842, 2009.

\bibitem[Hes04]{hesla2004collisions}
Todd~Inman Hesla.
\newblock {\em Collisions of smooth bodies in viscous fluids: A mathematical
  investigation}.
\newblock PhD thesis, University of Minnesota, 2004.

\bibitem[Hil07]{hillairetLackCollisionSolid2007}
Matthieu Hillairet.
\newblock Lack of collision between solid bodies in a {{2D}} incompressible
  viscous flow.
\newblock {\em Communications in Partial Differential Equations},
  32(9):1345--1371, 2007.

\bibitem[HK09]{healeyInjectiveWeakSolutions2009}
Timothy~J. Healey and Stefan Krömer.
\newblock Injective weak solutions in second-gradient nonlinear elasticity.
\newblock {\em ESAIM: Control, Optimisation and Calculus of Variations},
  15(4):863--871, 2009.

\bibitem[HSN75]{halphenMateriauxStandardGeneralises1975}
Bernard Halphen and Quoc Son~Nguyen.
\newblock Sur les mat\'eriaux standard g\'en\'eralis\'es.
\newblock {\em Journal de M\'ecanique}, 14:39--63, 1975.

\bibitem[KR19a]{kromerQuasistaticViscoelasticitySelfcontact2019}
Stefan Kr{\"o}mer and Tom{\'a}{\v s} Roubi{\v c}ek.
\newblock Quasistatic viscoelasticity with self-contact at large strains.
\newblock {\em Journal of Elasticity}, (accepted), April 2019.

\bibitem[KR19b]{kruzik2019mathematical}
Martin Kru{\v{z}}{\'\i}k and Tom{\'a}{\v{s}} Roub{\'\i}{\v{c}}ek.
\newblock {\em Mathematical methods in continuum mechanics of solids}.
\newblock Springer, 2019.

\bibitem[LDR95]{ledretQuasiconvexEnvelopeSaint1995}
Herv{\'e} Le~Dret and Annie Raoult.
\newblock The quasiconvex envelope of the {{Saint
  Venant}}\textendash{{Kirchhoff}} stored energy function.
\newblock {\em Proceedings of the Royal Society of Edinburgh Section A:
  Mathematics}, 125(6):1179--1192, 1995.

\bibitem[Len14]{lengelerWeakSolutionsIncompressible2014a}
Daniel Lengeler.
\newblock Weak solutions for an incompressible, generalized {{Newtonian}} fluid
  interacting with a linearly elastic {{Koiter}} type shell.
\newblock {\em SIAM Journal on Mathematical Analysis}, 46(4):2614--2649, 2014.

\bibitem[Len15]{lengelerStokestypeSystemArising2015}
Daniel Lengeler.
\newblock On a {{Stokes}}-type system arising in fluid vesicle dynamics.
\newblock {\em arXiv:1506.08991 [math]}, December 2015.

\bibitem[LR14]{lengelerWeakSolutionsIncompressible2014}
Daniel Lengeler and Michael R{\u u}{\v z}i{\v c}ka.
\newblock Weak solutions for an incompressible {{Newtonian}} fluid interacting
  with a {{Koiter}} type shell.
\newblock {\em Archive for Rational Mechanics and Analysis}, 211(1):205--255,
  2014.

\bibitem[M{\v C}13a]{muhaExistenceWeakSolution2013}
Boris Muha and Sun{\v c}ica {\v C}ani{\'c}.
\newblock Existence of a weak solution to a nonlinear fluid\textendash
  structure interaction problem modeling the flow of an incompressible, viscous
  fluid in a cylinder with deformable walls.
\newblock {\em Archive for rational mechanics and analysis}, 207(3):919--968,
  2013.

\bibitem[M{\v C}13b]{muhaNonlinear3DFluidstructure2013}
Boris Muha and Sun{\v c}ica {\v C}ani{\'c}.
\newblock A nonlinear, {{3D}} fluid-structure interaction problem driven by the
  time-dependent dynamic pressure data: A constructive existence proof.
\newblock {\em Communications in Information and Systems}, 13(3):357--397,
  2013.

\bibitem[M{\v C}15]{muhaFluidstructureInteractionIncompressible2015}
Boris Muha and Sun{\v c}ica {\v C}ani{\'c}.
\newblock Fluid-structure interaction between an incompressible, viscous {{3D}}
  fluid and an elastic shell with nonlinear {{Koiter}} membrane energy.
\newblock {\em Interfaces and free boundaries}, 17(4):465, 2015.

\bibitem[M{\v C}16]{muhaExistenceWeakSolution2016}
Boris Muha and Sun{\v c}ica {\v C}ani{\'c}.
\newblock Existence of a weak solution to a fluid\textendash elastic structure
  interaction problem with the {{Navier}} slip boundary condition.
\newblock {\em Journal of Differential Equations}, 260(12):8550--8589, 2016.

\bibitem[MR15]{mielkeRateIndependentSystems2015}
Alexander Mielke and Tom{\'a}{\v{s}} Roub{\'\i}{\v{c}}ek.
\newblock Rate-independent systems.
\newblock {\em Theory and Application (in preparation)}, 2015.

\bibitem[MR19]{mielkeThermoviscoelasticityKelvinVoigtRheology2019}
Alexander Mielke and Tomáš Roubíček.
\newblock Thermoviscoelasticity in {{Kelvin}}-{{Voigt}} rheology at large
  strains.
\newblock {\em Archive for rational mechanics and analysis}, (accepted), 2019.

\bibitem[MS19]{muhaExistenceRegularityWeak2019}
Boris Muha and Sebastian Schwarzacher.
\newblock Existence and regularity for weak solutions for a fluid interacting
  with a non-linear shell in {{3D}}.
\newblock {\em arXiv:1906.01962 [math]}, June 2019.

\bibitem[MT12]{miroshnikovVariationalApproximationScheme2012}
Alexey Miroshnikov and Athanasios~E. Tzavaras.
\newblock A variational approximation scheme for radial polyconvex elasticity
  that preserves the positivity of {{Jacobians}}.
\newblock {\em Communications in Mathematical Sciences}, 10(1):87--115, March
  2012.

\bibitem[Nef02]{neffKornFirstInequality2002}
Patrizio Neff.
\newblock On {{Korn}}'s first inequality with non-constant coefficients.
\newblock {\em Proceedings of the Royal Society of Edinburgh Section A:
  Mathematics}, 132(1):221--243, February 2002.

\bibitem[Ngu00]{nguyenStabilityNonlinearSolidMechanics2000}
Quoc~Son Nguyen.
\newblock {\em Stability and Nonlinear Solid Mechanics}.
\newblock Wiley, 2000.

\bibitem[PH17]{palmerInjectivitySelfcontactSecondgradient2017}
Aaron~Z. Palmer and Timothy~J. Healey.
\newblock Injectivity and self-contact in second-gradient nonlinear elasticity.
\newblock {\em Calculus of Variations and Partial Differential Equations},
  56(4):114, July 2017.

\bibitem[Pir82]{Pir82}
O.~Pironneau.
\newblock On the transport-diffusion algorithm and its applications to the
  {N}avier-{S}tokes equations.
\newblock {\em Numer. Math.}, 38(3):309--332, 1981/82.

\bibitem[Pom03]{pompeKornFirstInequality2003}
Waldemar Pompe.
\newblock Korn’s first inequality with variable coefficients and its
  generalizations.
\newblock {\em Comment. Math. Univ. Carolinae}, 44(1):57--70, 2003.

\bibitem[QTV00]{quarteroni2000computational}
Alfio Quarteroni, Massimiliano Tuveri, and Alessandro Veneziani.
\newblock Computational vascular fluid dynamics: problems, models and methods.
\newblock {\em Computing and Visualization in Science}, 2(4):163--197, 2000.

\bibitem[Ric17]{RichterFluidStructureInteractions17}
Thomas Richter.
\newblock {\em Fluid-structure interactions: models, analysis and finite
  elements}, volume 118.
\newblock Springer, 2017.

\bibitem[{\v S}il85]{silhavyPhaseTransitionsNonsimple1985}
M.~{\v S}ilhav{\`y}.
\newblock Phase transitions in non-simple bodies.
\newblock {\em Archive for rational mechanics and analysis}, 88(2):135--161,
  1985.

\bibitem[Tak03]{Takahashi03}
Tak\'eo Takahashi.
\newblock Analysis of strong solutions for the equations modeling the motion of
  a rigid-fluid system in a bounded domain.
\newblock {\em Adv. Differential Equations}, 8(12):1499--1532, 2003.

\bibitem[Tou62]{toupinElasticMaterialsCouplestresses1962}
R.~Toupin.
\newblock Elastic materials with couple-stresses.
\newblock {\em Archive for Rational Mechanics and Analysis}, 11(1):385--414,
  1962.

\bibitem[Tri02]{triebelFunctionSpacesLipschitz2002}
Hans Triebel.
\newblock Function spaces in {{Lipschitz}} domains and on {{Lipschitz}}
  manifolds. {{Characteristic}} functions as pointwise multipliers.
\newblock {\em Revista Matem\'atica Complutense}, 15(2):475--524, 2002.

\bibitem[TTW15]{TakTucWei15}
Tak\'{e}o Takahashi, Marius Tucsnak, and George Weiss.
\newblock Stabilization of a fluid-rigid body system.
\newblock {\em J. Differential Equations}, 259(11):6459--6493, 2015.

\bibitem[Zie89]{Ziemer89book}
W.~P. Ziemer.
\newblock {\em Weakly differentiable functions}, volume 120 of {\em Graduate
  Texts in Mathematics}.
\newblock Springer-Verlag, New York, 1989.
\newblock Sobolev spaces and functions of bounded variation.

\end{thebibliography}
\end{document}